\documentclass[reqno]{article}
\usepackage{amsmath,amsfonts,amssymb,amsthm}
\usepackage{bbm}
\usepackage[final]{graphicx}

\usepackage{geometry}

\newtheorem{theorem}{Theorem}
\newtheorem{lemma}[theorem]{Lemma}
\newtheorem{corollary}[theorem]{Corollary}
\newtheorem{remark}[theorem]{Remark}
\newtheorem{claim}[theorem]{Claim}
\theoremstyle{definition}
\newtheorem{example}[theorem]{Example}
\newtheorem*{definition}{Definition}
\newtheorem*{ack}{Acknowledgement}

\newcommand{\cA}{{\mathcal A}}
\newcommand{\cC}{{\mathcal C}}
\newcommand{\cD}{{\mathcal D}}
\newcommand{\cE}{{\mathcal E}}
\newcommand{\cF}{{\mathcal F}}
\newcommand{\cG}{{\mathcal G}}
\newcommand{\cH}{{\mathcal H}}
\newcommand{\cI}{{\mathcal I}}
\newcommand{\cJ}{{\mathcal J}}
\newcommand{\cK}{{\mathcal K}}
\newcommand{\cM}{{\mathcal M}}
\newcommand{\cS}{{\mathcal S}}
\newcommand{\fY}{{\mathfrak Y}}

\newcommand{\hI}{\cE}
\newcommand{\hH}{\cF}

\newcommand{\uzero}{{\underline{0}}}

\newcommand\Bin{\operatorname{Bin}}

\newcommand\E{{\mathbb E}}
\newcommand\Var{\operatorname{Var}} 
\renewcommand\Pr{{\mathbb P}}

\newcommand\NN{\mathbb{N}}

\newcommand{\indic}[1]{\mathbbm{1}_{\{{#1}\}}}
\newcommand{\indicb}[1]{\mathbbm{1}_{\bigl\{{#1}\bigr\}}}

\newcommand\ceil[1]{\lceil {#1} \rceil}
\newcommand\ceilL[1]{\left\lceil {#1} \right\rceil}

\newcommand\restr[2]{\ensuremath{\left.#1\right|_{#2}}}

\renewcommand{\epsilon}{\varepsilon}
\newcommand{\eps}{\varepsilon}

\long\def\symbolfootnote[#1]#2{\begingroup
\def\thefootnote{\fnsymbol{footnote}}\footnote[#1]{#2}\endgroup}

\let\OLDthebibliography\thebibliography
\renewcommand\thebibliography[1]{
  \OLDthebibliography{#1}
  \setlength{\parskip}{0pt}
  \setlength{\itemsep}{0pt plus 0.3ex}
}

\begin{document}
\title{On the missing log in upper tail estimates} 
\author{Lutz Warnke\thanks{School of Mathematics, Georgia Institute of Technology, 
Atlanta GA~30332, USA. E-mail: {\tt warnke@math.gatech.edu}.
Research partially supported by NSF Grant DMS-1703516 and a Sloan Research Fellowship. 
Part of the work was done while the author was a member of the Department of Pure Mathematics and Mathematical Statistics, University of Cambridge.}}
\date{29 February, 2016; revised January 11, 2019} 
\maketitle

\begin{abstract} 
In the late 1990s, Kim and Vu pioneered an inductive method for showing concentration of certain random variables~$X$.  
Shortly afterwards, Janson and Ruci{\'n}ski developed an alternative inductive approach, 
which often gives comparable results for the upper tail $\Pr(X \ge (1+\eps)\E X)$. 
In some cases, both methods yield upper tail estimates which are best possible up to a logarithmic factor in the exponent,  
but closing this narrow gap has remained a technical challenge. 
In this paper we present a BK-inequality based combinatorial sparsification idea that can recover this missing logarithmic term in the upper tail. 

As an illustration, we consider random subsets of the integers $\{1, \ldots, n\}$, and prove 
sharp upper tail estimates for various objects of interest in additive combinatorics. 
Examples include the number of arithmetic progressions, Schur triples, additive quadruples, and $(r,s)$-sums. 
\end{abstract}

\section{Introduction}\label{sec:intro}
Concentration inequalities are of great importance in discrete mathematics, theoretical computer science, and related fields. 
They intuitively quantify random fluctuations of a given random variable~$X$, by bounding the probability that~$X$ differs substantially from its expected value~$\mu = \E X$. 
In combinatorial applications, $X$~often counts certain objects (e.g., the number of subgraphs or arithmetic progressions), in which case the random variable~$X$ can usually be written as a low-degree polynomial of many independent random variables. 
In this context concentration inequalities with exponentially small estimates are vital 
(e.g., to make union bound arguments amenable), 
and here Kim and Vu~\cite{KimVu2000,Vu2000,Vu2002} achieved a breakthrough in the late~1990s. 
Their powerful concentration inequalities have since then, e.g., 
been successfully applied to many combinatorial problems, been included in standard textbooks, and 
earned Vu the George~P{\'o}lya~Prize in~2008.

In probabilistic combinatorics, the exponential rate of decay of the \emph{lower tail} $\Pr(X \le \mu -t)$ and \emph{upper tail} $\Pr(X \ge \mu + t)$ have received considerable attention, since they are of great importance in applications (of course, this is also an interesting problem in concentration of measure). 
The behaviour of the lower tail is nowadays well-understood due to the celebrated Janson- and Suen-inequalities~\cite{Janson,RW,JW,JSuen,DLP}. 
By contrast, the behaviour of the `infamous' upper tail has remained a well-known technical challenge (see also~\cite{UT,UTSG}). 
Here the inductive method of Kim and Vu~\cite{KimVu2000,Vu2002} from around~1998 often yields inequalities of the form 
\begin{equation}\label{eq:intro:UT}
\Pr(X \ge (1+\eps)\mu) \le \exp\bigl(- c(\eps) \mu^{1/q}\bigr) ,
\end{equation}
where $q \ge 1$ is some constant. In~2000, Janson and Ruci{\'n}ski~\cite{DL} developed an alternative inductive approach, 
which often gives comparable results for the upper tail, i.e., which recovers~\eqref{eq:intro:UT} up to the usually irrelevant numerical value of the parameter~$c$. 
Studying the sharpness of the tail inequality~\eqref{eq:intro:UT} is an important problem according to Vu (see Section~4.8 in~\cite{Vu2002}). 
In fact, one main aim of the paper~\cite{DL} was `to stimulate more research into these methods' since `neither of [them] seems yet to be fully developed'.  
In other words, Janson and Ruci{\'n}ski were asking for further improvements of the aforementioned fundamental proof techniques (the papers~\cite{DL,Vu2002} already contained several tweaking options for decreasing~$q$).

\enlargethispage{\baselineskip}

In this paper we address this technical challenge in cases where the inductive methods of Kim--Vu and Janson--Ruci{\'n}ski are nearly sharp. 
The crux is that, for several interesting classes of examples (naturally arising, e.g., in additive combinatorics), the upper tail inequality~\eqref{eq:intro:UT} is best possible up to a logarithmic factor in the exponent. 
Closing such narrow gaps has recently become an active area of research in combinatorial probability (see, e.g,~\cite{UT,UTSG,UTAP,K3TailCh,KkTailDK,AP}). 
The goal of this paper is to present a new idea that can add such missing logarithmic terms to the upper tail. 
From a conceptual perspective, this paper thus makes 
a new effect amenable to the rich toolbox of the Kim--Vu and Janson--Ruci{\'n}ski methods 
(we believe that our techniques will be useful elsewhere). 
For example, under certain 
somewhat natural %(fairly natural) 
technical assumptions,  
our methods allow us to improve the classical upper tail inequality~\eqref{eq:intro:UT} to estimates of the form
\begin{equation}\label{int:bd:2}
\Pr(X \ge (1+\eps)\mu) \le \exp\Bigl(-c(\eps) \min\bigl\{\mu, \; \mu^{1/q}s\bigr\} \Bigr) \quad \text{ with } \quad s \in \bigl\{\log n,\: \log(1/p)\bigr\} ,
\end{equation}
where the reader may wish to tentatively think of the parameters $n = \omega(1)$ and $p = o(1)$ as those in the binomial random graph~$G_{n,p}$ 
(here some extra assumptions are necessary, since there are examples where~\eqref{eq:intro:UT} is sharp, see Sections~\ref{sec:results:flavour} and~\ref{sec:applications:RISH}). 
This seemingly small improvement of~\eqref{eq:intro:UT} is conceptually important, since in several interesting 
applications the resulting inequality is best possible up to the value of~$c$. 
Indeed, as we shall see, sharp examples with 
$\Pr(X \ge (1+\eps)\mu) = \exp\bigl(-\Theta(\min\bigl\{\mu, \mu^{1/q}\log(1/p))\bigr)$ for $\eps=\Theta(1)$ 
naturally arise when~$X$ counts various objects of great interest in additive combinatorics, 
such as the number of arithmetic progressions (of given length) or additive quadruples in random subsets of the integers~$[n]=\{1, \ldots, n\}$.

In the remainder of this introduction we illustrate our methods with some applications, 
outline our high-level proof strategy, and discuss the structure of this paper. 
Noteworthily, our proof techniques do not solely rely on induction, but 
a blend of combinatorial and probabilistic arguments.

\subsection{Flavour of the results}\label{sec:results:flavour} 
We now illustrate the main flavour of our upper tail results with some concrete examples.  
Many important counting problems can be rephrased as the number of edges induced by the random induced subhypergraph $\cH_p=\cH[V_p(\cH)]$ (see, e.g.,~\cite{UT,RR1994,UTAP,AP,WG2012}), where $V_p(\cH)$ denotes the binomial random subset where each vertex $v \in V(\cH)$ is included independently with probability~$p$. 
Our methods yield the following upper tail inequality for~$\cH_p$, which extends one of the main results from~\cite{AP} for the special case $q=2$, and sharpens one of the principle results of Janson and Ruci{\'n}ski~\cite{UTAP} by a logarithmic factor in the~exponent. 
%(instead of~\eqref{eq:intro:indshg} below they obtained~\eqref{eq:intro:UT}, which can also be derived from~\cite{DL,Vu2002}). 
%
\begin{theorem}[Counting edges of random induced subhypergraphs]% 
\label{thm:intro:indshg}
Let $1 \le q < k$ and $\gamma,D> 0$. 
Assume that $\cH$ is a $k$-uniform hypergraph with $v(\cH) \le N$ vertices and $e(\cH) \ge \gamma N^q$ edges.
Suppose that $\Delta_q(\cH) \le D$, where $\Delta_q(\cH)$ denotes the maximum number of edges of~$\cH$ that contain $q$ given vertices. 
Let $X = e(\cH_p)$ and $\mu = \E X$. 
Then for any~$\eps > 0$ there is~$c=c(\eps,k,\gamma,D)>0$ such that for all~$p \in (0,1]$ we have 
\begin{equation}\label{eq:intro:indshg}
\Pr(X \ge (1+\eps)\mu)  \le  
\exp\Bigl(-c \min\bigl\{\mu, \; \mu^{1/q} \log(e/p)\bigr\}\Bigr). 
\end{equation}
\end{theorem}
%
%\noindent 
This upper tail inequality is conceptually best possible in several ways.
First, the restriction to $q<k$ is necessary (see Section~\ref{sec:applications:RISH} for a counterexample when~$q=k$), 
Second, in several important applications~\eqref{eq:intro:indshg} is sharp (yields the correct exponential rate of decay), % (closing the logarithmic gap present in previous work), %
i.e., there is a matching lower bound of form 
\begin{equation}\label{eq:intro:indshg:both}
%\indic{1 \le (1+\eps)\mu \le e(\cH)} \exp\Bigl(-C(\eps) \min\{\mu, \; \mu^{1/q} \log(1/p)\}\Bigr) \le \Pr(X \ge (1+\eps)\mu)  \le  \exp\Bigl(-c(\eps) \min\{\mu, \; \mu^{1/q} \log(1/p)\}\Bigr) .
\Pr(X \ge (1+\eps)\mu)  \ge \indic{1 \le (1+\eps)\mu \le e(\cH)} \exp\Bigl(-C(\eps) \min\bigl\{\mu, \; \mu^{1/q} \log(e/p)\bigr\}\Bigr) ,
\end{equation}
where the restriction~$1 \le (1+\eps)\mu \le e(\cH)$ is natural.\footnote{Note that~$\Pr(X \ge (1+\eps)\mu)=0$ when~$(1+\eps)\mu > e(H)$, and that~$\Pr(X \ge (1+\eps)\mu)= 1 -\Pr(X=0)$ when~$(1+\eps)\mu < 1$.} 
In particular, letting the edges of the hypergraph~$\cH$ with vertex-set~$V(H)=[n]$ encode classical objects from additive combinatorics and Ramsey Theory,  
sharp examples of type~\eqref{eq:intro:indshg}--\eqref{eq:intro:indshg:both} include the number of $k$-term arithmetic progressions, Schur triples~$x+y=2z$, 
additive quadruples~$x_1+x_2=y_1+y_2$, 
and $(r,s)$-sums~$x_1 + \cdots + x_r = y_1 + \cdots + y_s$
in the binomial random subset~$[n]_p=V_p(\cH)$ of the integers; see Section~\ref{sec:intro:examples} and~\ref{sec:applications:RISH} for more details/concrete examples.

The two expressions in the exponent of the upper tail~\eqref{eq:intro:indshg}--\eqref{eq:intro:indshg:both} correspond to different phenomena.\footnote{A phenomenon not relevant for the qualitative accuracy of~\eqref{eq:intro:indshg}--\eqref{eq:intro:indshg:both} is that~$|V_p(\cH)|$ can also be somewhat `bigger' than~$\E |V_p(\cH)|$, which in some range yields sub-Gaussian type tail behaviour, see also~\cite{AP,star}.}  
Namely, in some range we expect that~$X=e(\cH_p)$ is approximately Poisson, in which case $\Pr(X \ge 2\mu)$ decays roughly like $\exp(-c \mu)$. 
Similarly, the $\exp(-c \mu^{1/q}\log(1/p)) = p^{c \mu^{1/q}}$ term intuitively corresponds to `clustered' behaviour (see also~\cite{AP,star,UTSG}), 
where few vertices $U \subseteq V_p(\cH)$ induce many edges in $\cH_p = \cH[V_{p}(\cH)]$: e.g., in each of the above-mentioned examples there always is such a set with~$|U|=c\mu^{1/q}$ and~$e(\cH[U]) \ge 2 \mu$, which readily implies~$\Pr(X \ge 2\mu) \ge \Pr(U \subseteq V_p(\cH)) = p^{c \mu^{1/q}}$. 
Note that classical tail inequalities of form~\eqref{eq:intro:UT} fail to handle these phenomena properly (lacking Poisson behaviour and the extra~$\log(1/p)$ term). 
%\end{remark}

\subsubsection{Upper tail examples from additive combinatorics and Ramsey theory}\label{sec:intro:examples}
In the following exemplary upper tail bounds~\eqref{eq:AP:tail}--\eqref{eq:rssum:tail} we tacitly allow the implicit constants to depend on~$\eps$. 
\begin{example}%[Arithmetic progressions]
\label{ex:AP}
\emph{Arithmetic progressions} (APs) are central objects in additive combinatorics. 
Given $k \ge 3$, let~$X=X_{n,k,p}$ denote the number of arithmetic progressions of length~$k$ in the binomial random subset~$[n]_p$ of the integers (to clarify: we count $k$-subsets~$\{x_1, \ldots, x_k\} \subseteq [n]_p$ forming APs); note that $\mu=\E X = \Theta(n^2p^k)$. 
Then, for any~$\eps>0$ and~$p=p(n) \in (0,1]$ satisfying $1 \le (1+\eps)\mu \le X_{n,k,1}$, we~have 
\begin{equation}\label{eq:AP:tail}
\Pr(X \ge (1+\eps)\mu) = \exp\Bigl(-\Theta\bigl(\min\bigl\{\mu, \; \mu^{1/2} \log(1/p)\bigr\}\bigr)\Bigr) .
\end{equation} 
\end{example}
\begin{example}%[Schur triples]
\label{ex:ST}
\emph{Schur triples}~$\{x,y,z\} \subseteq [n]$ with~$x+y=z$ (where $x \neq y$) are classical objects in Number theory and Ramsey theory (see, e.g.,~\cite{Gr} and \cite{GRR,Schacht2009}). 
Let~$X=X_{n,p}$ denote the number of Schur triples in~$[n]_p$; note that $\mu=\E X = \Theta(n^2p^3)$. 
Then, for any~$\eps>0$ and~$p=p(n) \in (0,1]$ satisfying $1 \le (1+\eps)\mu \le X_{n,1}$, we~have 
\begin{equation}\label{eq:ST:tail}
\Pr(X \ge (1+\eps)\mu) = \exp\Bigl(-\Theta\bigl(\min\bigl\{\mu, \; \mu^{1/2} \log(1/p)\bigr\}\bigr)\Bigr) .
\end{equation} 
The same tail bound also holds for~$\ell$-sums (studied, e.g., in~\cite{BHKLS}), where the $3$-element subsets satisfy~$x+y=\ell z$. 
\end{example}
\begin{example}%[Additive quadruples]
\label{ex:AQ}
\emph{Additive quadruples} are $4$-subsets~$\{x_1,x_2,y_1,y_2\} \subseteq [n]$ satisfying~$x_1+x_2=y_1+y_2$. %, with all the variables being distinct. 
The number of these quadruples is also called additive energy, which is an important quantity in additive combinatorics (see, e.g.,~\cite{BK2012,B2014}). 
Let~$X=X_{n,p}$ denote the number of additive quadruples in~$[n]_p$; note that $\mu=\E X = \Theta(n^3p^4)$. 
Then, for any~$\eps>0$ and~$p=p(n) \in (0,1]$ satisfying $1 \le (1+\eps)\mu \le X_{n,1}$, we~have 
\begin{equation}\label{eq:AQ:tail}
\Pr(X \ge (1+\eps)\mu) = \exp\Bigl(-\Theta\bigl(\min\bigl\{\mu, \; \mu^{1/3} \log(1/p)\bigr\}\bigr)\Bigr) .
\end{equation} 
\end{example}
\begin{example}%[$(r,s)$-sums]
\label{ex:rssum}
\emph{$(r,s)$-sums} are~$(r+s)$-subsets ~$\{x_1,\ldots, x_r,y_1,\ldots, y_2\} \subseteq [n]$ satisfying~$x_1 + \cdots + x_r = y_1 + \cdots + y_s$. % with all the variables being distinct. 
In the special case~$r=s$ the number of these sets is called $2r$-fold additive energy, which is useful in the context of Roth's theorem (see, e.g.,~\cite{B2014}). 
Given $r,s \ge 1$ satisfying $r+s \ge 3$, let~$X=X_{n,r,s,p}$ denote the number of $(r,s)$-sums in~$[n]_p$; note that $\mu=\E X = \Theta(n^{r+s-1}p^{r+s})$. 
Then, for any~$\eps>0$ and~$p=p(n) \in (0,1]$ satisfying $1 \le (1+\eps)\mu \le X_{n,r,s,1}$, we~have 
\begin{equation}\label{eq:rssum:tail}
\Pr(X \ge (1+\eps)\mu) = \exp\Bigl(-\Theta\bigl(\min\bigl\{\mu, \; \mu^{1/(r+s-1)} \log(1/p)\bigr\}\bigr)\Bigr) .
\end{equation} 
\end{example}
\noindent 
Similar tail bounds also hold for integer solutions of linear homogeneous systems, see Section~\ref{sec:applications:RISH} for the~details.

\subsubsection{Subgraph counts in random graphs: sub-Gaussian type upper tail bounds} 
As a side-product, our proof techniques also yield new results with a slightly different flavour. 
To illustrate this with subgraph counts in the binomial random graph~$G_{n,p}$, let $X=X_H$ denote the number of copies of~$H$ in~$G_{n,p}$. Set $\mu = \E X$.  
Here \emph{sub-Gaussian type} upper tail estimates\footnote{For subgraph counts lower tail estimates of sub-Gaussian type follow from Janson's inequality (see, e.g.,~\cite{JW}).} of the form 
\begin{equation}\label{intro:SGT}
\Pr(X \ge \mu+t) \le C\exp(-ct^2/\Var X) 
\end{equation}
have been extensively studied~\cite{SGN,Vu2000,DL,MS,Kannan,WG2011,WG2012} during the last decades, 
usually with emphasis on small deviations of form~$\sqrt{\Var X} \le t = o(\mu)$, say (differing from the large deviations regime~$t = \Theta(\mu)$ considered in the classical upper tail problem for subgraph counts).  
%As we shall spell out in Section~\ref{sec:applications:RG}, in the Poisson range $\Var X \sim \E X=\mu$ we extend and improve 
%the main applications of Wolfovitz~\cite{WG2012} and {\v{S}}ileikis~\cite{MS}. 
In particular, for so-called `strictly balanced' graphs~$H$ %(defined, e.g., in Section~\ref{sec:applications:RG}) 
three different approaches~\cite{Vu2000,DL,MS} have been developed during the years~2000--2012, % (for the Poisson range $\Var X \sim \E X=\mu$), 
which each establish a form of inequality~\eqref{intro:SGT} for~$t \le \mu = O(\log n)$. Our methods allow us to break this logarithmic barrier slightly, answering a question of Janson and Ruci{\'n}ski~\cite{DLP}; see Section~\ref{sec:applications:RG:SG} for more~details.
\begin{theorem}[Subgraph counts: sub-Gaussian type upper tail bounds]% 
\label{thm:intro:sgcount}
For any strictly balanced graph~$H$ there are $n_0,c,C,\xi>0$ such that inequality~\eqref{intro:SGT} holds 
whenever~$n \ge n_0$ and~$0 < t \le \mu \le (\log n)^{1+\xi}$. 
\end{theorem}

%\subsubsection{A remark} 
%%
%The upper tail results (and setup) of this paper are more general than the ones discussed above. %in Section~\ref{sec:results:flavour}. 
%However, we resist the temptation of stating them here, since we feel that the reader will benefit 
%more from knowing the proof methods used (rather than knowing the technical statements of the theorems). 

\subsection{Glimpse of the proof strategy}\label{sec:proof:idea}
In contrast to most of the previous work, in this paper we take a more combinatorial perspective to concentration of measure (and avoid induction via a more iterative point of view). 
Our high-level proof strategy proceeds roughly as follows.  
In the \emph{deterministic part} of the argument, we define several `good' events $\cE_i=\cE_i(\cH,\eps)$, and show that the following implication holds:
\begin{equation}\label{eq:ps:det}
\text{all $\cE_i$ hold} \qquad \Longrightarrow \qquad X < (1+\eps)\E X .
\end{equation}
In the \emph{probabilistic part} of the argument, we show that for some suitable parameter~$\Psi$ we have  
\begin{equation}\label{eq:ps:prob}
\Pr(\text{some $\cE_i$ fails}) \le \exp(- \Psi) .
\end{equation}
Combining both parts then readily yields an exponential upper tail estimate of the form 
\[
\Pr(X \ge (1+\eps)\E X) \le \Pr(\text{some $\cE_i$ fails}) \le \exp(- \Psi) .
\]
In this paper we illustrate the above approach by implementing~\eqref{eq:ps:det}--\eqref{eq:ps:prob} in a general Kim--Vu/Janson--Ruci{\'n}ski type setup. 
To communicate our ideas more clearly, our below informal discussion again uses the simpler random induced subhypergraph setup (a more detailed sketch is given in Sections~\ref{sec:form}--\ref{sec:sketch}).

For the deterministic part~\eqref{eq:ps:det}, we shall crucially 
exploit a good event $\cE_{Q,\eps}$ of the following form: all subhypergraphs with `small' maximum degree have `not too many' edges, 
i.e., that $e(\cJ) < (1+\eps/2) \E X$ holds for all $\cJ \subseteq \cH_p$ with $\Delta_1(\cJ) \le Q$, say.  
Our \emph{sparsification idea} proceeds roughly as follows. 
First, using combinatorial arguments (and further good events) we find a nested sequence of subhypergraphs 
\begin{equation}\label{eq:ps:nested}
\cH_p = \cJ_q \supseteq \cJ_{q-1} \supseteq \cdots \supseteq \cJ_{2} \supseteq \cJ_1 ,
\end{equation} 
which gradually decreases the maximum degree down to $\Delta_1(\cJ_1) \le Q$. 
The crux is that $\cE_{Q,\eps}$ then implies $e(\cJ_1) < (1+\eps/2) \E X$. 
In the second step we exploit various good events (and properties of the constructed sequence) 
to show that we obtained $\cJ_1$ by removing relatively few edges from $\cH_p$, such~that 
\begin{equation}\label{eq:ps:nestedX}
X = e(\cH_p) = e(\cJ_1) + \sum_{1 \le j < q} e(\cJ_{j+1} \setminus \cJ_j) < (1+\eps/2) \E X + (\eps/2) \E X = (1+\eps)\E X .
\end{equation}
In fact, the combinatorial arguments leading to~\eqref{eq:ps:nested}--\eqref{eq:ps:nestedX} develop a `maximal matching' based sparsification idea from~\cite{AP}, 
which is key for handling some vertices of $\cH_p$ with exceptionally high degrees, say.

The probabilistic part~\eqref{eq:ps:prob} works hand in hand with the above deterministic arguments. 
Similar to~$\cE_{Q,\eps}$, we shall throughout work with `relative estimates', i.e., which are valid for all subhypergraphs of~$\cH_p$ satisfying some extra properties (e.g., that $\Delta_{j}(\cJ) \le R_j$ holds for all $J \subseteq \cH_p$ with $\Delta_{j+1}(\cJ) \le R_{j+1}$). 
These estimates are crucial for bringing combinatorial arguments of type~\eqref{eq:ps:nested}--\eqref{eq:ps:nestedX} into play (instead of relying solely on inductive reasoning), and 
they hinge on a concentration inequality from~\cite{AP}. 
Perhaps surprisingly, this inequality allows us to estimate $\Pr(\neg\cE_{Q,\eps})$ and similar `relative' events \emph{without} taking a union bound over all subhypergraphs. 
For the matching based sparsification idea briefly mentioned above, we exploit the fact that the relevant `matchings' guarantee the `disjoint occurrence' of suitably defined events. 
This observation allows us to estimate the probability of certain `bad' events via BK-inequality based moment arguments. 

Finally, in our probabilistic estimates the logarithmic terms in~\eqref{int:bd:2}--\eqref{eq:intro:indshg} arise in a fairly delicate way (which comes as no surprise, since there are examples where~\eqref{eq:intro:UT} is sharp). 
We now illustrate the underlying technical idea for binomial random variables $X \sim \Bin(n,p)$ with $\mu = np$, where for $x \ge e (e/p)^{\alpha}\mu$ we have 
\[
\Pr(X \ge x) \le \binom{n}{x}p^x \le \left(\frac{e \mu}{x}\right)^x \le \left(\frac{p}{e}\right)^{\alpha x} = \exp\Bigl(-\alpha x \log\bigl(e/p\bigr)\Bigr) .
\]
Our proofs apply this `overshooting the expectation yields extra terms in the exponent' idea to a set of carefully chosen auxiliary random variables. 
As the reader can guess, the technical details are, e.g., complicated by the fact 
that the edges of $\cH_p$ are \emph{not} independent, and that we may \emph{not} assume $x \gg \mu$.

\subsection{Guide to the paper} 
In Section~\ref{sec:prelim} we introduce our key probabilistic tools. 
In Section~\ref{sec:core} we give a fairly detailed proof outline, and 
present our main combinatorial and probabilistic arguments in the 
random induced subhypergraphs setup. 
In Section~\ref{sec:general} we then extend the discussed arguments to a more general setup. 
In Section~\ref{sec:UT} we derive some concrete upper tail inequalities,  
which in Section~\ref{sec:applications} are then applied to several pivotal examples.

The reader interested in our \emph{proof techniques} may wish to focus on Section~\ref{sec:core}, which contains our core ideas and arguments. 
The reader interested in \emph{applications} may wish to skip to Section~\ref{sec:applications}, 
where the `easy-to-apply' concentration inequalities of Section~\ref{sec:easy} are used in several different examples. 
Finally, the reader interested in \emph{comparing} our results \emph{with the literature} may wish to focus on the general setup of Section~\ref{sec:general:setup} and the concentration inequalities in Section~\ref{sec:extended}.

\section{Probabilistic preliminaries}\label{sec:prelim}
\subsection{A Chernoff-type upper tail inequality}\label{sec:c}
In this subsection we state a powerful Chernoff-type upper tail inequality from~\cite{AP}. 
It might be instructive to check that, for sums $X=\sum_{i \in \cA} \xi_i$ of independent random variables $\xi_i \in [0,1]$, inequality~\eqref{eq:C} below reduces to the classical Chernoff bound (writing $i \sim j$ if $i=j$, for $Y_{i}=\xi_{i}$, $\cI=\cA$ and $C=1$ we have $X=Z_C$). 
We think of~$\sim$ as a `dependency relation': $\alpha \not\sim \beta$ implies that the random variables $Y_{\alpha}$ and $Y_{\beta}$ are independent. 
For indicator random variables $Y_{\alpha} \in \{0,1\}$ the condition $\max_{\beta \in \cJ}\sum_{\alpha \in \cJ: \alpha \sim \beta} Y_{\alpha} \leq C$ 
essentially ensures that each variable $Y_{\beta}$ with $\beta \in \cJ$ 
`depends' on at most~$C$ variables $Y_{\alpha}$ with $\alpha \in \cJ$. 
Intuitively, $Z_C$~defined below thus corresponds to an approximation of $X=\sum_{\alpha \in \cI} Y_{\alpha}$ with `bounded dependencies'. 
\begin{theorem}\label{thm:C}
Given a family of non-negative random variables $(Y_{\alpha})_{\alpha \in \cI}$ with $\sum_{\alpha \in \cI} \E Y_{\alpha} \le \mu$, assume that~$\sim$ is a symmetric relation on $\cI$ such that each $Y_{\alpha}$ with $\alpha \in \cI$ is independent of $\{Y_{\beta}: \text{$\beta \in \cI$ and $\beta \not\sim\alpha$}\}$. 
Let $Z_C=\max \sum_{\alpha \in \cJ} Y_{\alpha}$, where the maximum is taken over all $\cJ \subseteq \cI$ with $\max_{\beta \in \cJ}\sum_{\alpha \in \cJ: \alpha \sim \beta} Y_{\alpha} \leq C$. 
Set $\varphi(x)=(1+x)\log(1+x)-x$.  
Then for all $C,t>0$ we have 
\begin{equation}\label{eq:C}
\begin{split}
\Pr(Z_C \ge \mu +t) 
& 
\le \exp\left(-\frac{\varphi(t/\mu)\mu}{C} \right) = 
e^{-\mu/C} \cdot 
\left(\frac{e\mu}{\mu+t}\right)^{(\mu+t)/C} \\ 
& 
\le \min\left\{ \exp\left(-\frac{t^2}{2C(\mu +t/3)}\right), \; \left(1+\frac{t}{2\mu}\right)^{-t/(2C)}\right\}  
\le \left(1+\frac{t}{\mu}\right)^{-t/(4C)}. 
\end{split}
\end{equation}
\end{theorem}
\begin{remark}\label{rem:C}
In applications there often is a family of independent random variables $(\xi_\sigma)_{\sigma \in \cA}$ such that each~$Y_{\alpha}$ is a function of $(\xi_\sigma)_{\sigma \in \alpha}$. 
Then it suffices to define $\alpha \sim \beta$ if $\alpha \cap \beta \neq \emptyset$ (as $\alpha \not\sim \beta$ implies that~$Y_{\alpha}$ and~$Y_{\beta}$ depend on disjoint sets of variables~$\xi_\sigma$). 
\end{remark}
\begin{remark}\label{rem:C:NC}
Theorem~\ref{thm:C} remains valid after weakening the independence assumption to a form of negative correlation: it suffices if 
$\E (\prod_{i \in [s]} Y_{\alpha_i}) \le \prod_{i \in [s]}  \E Y_{\alpha_i}$ for all $(\alpha_1, \ldots, \alpha_s) \in \cI^s$ satisfying $\alpha_i \not\sim \alpha_j$ for $i \neq j$.  
For example, writing $\alpha \sim \beta$ if $\alpha \cap \beta \neq \emptyset$, it is not hard to check that this weaker condition holds for variables of form $Y_{\alpha} = w_{\alpha}\indic{\alpha \in \cH_m}$, where the uniform model $\cH_m=\cH[V_m(\cH)]$ is defined as in Section~\ref{sec:uniform}. 
\end{remark}
\begin{remark}\label{rem:C:NC:2}
Replacing the assumption $\sum_{\alpha \in \cI} \E Y_\alpha\le \mu$ of Theorem~\ref{thm:C} with $\sum_{\alpha \in \cI} \lambda_{\alpha} \le \mu$ and $\min_{\alpha \in \cI}\lambda_{\alpha} \ge 0$, the correlation condition of Remark~\ref{rem:C:NC} can be further weakened to $\E (\prod_{i \in [s]} Y_{\alpha_i}) \le \prod_{i \in [s]} \lambda_{\alpha_i}$. 
\end{remark}
\begin{remark}\label{rem:varphi}
Note that inequality~\eqref{eq:C} implies $\varphi(\eps) \ge \eps^2/[2(1+\eps/3)] \ge \min\{\eps^2,\eps\}/3$ for $\eps \ge 0$. 
\end{remark}
\noindent
Remarks~\ref{rem:C:NC}--\ref{rem:C:NC:2} suggest that the proof of Theorem~\ref{thm:C} is fairly robust (it exploits independence only in a limited way; see also the discussion in~\cite{AP} and the proof of Lemma~4.5 in~\cite{K4free}).

\subsection{The BK-inequality}\label{sec:BK}
In this subsection we state a convenient consequence of the BK-inequality of van~den~Berg and Kesten~\cite{BK} and Reimer~\cite{BKR}. %. %, which uses the concept of `disjoint occurrence'.  
As usual in this context, we consider a sample space~$\Omega=\Omega_1 \times \cdots \times \Omega_M$ with finite~$\Omega_i$, and 
write~$\omega=(\omega_1, \ldots, \omega_M) \in \Omega$. 
Given an event~$\cE \subseteq \Omega$ and an index set $I \subseteq [M]=\{1, \ldots, M\}$, we define  
\[
\restr{\cE}{I} = \bigl\{ \omega \in \cE \: : \: \text{for all $\pi \in \Omega$ we have $\pi \in \cE$ whenever $\pi_j=\omega_j$ for all $j \in I$}\bigr\} .
\]
In intuitive words, the event $\restr{\cE}{I}$ occurs if knowledge of the variables indexed by~$I$ already `guarantees' the occurrence of $\cE$ (note that all other variables are irrelevant for $\restr{\cE}{I}$).
Given a collection $(\cE_i)_{i \in \cC}$ of events, for the purposes of this paper it seems easiest to introduce the convenient definition  
\begin{equation}\label{eq:square}
\boxdot_{i \in \cC}\cE_i = \Bigl\{\text{there are pairwise disjoint $I_i \subseteq [M]$ such that $\bigcap_{i \in \cC}\restr{\cE_{i}}{I_i}$ occurs}\Bigr\} .
\end{equation}
The event $\boxdot_{i \in \cC}\cE_i$ intuitively states that all~$\cE_i$ `occur disjointly', 
i.e., that there are \emph{disjoint} subsets of variables which guarantee the occurrence of each event~$\cE_i$ (the definition of~$\boxdot$ 
sidesteps that the usual box product $\square$ is, in general, not associative). 
The general BK-inequality of Reimer~\cite{BKR} implies the following estimate. 
\begin{theorem}\label{thm:BKR:ind}
Let~$\Pr$ be a product measure on~$\Omega=\Omega_1 \times \cdots \times \Omega_M$ with finite~$\Omega_i$. 
Then for any collection~$(\cE_i)_{i \in \cC}$ of events we have  
\begin{equation}\label{eq:BKR:ind}
\Pr\bigl(\boxdot_{i \in \cC}\cE_i\bigr) \le \prod_{i \in \cC}\Pr(\cE_i) .
\end{equation}
\end{theorem}
\begin{remark}\label{rem:BKR:ind}
For increasing events $\cE_i$, 
\cite{BKkoutn} implies that inequality~\eqref{eq:BKR:ind} 
also holds for~$\Pr$ assigning equal probability to all outcomes 
$\omega \in \{0,1\}^M$ with exactly~$m$ ones 
(as usual, an event~$\cE$ is called increasing if for all~$\omega \in \cE$ and~$\pi \in \Omega$ 
we have~$\pi \in \cE$ whenever~$\omega_j \le \pi_j$ for all~$j \in [M]$). 
\end{remark}

\section{Core ideas and arguments}\label{sec:core}
In this section we present our core combinatorial and probabilistic arguments in a slightly simplified setup. 
Our main focus is on the new proof ideas and methods (which we believe are more useful to the reader than the theorems), so we defer applications and concrete upper tail inequalities to Sections~\ref{sec:UT}--\ref{sec:applications}. 
This organization of the paper also makes the extension to the more general setup of Section~\ref{sec:general} more economical. 
Indeed, similar to the high-level proof strategy discussed in Section~\ref{sec:proof:idea}, the main results of this section are Theorem~\ref{thm:sparse:iter} of form $\Pr(X \ge (1+\eps)\E X) \le \sum_i \Pr(\neg\cE_i)$ and Theorem~\ref{thm:prob} of form $\Pr(\neg\cE_i) \le \exp(-\Psi_i)$. 
Together they yield upper tail inequalities, and in Section~\ref{sec:prf:extended} we adapt both to our more general setup.

In Section~\ref{sec:overview} we give a detailed proof overview, and introduce the simpler random induced subhypergraphs setup (where our main arguments and ideas are more natural). 
As a warm-up, in Section~\ref{sec:iterative} we revisit 
existing inductive concentration methods, and reinterpret some of the underlying ideas.
Section~\ref{sec:spars} contains our key combinatorial arguments, which hinge on `sparsification' ideas and the BK-inequality. 
In Section~\ref{sec:prob} these arguments are complemented by probabilistic estimates, 
which rely on the Chernoff-type tail inequality Theorem~\ref{thm:C}. 
Finally, in Section~\ref{sec:uniform} we demonstrate that our proofs are somewhat `robust'.

\subsection{Overview}\label{sec:overview}
\subsubsection{Simplified setup: random induced subhypergraph~$\cH_p$}\label{sec:basic}
Our basic setup concerns random induced subhypergraphs.  
For a hypergraph~$\cH$ with vertex set~$V(\cH)$, let~$V_p(\cH)$ denote the binomial random vertex subset 
where each~$v \in V(\cH)$ is included independently with probability~$p$. 
We define the subhypergraph of $\cH$ induced by $V_p(\cH)$ as 
\begin{equation}
\label{def:Hp}
\cH_p =\cH[V_p(\cH)]. 
\end{equation}
Given non-negative weights $(w_f)_{f \in \cH}$, for every $\cG \subseteq \cH$ we set 
\begin{equation}
\label{def:wG}
w(\cG) = \sum_{f \in \cG} w_f \indic{f \in \cH_p} ,
\end{equation}
where our main focus is on the weighted number of induced edges~$w(\cH)=w(\cH_p)$. 
The `unweighted' case with $w_f=1$ occurs frequently in the literature (see, e.g.,~\cite{UT,RR1994,UTAP,AP,WG2012}), 
where the random variable $w(\cH)=e(\cH_p)$ simply counts the number of edges of~$\cH$ induced by~$V_p(\cH)$. 
Our arguments will also carry over to the uniform variant $\cH_m =\cH[V_m(\cH)]$ defined in Section~\ref{sec:uniform} (see Remark~\ref{rem:prob}).

To formulate our results, we need some more notation and definitions. 
As usual, we write 
\begin{align}
\label{def:GammaUG}
\Gamma_U(\cH) &= \{f \in \cH: U \subseteq f\} ,\\
\label{def:DeltajG}
\Delta_j(\cH) &= \max_{U \subseteq V(\cH): |U|=j} |\Gamma_U(\cH)| . 
\end{align}
In concrete words, $\Gamma_U(\cH)$ corresponds to the set of all edges $f \in \cH$ that contain the vertex subset~$U\subseteq V(\cH)$, 
and $\Delta_j(\cH)$ denotes the maximum number of edges that contain 
$j$~given vertices (which we think of as a `maximum degree' parameter).  
Inspired by~\cite{DL,KimVu2000,Vu2000,Vu2002}, we now define the following two crucial \emph{assumptions}~(P') and~(P$q$), where $q \in \NN$ is a parameter:   
{\begin{enumerate}
\leftmargin 1.5em \itemindent 2.5em \itemsep 0.125em \parskip 0em  \partopsep=0pt \parsep 0em 
	\item[(P')] Assume that 
	$\max_{f \in \cH}|f| \le k$, $\max_{f \in \cH}w_f\le L$ and $v(\cH) \le N$. Define $\mu=\E w(\cH)$ and  
\begin{equation}
\label{def:muj:basic}
\mu_j = \max_{U \subseteq V(\cH): |U|=j} \sum_{f \in \Gamma_U(\cH)}p^{|f|-|U|} . 
\end{equation}
	\item[(P$q$)] Assume that $\Delta_q(\cH) \le D$. 
\end{enumerate}}\noindent 
Property~(P') ensures that every edge $f \in \cH$ has at most~$k$ vertices, that the associated edge weights satisfy $0 \le w_f \le L$, and that~$\cH$ contains at most $v(\cH) \le N$ vertices. 
Although we shall not assume this, our main focus is on the common case where $k+L=O(1)$ and $N=\omega(1)$ holds. 
Property~(P$q$) will be useful when $D=O(1)$ holds for $q < k$ (this is trivial for $q=k$). 
The key parameters $\mu_j$ intuitively quantify the `dependencies' between the edges, 
and we think of them as average variants of the `maximum degree' parameter $\Delta_j(\cH_p)$ from~\eqref{def:DeltajG}.  
To see this, note that $\Pr(f \in \cH_p \mid U \subseteq V_p(\cH))=p^{|f|-|U|}$, so~\eqref{def:muj:basic} equals 
\begin{equation}
\label{def:muj:basic:alt}
\mu_j = 
\max_{U \subseteq V(\cH): |U|=j} \E\bigl( |\Gamma_U(\cH_p)| \; \big| \; U \subseteq V_p(\cH)\bigr) .
\end{equation}
In concrete words, after conditioning on the presence of any vertex subset $U \subseteq V_p(\cH)$ of size $|U|=j$, the expected number of edges in $\cH_p$ that contain~$U$ is at most~$\mu_{j}$ (for this reason, $\mu_j$ can be interpreted as the `maximum average effect' of any $j$ vertices or variables, see also~\cite{KimVu2000,Vu2002}). 
For example, if the edges of the $k$-uniform hypergraph $\cH=\cH_n$ correspond to $k$-term arithmetic progressions, then we can take $V(\cH)=[n]$, $N=n$, $L=1$, $\mu = \Theta(n^2p^k)$ and $\mu_j =\Theta(n^{2-j}p^{k-j})$ for $1 \le j \le q=2$ (note that $\Delta_2(\cH)=O(1)$ holds).

\subsubsection{The basic form of our tail estimates}\label{sec:form} 
In this subsection we discuss the approximate form of our upper tail estimates. 
As we shall see in Section~\ref{sec:iterative}, 
for hypergraphs~$\cH$ with $\Delta_q(\cH) \le D$ the usual inductive concentration of measure methods~\cite{KimVu2000,DL,Vu2002} 
yield basic inequalities of the following form~(omitting several technicalities).  
Given positive parameters~$(R_j)_{1 \le j \le q}$ with~$R_q \ge D$, 
for every~$\eps>0$ there are positive constants~$a=a(\eps,k)$ and~$b=b(k)$ such that~roughly 
\begin{equation}\label{heur:ind}
\Pr(e(\cH_p) \ge (1+\eps)\mu) \le 
\exp\bigl(-a \mu/R_1 \bigr)  
+ \sum_{1 \le j < q} \left(\frac{\mu_j}{R_j}\right)^{b R_j/R_{j+1}} ,
\end{equation}
say (see~\eqref{eq:cl:basic} of Claim~\ref{cl:basic}; the freedom of choosing the parameters~$(R_j)_{1 \le j \le q}$ is part of the method, though one naturally aims at roughly~$\mu/R_1 \approx R_j/R_{j+1}$). 
The `prepackaged versions' of these inequalities 
usually assume that the parameters satisfy roughly $\mu/R_1 \ge \lambda$ and $R_j \ge \max\{2\mu_j,\lambda R_{j+1}\}$ 
(see, e.g., Theorem~4.2 in~\cite{Vu2002} or Theorem~3.10 in~\cite{DL}).  
In this case there are positive constants $c=c(a,b)$ and $C=C(q)$ such that  
\begin{equation}\label{heur:ind:2}
\Pr(e(\cH_p) \ge (1+\eps)\mu) \le C \exp\bigl(- c \lambda \bigr) .
\end{equation}

The punchline of this paper is that we can often improve the exponential decay of~\eqref{heur:ind:2} 
if stronger bounds than $R_j \ge 2 \mu_j$ hold. 
For example, setting $\lambda \approx \mu^{1/q}$ and $R_j \approx \lambda^{q-j}$ (similar to, e.g., the proof of Corollary~6.3 in~\cite{Vu2002} or Theorem~2.1 in~\cite{Vu2001}), in the applications of Section~\ref{sec:applications:RISH} we naturally arrive at bounds of form 
\begin{equation}\label{heur:Rjmuj:cond}
\max_{1 \le j < q} \frac{\mu_j}{R_j} \approx \max_{1 \le j < q} \frac{\mu_j}{\mu^{(q-j)/j}} = O(p^{\alpha}) . 
\end{equation}
It might be instructive to check that~\eqref{heur:Rjmuj:cond} holds with $\alpha=1/2$ for $k$-term arithmetic progressions with $k \ge 3$. 
Intuitively, replacing $R_j \ge 2 \mu_j$ by the stronger assumption~\eqref{heur:Rjmuj:cond} 
improves the exponential decay of the sum-terms in~\eqref{heur:ind} by a factor of roughly~$\log(1/p)$ for small~$p$. 
Hence the $\exp\bigl(-a \mu/R_1 \bigr)$ term in~\eqref{heur:ind} is the main obstacle for improving inequality~\eqref{heur:ind:2}. 
Here our new `sparsification' based approach is key: after some technical work it essentially allows us to replace~$R_1$ by 
\[
Q_{1} = \max\bigl\{R_1/\log(1/p), \: B\bigr\}, 
\]
where $B \ge 1$ is some constant (of course, we later need to be a bit careful when $p \approx 1$ holds, e.g., replacing $\log(1/p)$ with $\log(e/p)$, say). 
More concretely, assuming~\eqref{heur:Rjmuj:cond}, for $\mu/R_1 \ge \lambda$, $R_j \ge \lambda R_{j+1}$ and $p=o(1)$ 
we eventually arrive (ignoring some technicalities) at a bound that is roughly of the form 
\begin{equation}\label{heur:iter}
\begin{split}
\Pr(e(\cH_p) \ge (1+\eps)\mu) & \le 
\exp\bigl(-a \mu/Q_1 \bigr)  
+ \sum_{1 \le j < q}\left[\left(\frac{\mu_j}{R_j}\right)^{b R_j/R_{j+1}} + \left(\frac{\mu_j}{R_j}\right)^{a \mu/R_1} \right] \\
& \le C\exp\Bigl(-c \min\bigl\{\mu, \; \lambda\log(1/p)\bigr\} \Bigr) ,
\end{split}
\end{equation}
with $c=c(a,b,\alpha,B)>0$ and $C=q$ (see~\eqref{eq:thm:extended} of~Theorem~\ref{thm:extended}). 
In words, \eqref{heur:iter} essentially adds a logarithmic factor to the exponent of the classical bound~\eqref{heur:ind:2}. 
This improvement of~\eqref{heur:ind}--\eqref{heur:ind:2} is conceptually important, since in several interesting examples 
the resulting estimate~\eqref{heur:iter} is qualitatively best possible (see Section~\ref{sec:applications:RISH}).

\subsubsection{Sketch of the argument}\label{sec:sketch} 
In this subsection we expand on the high-level proof strategy from Section~\ref{sec:proof:idea}, and give a rough sketch of our main \emph{combinatorial} line of reasoning (the full details are deferred to Sections~\ref{sec:iterative}--\ref{sec:prob} and~\ref{sec:prf:extended}). 
As we shall argue in Section~\ref{sec:iterative}, at the conceptual heart of the usual inductive concentration approaches lies the following combinatorial `degree' event $\cD_{j}$: $\Delta_{j+1}(\cH_p) \le R_{j+1}$ implies $\Delta_{j}(\cH_p) \le R_{j}$. 
Given a hypergraph $\cH$ with $\Delta_q(\cH) \le R_q$, for the induced number of edges $e(\cH_p)$ the basic idea is that an iterative application of the events $\cD_{q-1} \cap \cdots \cap \cD_1$ reduces the upper tail problem to 
\begin{equation}\label{eq:heur:eHp:1}
\begin{split}
\Pr(e(\cH_p) \ge (1+\eps)\mu) & \le \Pr(e(\cH_p) \ge (1+\eps)\mu \text{ and } \Delta_q(\cH_p) \le R_q)\\
& \le \Pr(e(\cH_p) \ge (1+\eps)\mu \text{ and } \Delta_1(\cH_p) \le R_1) + \sum_{1 \le j < q} \Pr(\neg\cD_{j}) .
\end{split}
\end{equation}
It turns out that all the probabilities on the right hand side of~\eqref{eq:heur:eHp:1} can easily be estimated by the concentration inequality Theorem~\ref{thm:C} (see Claim~\ref{cl:basic:2} and Theorem~\ref{thm:prob}), which eventually yields a variant of the upper tail estimate~\eqref{heur:ind}. 
As before, the crux is that smaller values of the `maximum degree'~$R_1$ translate into better tail estimates. 
To surpass the usual inductive approaches, similar to~\eqref{heur:iter} our plan is thus to reduce the `degree bound'~$R_1$ down to~$Q_1$, and here our 
new `sparsification idea' will be key, achieving this `degree reduction' by deleting up to~$\eps \mu/2$ edges.

Our starting point is the observation that, via Theorem~\ref{thm:C}, we can strengthen the 
degree event~$\cD_j$ to all subhypergraphs $\cG \subseteq \cH_p$ (see Claim~\ref{cl:basic:2} and Theorem~\ref{thm:prob}). 
Namely, let $\cD_j^+$ denote the event that $\Delta_{j+1}(\cG) \le Q_{j+1}$ implies $\Delta_{j}(\cG) \le Q_{j}$ for all $\cG \subseteq \cH_p$. 
A crucial aspect of our argument is that the events $\cD_j$, $\cD_j^+$ work hand in hand with the following combinatorial `sparsification' event $\cE_q$: $\Delta_{1}(\cH_p) \le R_1$ implies existence of a subhypergraph $\cG \subseteq \cH_p$ with $e(\cH_p \setminus \cG) \le \eps\mu/2$ and $\Delta_{q-1}(\cG) \le Q_{q-1}$ (tacitly assuming $q \ge 2$). 
Intuitively, $\cE_q$ states that the deletion of `few' edges reduces the degree $\Delta_{q-1}(\cH_p)$ down to $\Delta_{q-1}(\cG) \le Q_{q-1}$.

The basic combinatorial idea of our approach is roughly as follows (see Section~\ref{sec:spars} for the more involved details). 
We first (i) obtain the coarse degree bound $\Delta_1(\cH_p) \le R_1$ via an iterative application of the degree events $\cD_{q-1} \cap \cdots \cap \cD_1$, then (ii) exploit the sparsification event $\cE_q$ to find a subhypergraph $\cG \subseteq \cH_p$ with $e(\cH_p \setminus \cG) \le \eps\mu/2$ and $\Delta_{q-1}(\cG) \le Q_{q-1}$, and finally (iii) deduce the improved degree bound $\Delta_1(\cG) \le Q_1$ via an iterative application of the degree events $\cD^+_{q-2} \cap \cdots \cap \cD^+_1$. 
Taking into account that we obtain $\cG \subseteq \cH_p$ by deleting up to $\eps \mu/2$ edges, for hypergraphs $\cH$ with $\Delta_q(\cH) \le R_q$ we eventually arrive at 
\begin{equation}\label{eq:heur:eHp:2}
\begin{split}
\Pr(e(\cH_p) \ge (1+\eps)\mu) & \le \Pr(e(\cG) \ge (1+\eps/2)\mu \text{ and } \Delta_1(\cG) \le Q_1 \text{ for some $\cG \subseteq \cH_p$}) \\
& \qquad + \sum_{1 \le j < q} \Pr(\neg\cD_{j}) + \Pr(\neg\cE_q) + \sum_{1 \le j < q-1} \Pr(\neg\cD^+_{j}).
\end{split}
\end{equation}
The crux is that we can again obtain good tail estimates for $\Pr(e(\cG) \ge (1+\eps/2)\mu \: \cdots)$ and $\Pr(\neg\cD_{j})+\Pr(\neg\cD^+_{j})$ via Theorem~\ref{thm:C} (see Claim~\ref{cl:basic:2} and Theorem~\ref{thm:prob}), so in~\eqref{eq:heur:eHp:2} it remains to bound $\Pr(\neg\cE_q)$.

To estimate the probability that the sparsification event $\cE_q$ fails, we shall rely on combinatorial arguments and the BK-inequality, developing a `maximal matching' based idea from~\cite{AP}. 
Simplifying slightly (see Section~\ref{sec:DL} for the full details), for any vertex set $U \subseteq V(\cH)$ with $|U|=q-1$ we tentatively call $\cK_U \subseteq \Gamma_U(\cH)=\{f \in \cH: U \subseteq f\}$ with $|\cK_U|=r$ an \emph{$r$-star}, where we set $r=Q_{q-1}$ for brevity. 
The basic idea is to take a maximal vertex disjoint collection of $r$-stars in~$\cH_p$, which we denote by~$\cM$ (to clarify: the edges from any two distinct $r$-stars $\cK_U,\cK_W \in \cM$ are vertex disjoint), and remove all edges $f \in \cH_p$ that are incident to~$\cM$, i.e., which share at least one vertex with some $r$-star from~$\cM$. 
Denoting the resulting subhypergraph by~$\cG \subseteq \cH_p$, using maximality of~$\cM$ it is not difficult to argue that $\Delta_{q-1}(\cG) < r=Q_{q-1}$ holds (otherwise we could add another $r$-star to~$\cM$). 
Furthermore, by construction the deleted number of edges is at most 
\begin{equation}\label{eq:heur:removal}
e(\cH_p \setminus \cG) \le \sum_{\cK_U \in \cM}\sum_{f \in \cK_U}\sum_{v \in f} |\Gamma_{\{v\}}(\cH_p)| \le |\cM| \cdot r \cdot k \cdot \Delta_1(\cH_p) . 
\end{equation}
Since the event~$\cE_{q}$ presupposes $\Delta_{1}(\cH_p) \le R_1$, 
we thus see that $|\cM| \le \eps \mu/(2 r k R_1)$ implies $|\cH_p \setminus \cG| \le \eps\mu/2$. 
It remains to estimate the probability that $|\cM|$ is big, and here we shall exploit the fact that the $r$-stars $\cK_{U} \in \cM$ satisfy two properties: they (i)~are pairwise vertex disjoint, and 
(ii)~each `guarantee' that $|\Gamma_U(\cH_p)| \ge r$ holds. 
Intuitively, the point of~(i) and~(ii) is that $|\cM|$ events of from $|\Gamma_U(\cH_p)| \ge r$ `occur disjointly' in the sense of Section~\ref{sec:BK}, which allows us to bring the BK-inequality~\eqref{eq:BKR:ind} into play. 
Indeed, by analyzing a $\boxdot$-based moment of $\sum_{U:|U|=q-1}\indic{|\Gamma_U(\cH_p)| \ge r}$,  
we then eventually obtain sufficiently good estimates for $\Pr(\neg\cE_q)$, as desired 
(see the proofs of Lemma~\ref{lem:sparse:event} and inequality~\eqref{eq:thm:sparse} of Theorem~\ref{thm:prob}).

As the reader can guess, the actual details are more involved. 
For example, instead of just $\cE_q$ for $\Delta_{q-1}(\cdot)$, we also need to consider similar sparsification events 
for the others degrees $\Delta_j(\cdot)$ with $1 \le j < q$. 
In fact, analogous to $\cD^+_j$, these events must moreover apply to all subhypergraphs $\cG \subseteq \cH_p$ simultaneously (see $\cE_{j,\ell}(x,r,y,z)$ defined in Section~\ref{sec:spars}). 
Furthermore, due to technical reasons, the decomposition~\eqref{eq:heur:eHp:2} requires some extra bells and whistles (see~\eqref{eq:thm:sparse:iter} of Theorem~\ref{thm:sparse:iter}). 
Finally, we have also ignored how Theorem~\ref{thm:C} and the BK-inequality~\eqref{eq:BKR:ind} eventually allow us to convert the decompositions~\eqref{eq:heur:eHp:1}--\eqref{eq:heur:eHp:2} into concrete upper tail inequalities of form~\eqref{heur:ind} and~\eqref{heur:iter}; see Sections~\ref{sec:DL}, \ref{sec:prob}, \ref{sec:prf:extended} and~\ref{sec:prob:thm} for 
these technical calculations.

\subsection{Inductive concentration proofs revisited}\label{sec:iterative} % 
The goal of this warm-up section is to reinterpret the classical 
inductive concentration proofs from~\cite{DL,KimVu2000,Vu2002} using the following `degree intuition': an (improved) upper bound for $\Delta_{j+1}(\cH_p)$ and $\Delta_{1}(\cH_p)$ translates into an improved upper tail estimate for $\Delta_{j}(\cH_p)$ and $w(\cH_p)$, respectively. 
We exemplify this with the following claim, which is usually stated for~$\cG=\cH_p$ only 
(the proof of is based on routine applications of Theorem~\ref{thm:C}, and thus deferred to Section~\ref{sec:prob}). 
We find inequalities~\eqref{eq:basic:e:2}--\eqref{eq:basic:deg:2} below remarkable, 
since they intuitively yield bounds for all subhypergraphs~$\cG \subseteq \cH_p$ \emph{without} taking a union bound. 
\begin{claim}\label{cl:basic:2}
Given $\cH$, assume that (P') holds. 
Then for all $t,x,y>0$ and $1 \le j < k$ we have 
\begin{align}
\label{eq:basic:e:2}
\Pr\bigl(w(\cG) \ge \mu+t \text{ and } \Delta_{1}(\cG) \le y \text{ for some } \cG \subseteq \cH_p\bigr) &\le 
\left(1+\frac{t}{\mu}\right)^{-t/(4Lky)}, \\ 
\label{eq:basic:deg:2}
\Pr\bigl(\Delta_{j}(\cG) \ge \mu_j + x \text{ and } \Delta_{j+1}(\cG) \le y \text{ for some } \cG \subseteq \cH_p\bigr) &\le N^{j} \left(1+\frac{x}{\mu_j}\right)^{-x/(4ky)}. 
\end{align}
\end{claim}
\noindent 
Now, by a straightforward iterative degree argument similar to~\eqref{eq:heur:eHp:1}, we obtain the simple estimate 
\begin{equation}\label{eq:iteration:2}
\begin{split}
& \Pr\bigl(w(\cG) \ge \mu+t \text{ and } \Delta_{q}(\cG) \le R_q \text{ for some } \cG \subseteq \cH_p\bigr) \\ 
& \qquad \le \Pr\bigl(w(\cG) \ge \mu+t \text{ and } \Delta_{1}(\cG) \le R_1 \text{ for some } \cG \subseteq \cH_p\bigr) \\
& \qquad \quad + \sum_{1 \le j < q}\Pr\bigl(\Delta_{j}(\cG) > R_j \text{ and }  \Delta_{j+1}(\cG) \le R_{j+1} \text{ for some } \cG \subseteq \cH_p\bigr) .
\end{split}
\end{equation}
Restricting to the special case~$w(\cH_p)$, 
using Claim~\ref{cl:basic:2} it turns out that inequality~\eqref{eq:iteration:2} is essentially equivalent to the basic induction of Janson and Ruci{\'n}ski~\cite{DL} (see the proof of Theorem~3.10 in~\cite{DL}), which in turn qualitatively recovers the upper tail part of Kim and Vu~\cite{KimVu2000} (see Section~5 of~\cite{DL,DLP}). 
The iterative point of view~\eqref{eq:iteration:2} is somewhat more flexible than induction, making the arguments subjectively easier to modify (as there is no need to formulate a suitable induction hypothesis). 
Estimates for all subhypergraphs $\cG \subseteq \cH_p$ also make room for additional combinatorial arguments, which is crucial for the purposes of this paper.

\subsection{Combinatorial sparsification: degree reduction by deletion}\label{sec:spars} 
In this section we introduce our key combinatorial arguments, which eventually allow us to obtain improved upper tail estimates by `sparsifying' $\cH_p$, i.e., deleting edges from $\cH_p$. 
Loosely speaking, via this \emph{sparsification idea} we can effectively ignore certain `exceptional' edges from $\cH_p$ (which contain vertices with extremely high degree, say). 
For the purpose of this paper, we encapsulate this heuristic idea with the definition below. 
In intuitive words, for $\ell=1$ the `sparsification' event $\cE_{j,1}(x,r,y,z)$ essentially ensures that every $\cG \subseteq \cH_p$ with bounded $\Delta_{j+1}(\cG)$ and $\Delta_{1}(\cG)$ contains a large subhypergraph $\cJ \subseteq \cG$ with small $\Delta_{j}(\cJ)$. 
\begin{definition}[Sparsification event]%  
Let $\cE_{j,\ell}(x,r,y,z)$ denote the event that for every $\cG \subseteq \cH_p$ with $\Delta_{j+1}(\cG) \le y$ and $\Delta_{\ell}(\cG) \le z$ there is $\cJ \subseteq \cG$ with $\Delta_{j}(\cJ) \le x$ and $e(\cG \setminus \cJ) \le r$.
\end{definition} 
\noindent 
Here one conceptual difference to the `deletion lemma' of R{\"o}dl and Ruci{\'n}ski~\cite{RR1994,UT} 
is that our focus is on `local properties' such as degrees (somewhat in the spirit of~\cite{SSW}), and not on `global properties' such as subgraph counts. 
Furthermore, we are deleting edges from $\cH_p=\cH[V_p(\cH)]$, whereas the classical approach corresponds to deleting vertices from~$V_p(\cH)=E(G_{n,p})$, say.

With $\cE_{j,1}(x,r,y,z)$ in hand, we now refine\footnote{Note that by setting $D_j=R_j=S_j$ the indicators in~\eqref{eq:thm:sparse:pj2}--\eqref{eq:thm:sparse:pj3} are zero, so~\eqref{eq:thm:sparse:iter} qualitatively reduces to~\eqref{eq:iteration:2}.} the basic estimate~\eqref{eq:iteration:2} via the strategy outlined in Section~\ref{sec:sketch} (see also~\eqref{eq:heur:eHp:2} therein). 
We believe that the ideas used in the proof of Theorem~\ref{thm:sparse:iter} below are more important than its concrete statement (which is optimized for the purposes of this paper). 
Here one new ingredient is the edge deletion of the sparsification events in $\Pr_{j,3,\ell}$ of~\eqref{eq:thm:sparse:pj3}, which allows us to decrease certain maximum degrees. 
The total weight of the deleted edges can be as large as 
$t/2$, which is the reason why in~\eqref{eq:thm:sparse:iter} we need to relax $w(\cG) \ge \mu+t$ to $w(\cG) \ge \mu+t/2$. 
In later applications we shall use $S_j \approx R_j/s$ with $s=\omega(1)$, 
and then the parametrization $Q_j = \max\{S_j,D_j\}$ allows us to 
easily deal with $S_j = o(1)$ border cases. 
The indicators in~\eqref{eq:thm:sparse:pj2}--\eqref{eq:thm:sparse:pj3} can safely be ignored on first reading 
(they mainly facilitate certain technical estimates). 
A key aspect of~\eqref{eq:thm:sparse:iter} is that we intuitively replace $\Delta_{1}(\cG) \le R_1$ of~\eqref{eq:iteration:2} with $\Delta_{1}(\cG) \le \min\{Q_1,R_1\}$, 
which by the discussion of Section~\ref{sec:iterative} is crucial for obtaining improved tail estimates (see also Theorem~\ref{thm:prob}). 
\begin{theorem}[Combinatorial decomposition of the upper tail]%
\label{thm:sparse:iter}
Given $\cH$ with $1 \le q \le k$, assume that (P') holds. 
Suppose that $t>0$. 
Given positive $(D_j)_{1 \le j \le q}$, $(R_j)_{1 \le j < q}$ and $(S_j)_{1 \le j < q}$, define $R_q=Q_q=D_q$ and $Q_j = \max\{S_j,D_j\}$ for $1 \le j < q$.  
Then we have  
\begin{equation}\label{eq:thm:sparse:iter}
\begin{split}
& \Pr\bigl(w(\cG) \ge \mu+t \text{ and } \Delta_{q}(\cG) \le D_q \text{ for some } \cG \subseteq \cH_p\bigr) \\
& \qquad \le \Pr\bigl(w(\cG) \ge \mu+t/2 \text{ and } \Delta_{1}(\cG) \le \min\{Q_1,R_1\} \text{ for some } \cG \subseteq \cH_p\bigr) \\
& \qquad \quad + \sum_{1 \le j < q} \bigl[\Pr_{j,1} + \Pr_{j,2} +  \Pr_{j,3,1}\bigr] ,
\end{split}
\end{equation}
where
\begin{align}
\label{eq:thm:sparse:pj1}
\Pr_{j,1} & = \Pr\bigl(\Delta_{j}(\cG) > R_j \text{ and } \Delta_{j+1}(\cG) \le R_{j+1} \text{ for some } \cG \subseteq \cH_p\bigr), \\
\label{eq:thm:sparse:pj2}
\Pr_{j,2} & = \indic{Q_j<R_j \text{ and } Q_{j+1}>D_{j+1}}\Pr\bigl(\Delta_{j}(\cG) > Q_j \text{ and } \Delta_{j+1}(\cG) \le S_{j+1} \text{ for some } \cG \subseteq \cH_p \bigr), \\
\label{eq:thm:sparse:pj3}
\Pr_{j,3,\ell} & = \indic{Q_j<R_j \text{ and } Q_{j+1}=D_{j+1}}\Pr\bigl(\neg\cE_{j,\ell}(Q_{j},t/(2Lq),D_{j+1},R_\ell)\bigr) .
\end{align}
\end{theorem}
\noindent 
The combinatorial proof proceeds in two sparsification rounds. 
In the first round we use our usual iterative degree argument to deduce that $\Delta_q(\cG) \le R_q$ implies $\Delta_j(\cG) \le R_j$ for all $1 \le j \le q$. 
We start the second round with the sparsification event, by deleting edges such that $\cJ \subseteq \cG$ satisfies $\Delta_{q-1}(\cJ) \le Q_{q-1}$ (tacitly assuming $Q_{q-1}< R_{q-1}$, say). 
The idea is that our usual iterative degree argument should then allow us to deduce that $\Delta_{j+1}(\cJ) \le Q_{j+1}$ implies $\Delta_j(\cJ) \le Q_j$ for all $1 \le j < q-1$. 
Unfortunately, our later probabilistic estimates break down if the parameter $Q_{j+1}$ is `too small'. 
With foresight we thus use our alternative `degree reduction' argument whenever $Q_{j+1}=D_{j+1}$ holds, i.e., we again delete edges. 
\begin{proof}[Proof of Theorem~\ref{thm:sparse:iter}]
Inequality~\eqref{eq:thm:sparse:iter} is trivial for~$q=1$ (since~$R_1=Q_1=D_1$). 
For~$q \ge 2$ the plan is to show that properties~(a)--(d) below \emph{deterministically} imply that~$w(\cG) < \mu+t$ for every~$\cG \subseteq \cH_p$ with~$\Delta_{q}(\cG) \le D_q$. 
Using a union bound argument this then completes the proof 
(it is routine to check that~(a)--(d) correspond to the complements of the events on the right hand side of~\eqref{eq:thm:sparse:iter}, 
since $Q_{j+1}>D_{j+1}$ implies $S_{j+1}=Q_{j+1}$). 
Turning to the details, we henceforth assume that the following properties hold for all~$\cG \subseteq \cH_p$ and~$1 \le j < q$:% 
{\vspace{-0.5em}\begin{enumerate}
\itemindent 0.75em \itemsep 0.125em \parskip 0em  \partopsep=0pt \parsep 0em 
 \item[(a)] $\Delta_{1}(\cG) \le \min\{Q_1,R_1\}$ implies $w(\cG) < \mu+t/2$, 
 \item[(b)] $\Delta_{j+1}(\cG) \le R_{j+1}$ implies $\Delta_{j}(\cG) \le R_{j}$, 
 \item[(c)] if $Q_j<R_j$ and $Q_{j+1}>D_{j+1}$, then $\Delta_{j+1}(\cG) \le Q_{j+1}$ implies $\Delta_j(\cG) \le Q_j$, and 
 \item[(d)] if $Q_j < R_j$ and $Q_{j+1}=D_{j+1}$, then $\Delta_{j+1}(\cG) \le Q_{j+1}$ and $\Delta_{1}(\cG) \le R_1$ implies existence of $\cJ \subseteq \cG$ with $\Delta_{j}(\cJ) \le Q_{j}$ and $e(\cG \setminus \cJ) \le t/(2Lq)$.% 
 \end{enumerate}\vspace{-0.5em}\noindent}%
For the remaining deterministic argument 
we fix $\cG \subseteq \cH_p$ with $\Delta_{q}(\cG) \le D_q$, 
and claim that we can construct a hypergraph sequence $\cG=\cJ_q \supseteq \cdots \supseteq \cJ_{1}$ such that 
\begin{align}
\label{eq:seq:degij}
\Delta_{i}(\cJ_j) & \le 
 \begin{cases}
		R_i, & ~~\text{if $1 \le i < j$}, \\
		\min\{Q_i,R_i\}, & ~~\text{if $j \le i \le q$}, 
	\end{cases}\\
\label{eq:seq:edgesij}
e(\cJ_{j+1} \setminus \cJ_{j}) & \le t/(2Lq) .
\end{align}
With this sequence in hand, using~\eqref{eq:seq:edgesij} we have 
\[
w(\cJ_{j+1} \setminus \cJ_{j}) = \sum_{f \in \cJ_{j+1} \setminus \cJ_{j}} w_f \le \bigl(\max_{f \in \cJ_{j+1} \setminus \cJ_{j}} w_f\bigr) \cdot e(\cJ_{j+1} \setminus \cJ_{j}) \le L \cdot  t/(2Lq) = t/(2q) ,
\]
which together with~$\Delta_{1}(\cJ_1) \le \min\{Q_1,R_1\}$ of~\eqref{eq:seq:degij} and~(a) then yields 
\begin{equation}\label{eq:thm:sparse:iter:removed}
w(\cG) = w(\cJ_1) + \sum_{1 \le j < q} w(\cJ_{j+1} \setminus \cJ_{j}) < (\mu + t/2) + (q-1) \cdot t/(2q) \le \mu + t .
\end{equation}
It thus remains to construct $\cG=\cJ_q \supseteq \cdots \supseteq \cJ_{1}$ with the claimed properties. 
For the base case~$\cG=\cJ_q$, using $\Delta_{q}(\cJ_q) = \Delta_{q}(\cG)\le D_q=R_q$ repeated applications of~(b) yield that $\Delta_{i}(\cJ_q) \le R_{i}$ for all $1 \le i \le q$, so~\eqref{eq:seq:degij} holds since $\Delta_{q}(\cJ_q) \le R_q = \min\{R_q, Q_q\}$. 
Given $\cJ_{j+1}$ with $1 \le j < q$, our construction of $\cJ_j \subseteq \cJ_{j+1}$ distinguishes several cases; 
in view of $\Delta_{i}(\cJ_j) \le \Delta_{i}(\cJ_{j+1})$ 
it clearly suffices to check~\eqref{eq:seq:degij} for~$\Delta_{j}(\cJ_j)$ only.  

If~$Q_{j}\ge R_j$, then we set $\cJ_{j}=\cJ_{j+1}$, which satisfies $\Delta_{j}(\cJ_{j}) = \Delta_{j}(\cJ_{j+1}) \le R_j = \min\{Q_j,R_j\}$ by~\eqref{eq:seq:degij}.

If~$Q_{j}< R_j$ and $Q_{j+1}>D_{j+1}$, then we set $\cJ_{j}=\cJ_{j+1}$, which by~\eqref{eq:seq:degij} satisfies $\Delta_{j+1}(\cJ_{j}) =\Delta_{j+1}(\cJ_{j+1}) \le Q_{j+1}$. Hence~(c) implies $\Delta_{j}(\cJ_{j}) \le Q_j= \min\{Q_j,R_j\}$. 

Finally, if~$Q_{j}< R_j$ and $Q_{j+1}=D_{j+1}$, then by~\eqref{eq:seq:degij} we have $\Delta_{j+1}(\cJ_{j+1}) \le Q_{j+1}$ and $\Delta_{1}(\cJ_{j+1}) \le R_1$. 
Hence~(d) implies existence of $\cJ_{j} \subseteq \cJ_{j+1}$ satisfying $\Delta_{j}(\cJ_{j}) \le Q_j = \min\{Q_j,R_j\}$ and $e(\cJ_{j+1} \setminus \cJ_{j}) \le t/(2Lq)$, completing the proof. 
\end{proof}
The above proof demonstrates that estimates for all subhypergraphs $\cG \subseteq \cH_p$ 
are extremely powerful along with combinatorial arguments. 
It seems likely that the above sparsification approach can be sharpened in specific 
applications, i.e., that there is room for alternative (ad-hoc) arguments which apply the `degree reduction' idea differently. %; see also~\eqref{eq:sparse:simpler}.
For example, in~\cite{AP} the degrees are iteratively reduced by a factor of two, say (replacing the finite sum in~\eqref{eq:thm:sparse:iter:removed} by a convergent geometric series). 
In~\cite{star} the iterative argument also takes `trivial' upper bounds for the~$\Delta_j(\cH)$ into account (which can be smaller than~$R_j$ or~$Q_j$).

\subsubsection{A combinatorial local deletion argument}\label{sec:DL} 
The goal of this subsection is to estimate $\Pr\bigl(\neg \cE_{j,1}(x,r,y,z)\bigr)$, i.e., the probability that our `sparsification' event fails. 
As indicated in Section~\ref{sec:sketch}, our proof uses a \emph{maximal matching based idea} which relies on combinatorial arguments and the BK-inequality. 
The following auxiliary event~$\cD_{U,x,y}$ intuitively states that, in~$\cH_p$, the vertex set~$U$ is the centre of a `star' with at least~$x$ spikes (satisfying 
some degree constraint). 
\begin{definition}[Auxiliary degree event]%  
Let $\cD_{U,x,y}$ denote the event that there is $\cK \subseteq \Gamma_{U}(\cH_p)$ with $|\cK| \ge x$ and $\Delta_{|U|+1}(\cK) \le y$. 
\end{definition} 
\noindent 
To put this definition into our `all subhypergraphs' context, note that~$\neg \cD_{U,x,y}$ implies $|\Gamma_{U}(\cG)| < x$ for all~$\cG \subseteq \cH_p$ with~$\Delta_{|U|+1}(\cG) \le y$. 
It might also be instructive to note that a union bound argument yields  
\begin{equation}\label{eq:deg:UB}
\Pr\bigl(\Delta_{j}(\cG) \ge x \text{ and } \Delta_{j+1}(\cG) \le y \text{ for some } \cG \subseteq \cH_p\bigr) \le \sum_{U \subseteq V(\cH): |U|=j}\Pr(\cD_{U,x,y}) .
\end{equation}

The next result relates the auxiliary event~$\cD_{U,x,y}$ 
with the sparsification event~$\cE_{j,1}(x,r,y,z)$. 
For example, $\sum_U \Pr(\cD_{U,x,y}) \le B^{-x/y}$ 
translates into $\Pr(\neg \cE_{j,1}(x,r,y,z)) \le B^{-r/(kyz)}$ 
by inequality~\eqref{eq:lem:sparse:event}. 
\begin{lemma}[Auxiliary result for the sparsification event]%
\label{lem:sparse:event}
Given $\cH$, assume that $\max_{f \in \cH} |f| \le k$ holds. 
Then for all $x,r,y,z>0$ and $1 \le j < k$ we have 
\begin{equation}\label{eq:lem:sparse:event}
\Pr\bigl(\neg \cE_{j,1}(x,r,y,z)\bigr) \le \biggl(\sum_{U \subseteq V(\cH): |U|=j}\Pr(\cD_{U,x,y})\biggr)^{\ceilL{r/(k\ceil{x}z)}} . 
\end{equation}
\end{lemma}
\begin{remark}\label{rem:sparse:event}
Inequality~\eqref{eq:lem:sparse:event} remains valid after dividing the right hand side by $\ceilL{r/(k\ceil{x}z)}!$. 
\end{remark}
\noindent 
The proof of Lemma~\ref{lem:sparse:event} develops a combinatorial idea from~\cite{AP}, which in turn was partially inspired by~\cite{Spencer1990,UT}. 
We call $(U,\cK_U)$ an \emph{$(j,x,y)$-star in~$\cG$} if $U \subseteq V(\cG)$ and $\cK_U \subseteq \Gamma_{U}(\cG)=\{f \in \cG : U \subseteq f\}$ satisfy $|U|=j$, $|\cK_U|=\ceil{x}$ and $\Delta_{j+1}(\cK_U) \le y$. 
Note that we allow for overlaps of the edges $f,g \in \cK_U$ outside of the `centre'~$U$. 
Writing $\cS_{j,x,y}(\cG)$ for the collection of all $(j,x,y)$-stars in~$\cG$, we define $M_{j,x,y}(\cG)$ as the size of the largest 
$\cM \subseteq \cS_{j,x,y}(\cG)$ satisfying $V(\cK_U) \cap V(\cK_W) = \emptyset$ for all distinct $(U,\cK_U),(W,\cK_W) \in \cM$.  
In intuitive words, $M_{j,x,y}(\cG)$ denotes the size of the `largest $(j,x,y)$-star matching' in~$\cG$, i.e., vertex-disjoint collection of stars. 
We are now ready to follow the strategy sketched in Section~\ref{sec:sketch} (see also~\eqref{eq:heur:removal} therein). 
\begin{proof}[Proof of Lemma~\ref{lem:sparse:event}]
Let $\tilde{r}=r/(k\ceil{x}z)$ and $R=\ceil{\tilde{r}}$. We first assume that $M_{j,x,y}(\cH_p) \le \tilde{r}$ holds, and claim that this implies the occurrence of $\cE_{j,1}(x,r,y,z)$. 
For any $\cG \subseteq \cH_p$ with $\Delta_{j+1}(\cG) \le y$ and $\Delta_{1}(\cG) \le z$, it clearly suffices to show that there is $\cJ \subseteq \cG$ with $\Delta_{j}(\cJ) \le x$ and $e(\cG \setminus \cJ) \le r$. 
Let $\cM \subseteq \cS_{j,x,y}(\cG)$ attain the maximum in the definition of $M_{j,x,y}(\cG)$. 
We then remove all edges~$f \in \cG$ which overlap some star $(U,\cK_U) \in \cM$, where overlap means that $f \cap g \neq \emptyset$ for some edge $g \in \cK_U$. 
We denote the resulting subhypergraph by $\cJ \subseteq \cG$. 
Using $\Delta_{j+1}(\cJ) \le \Delta_{j+1}(\cG) \le y$ and maximality of $\cM$, we then infer $\Delta_{j}(\cJ) \le \ceil{x}-1 < x$ (because otherwise we could add another $(j,x,y)$-star to $\cM$). 
Furthermore, since $|\cM| = M_{j,x,y}(\cG) \le M_{j,x,y}(\cH_p) \le \tilde{r}$ and $\Delta_{1}(\cG) \le z$, by construction the number of deleted edges is at most 
\begin{equation}\label{eq:removalG}
e(\cG \setminus \cJ) \le \sum_{K_U \in \cM}\sum_{f \in \cK_U}\sum_{v \in f}|\Gamma_{\{v\}}(\cG)| \le |\cM|  \cdot \ceil{x} \cdot \bigl(\max_{f \in \cG}|f|\bigr) \cdot \Delta_1(\cG) \le \tilde{r} \cdot \ceil{x} k z = r .
\end{equation}
It follows that $M_{j,x,y}(\cH_p) \le \tilde{r}$ implies $\cE_{j,1}(x,r,y,z)$, as claimed.

For~\eqref{eq:lem:sparse:event} it remains to estimate $\Pr(M_{j,x,y}(\cH_p) > \tilde{r})$. 
Similar to the proof of Theorem~11 in~\cite{AP}, we set 
\begin{equation}\label{def:ZR}
Z_R = \sum_{\substack{(U_1, \ldots, U_R):\\ U_i \subseteq V(\cH) \text{ and } |U_i|=j}} \indicb{\boxdot_{i \in [R]}\cD_{U_i,x,y}} ,
\end{equation}
where $\boxdot$ is defined as in~\eqref{eq:square}. 
If $M_{j,x,y}(\cH_p) > \tilde{r}$, then there is $\cM \subseteq \cS_{j,x,y}(\cH_p)$ of size $|\cM|= \ceil{\tilde{r}} = R$ which satisfies $V(\cK_U) \cap V(\cK_W) = \emptyset$ for all distinct $(U,\cK_U),(W,\cK_W) \in \cM$. 
So, since the disjoint vertex sets $V(\cK_U) \subseteq V_p(\cH)$ guarantee the occurrence of each event $\cD_{U,x,y}$, it follows that $\boxdot_{(U,\cK_U) \in \cM}\cD_{U,x,y}$ occurs. 
As $U \subseteq V(\cK_U)$ holds, by vertex disjointness of the $V(\cK_U)$ we deduce that the corresponding `star-centres'~$U$ are distinct. 
Since~$Z_R$ counts ordered $R$-tuples, we thus infer $Z_R \ge R!$. Hence, Markov's inequality yields  
\begin{equation}\label{eq:MZQ}
\Pr(M_{j,x,y}(\cH_p) > \tilde{r}) \le \Pr(Z_R \ge R!) \le (\E Z_R)/R! .
\end{equation}
Turning to~$\E Z_R$, using the BK-inequality~\eqref{eq:BKR:ind} we readily obtain
\begin{equation}\label{eq:EZQ}
\begin{split}
\E Z_R & = \sum_{\substack{(U_1, \ldots, U_R):\\ U_i \subseteq V(\cH) s\text{ and } |U_i|=j}} \Pr\bigl(\boxdot_{i \in [R]}\cD_{U_i,x,y}\bigr) \\
& \le \sum_{\substack{(U_1, \ldots, U_R):\\ U_i \subseteq V(\cH) \text{ and } |U_i|=j}} \prod_{i \in [R]}\Pr(\cD_{U_i,x,y}) \le \biggl(\sum_{U \subseteq V(\cH): |U|=j} \Pr(\cD_{U,x,y})\biggr)^R ,
\end{split}
\end{equation}
which together with~\eqref{eq:MZQ} and $R \ge 1$ completes the proof. 
\end{proof}
The `star-matching' based deletion argument used in the above proof seems of independent interest. 
In applications it might be easier to avoid $\cE_{j,1}(x,r,y,z)$, and directly work with 
the random variable $M_{j,x,y}(\cH_p)$, see also~\cite{AP,star}. 
The above estimates~\eqref{eq:MZQ}--\eqref{eq:EZQ} 
exploit the BK-inequality to relate $M_{j,x,y}(\cH_p)$ with the simpler events $\cD_{U,x,y}$. 
In $\cH_p$ and other probability spaces one can sometimes also estimate $\Pr(M_{j,x,y}(\cH_p) \ge z)$ more directly 
(see, e.g., the remark after the proof of Lemma~17 in~\cite{AP}, or the proof of Lemma~9 in~\cite{star}).

\subsection{Probabilistic estimates}\label{sec:prob} 
In this section we introduce our key probabilistic estimates, which complement the combinatorial decomposition of Theorem~\ref{thm:sparse:iter}, i.e., 
allow us to bound the right hand side of~\eqref{eq:thm:sparse:iter}.  
A key aspect of inequalities~\eqref{eq:thm:e}--\eqref{eq:thm:deg} is that improved degree constraints $\Delta_{i}(\cG) \le y$ translate into improved tail estimates. 
In our applications~\eqref{eq:thm:sparse} below often reduces to $\Pr\bigl(\neg \cE_{j,1}(x,r,y,z)\bigr) \le (e\mu_j/x)^{-\Theta(r/(y z))}$, say 
(see, e.g., the proof of Theorem~\ref{thm:extended}). 
\begin{theorem}[Probabilistic upper tail estimates]%
\label{thm:prob}
Given $\cH$, assume that (P') holds. 
Set $\varphi(x)=(1+x)\log(1+x)-x$.  
Then for all $x,r,y,z,t>0$ and $1 \le j < k$ we have 
\begin{align}
\label{eq:thm:e}
\Pr\bigl(w(\cG) \ge \mu+t/2 \text{ and } \Delta_{1}(\cG) \le y \text{ for some } \cG \subseteq \cH_p\bigr) & \le \exp\left(-\frac{\varphi(t/\mu)\mu}{4Lky} \right) , \\
\label{eq:thm:deg}
\Pr\bigl(\Delta_{j}(\cG) \ge x \text{ and } \Delta_{j+1}(\cG) \le y \text{ for some } \cG \subseteq \cH_p\bigr) & \le N^{j} \left(\frac{e\mu_j}{x}\right)^{x/(ky)} ,\\
\label{eq:thm:sparse}
\Pr\bigl(\neg \cE_{j,1}(x,r,y,z)\bigr) & \le \left(N^j \left(\frac{e\mu_j}{\ceil{x}}\right)^{\ceil{x}/(ky)}\right)^{\ceilL{r/(k\ceil{x}z)}}.
\end{align}
\end{theorem}
\noindent 
The proofs of~\eqref{eq:thm:e}--\eqref{eq:thm:deg} are based on fairly \emph{routine} applications of Theorem~\ref{thm:C}. 
The crux is that the restrictions $\Delta_{1}(\cG) \le y$ and $\Delta_{j+1}(\cG) \le y$ 
translate into bounds for the parameter~$C$ in~\eqref{eq:C}, which intuitively controls the `largest dependencies' 
($\Delta_{1}(\cG) \le y$ ensures that every edge $f \in \cG$ overlaps at most $|f| \cdot \Delta_{1}(\cG) \le k y$ edges~$e \in \cG$). 
For verifying the independence assumption of Theorem~\ref{thm:C}, we use the following simple observation: 
$e \cap f = \emptyset$ implies that $\indic{e \in \cH_p}=\indic{e \subseteq V_p(\cH)}$ and $\indic{f \in \cH_p}=\indic{f \subseteq V_p(\cH)}$ are independent, 
since both depend on \emph{disjoint} sets of independent variables $\xi_{\sigma} = \indic{\sigma \in V_p(\cH)}$. 
Assuming $(e \cap f) \setminus U = \emptyset$, we below exploit that an analogous (conditional independence) reasoning 
works after conditioning on $U \subseteq V_{p}(H)$. 
\begin{proof}[Proof of Theorem~\ref{thm:prob}]
With an eye on Theorem~\ref{thm:C}, inspired by Remark~\ref{rem:C} we set $\xi_\sigma = \indic{\sigma \in V_p(\cH)}$.

We first prove~\eqref{eq:thm:e}. Let $Y_f = w_f \indic{f \in \cH_p}$, which satisfies $Y_f = w_f \prod_{\sigma \in f} \xi_\sigma$ and $\sum_{f \in \cH}\E Y_f = \E w(\cH) = \mu$. 
Furthermore, $w(\cG) = \sum_{w \in \cG} Y_f$ for any $\cG \subseteq \cH_p$. 
Defining $\alpha \sim \beta$ if $\alpha \cap \beta \neq \emptyset$, the independence assumption of Theorem~\ref{thm:C} holds by Remark~\ref{rem:C}. 
Observe that for any $f \in \cG \subseteq \cH$ with $\Delta_{1}(\cG) \le y$ we have 
\[
\sum_{e \in \cG: e \sim f} Y_e \le \bigl(\max_{e \in \cG} w_e\bigr) \cdot \sum_{e \in \cG: e \cap f \neq \emptyset}\indic{e \in \cH_p} \le L \cdot \sum_{v \in f} |\Gamma_{\{v\}}(\cG)|  \le L \cdot |f| \cdot \Delta_1(\cG) \le Lk y. 
\]
To sum up, if $w(\cG) \ge \mu+t/2$ and $\Delta_{1}(\cG) \le y$ for some $\cG \subseteq \cH_p$, then $Z_C \ge \mu+t/2$ holds with $C=Lky$, where $Z_C$ is defined as in Theorem~\ref{thm:C} with $\cI = \cH$. 
So, applying~\eqref{eq:C}, we deduce   
\begin{equation}\label{eq:thm:e:proof}
\Pr\bigl(w(\cG) \ge \mu+t/2 \text{ and } \Delta_{1}(\cG) \le y \text{ for some } \cG \subseteq \cH_p\bigr) \le \Pr(Z_C \ge \mu+t/2) \le \exp\left(-\frac{\varphi(t/(2\mu))\mu}{Lky} \right) .
\end{equation}
Using calculus (see, e.g., the proof of Lemma~13 in~\cite{AP}) it is easy to check that $\varphi(t/(2\mu) \ge \varphi(t/\mu)/4$. 
In view of~\eqref{eq:thm:e:proof} and~\eqref{eq:C}, inequality~\eqref{eq:thm:e} now follows.

Next we turn to~\eqref{eq:thm:deg}, which hinges on the union bound estimate~\eqref{eq:deg:UB}. 
Note that $v(\cH)< 1$ implies $\cH = \emptyset$, so~\eqref{eq:thm:deg} is trivial for $N < 1$ (the left hand side is zero). 
Similarly, \eqref{eq:thm:deg} is also trivial for $x \le e\mu_j$ and $N \ge 1$ (the expression on the right hand side is at least one). 
To sum up, we henceforth may assume~$x > e\mu_j$ and~$N \ge 1$. 
Given $U \subseteq V(\cH)$ with $|U|=j$, set $\cI := \Gamma_U(\cH) = \{f \in \cH: U \subseteq f\}$. 
Let~$Y_f = \indic{f \in \cH_p}$, and define~$\alpha \sim \beta$ if~$(\alpha \cap \beta)\setminus U \neq \emptyset$.
Note that for any $f \in \cK \subseteq \cI$ with $\Delta_{|U|+1}(\cK) \le y$ we have 
\begin{equation}\label{eq:thm:deg:proof:overlap}
\sum_{e \in \cK: e \sim f} Y_e = \sum_{e \in \cK: (e \cap f)\setminus U \neq \emptyset} \indic{e \in \cH_p} \le \sum_{v \in f \setminus U} |\Gamma_{U \cup \{v\}}(\cK)| \le |f \setminus U| \cdot \Delta_{|U|+1}(\cK) \le k y .
\end{equation}
So, if $\cD_{U,x,y}$ occurs, then $Z_C \ge x$ holds with $C=ky$, where $Z_C$ is defined as in Theorem~\ref{thm:C} with $\cI = \Gamma_U(\cH)$.
For $f \in \cI$, note that $U \not\subseteq V_p(\cH)$ implies $f \not\in \cH_p = \cH[V_p(\cH)]$. 
Recalling $Y_f = \indic{f \in \cH_p}$ and $\xi_\sigma = \indic{\sigma \in V_p(\cH)}$, using the definition of~$\mu_j$ (see~\eqref{def:muj:basic}) it follows that 
\begin{equation}\label{eq:thm:deg:proof:muj}
\begin{split}
\sum_{f \in \cI}\E (Y_f \mid (\xi_\sigma)_{\sigma \in U}) & = \sum_{f \in \Gamma_U(\cH)}\Pr(f \in \cH_p \mid (\xi_\sigma)_{\sigma \in U})\indic{U \subseteq V_p(\cH)}\\
& \le \sum_{f \in \Gamma_U(\cH)}\Pr(f \in \cH_p \mid U \subseteq V_p(\cH)) = \sum_{f \in \Gamma_U(\cH)}p^{|f|-|U|} \le \mu_{|U|} = \mu_j. 
\end{split}
\end{equation}
Furthermore, conditional on $(\xi_\sigma)_{\sigma \in U}$, the independence assumption of Theorem~\ref{thm:C} holds by the same reasoning as in Remark~\ref{rem:C} (in the conditional space, each $Y_f$ is a function of the independent random variables $(\xi_\sigma)_{\sigma \in f \setminus U}$). 
So, applying~\eqref{eq:C} with $\mu=\mu_j$ and $\mu+t=x > e \mu_j$, we deduce the conditional inequality    
\begin{equation}\label{eq:thm:deg:Duxy}
\Pr(\cD_{U,x,y} \mid (\xi_\sigma)_{\sigma \in U}) \le \Pr(Z_C \ge x \mid (\xi_\sigma)_{\sigma \in U}) \le \left(\frac{e \mu_j}{x}\right)^{x/(ky)} .
\end{equation}
Taking expectations, by summing over all relevant $U \subseteq V(\cH)$ we thus infer 
\begin{equation}\label{eq:thm:deg:Duxy:UB}
\sum_{U \subseteq V(\cH): |U|=j}\Pr(\cD_{U,x,y}) = \sum_{U \subseteq V(\cH): |U|=j} \E \Pr(\cD_{U,x,y} \mid (\xi_\sigma)_{\sigma \in U}) \le N^j \left(\frac{e \mu_j}{x}\right)^{x/(ky)} ,
\end{equation}
and~\eqref{eq:thm:deg} follows in view of~\eqref{eq:deg:UB}.

It remains to establish~\eqref{eq:thm:sparse}. Exploiting integrality of the underlying variables, note in~\eqref{eq:thm:deg:Duxy} we can strengthen $Z_C \ge x$ to $Z_C \ge \ceil{x}$. 
In~\eqref{eq:thm:deg:Duxy}--\eqref{eq:thm:deg:Duxy:UB} we thus may replace $(e \mu_j/x)^{x/(ky)}$ by $(e \mu_j/\ceil{x})^{\ceil{x}/(ky)}$, and so~\eqref{eq:thm:sparse} follows from~\eqref{eq:lem:sparse:event} of Lemma~\ref{lem:sparse:event}, with room to spare.  
\end{proof}
The proof of Claim~\ref{cl:basic:2} 
(only used in our informal discussion) 
is very similar, and thus left to the reader.

\subsection{Extension: uniform random induced subhypergraph~$\cH_m$}\label{sec:uniform}
The proofs in Sections~\ref{sec:spars}--\ref{sec:prob} 
exploited the independence of~$\cH_p=\cH[V_p(\cH)]$ in a limited way. 
In this section we record that they extend to the uniform model~$\cH_m=\cH[V_m(\cH)]$, where the vertex subset $V_m(\cH) \subseteq V(\cH)$ of size $|V_m(\cH)|=m$ is chosen uniformly at random (this is a natural variant of~$\cH_p$ with mild dependencies). 
\begin{remark}\label{rem:prob}
Theorems~\ref{thm:sparse:iter} and~\ref{thm:prob} carry over to~$\cH_m$ after setting~$p=m/v(\cH)$ in~\eqref{def:muj:basic}. 
\end{remark}
\begin{proof}
The proof of Theorem~\ref{thm:sparse:iter} is based on (deterministic) combinatorial arguments, and after replacing~$\cH_p$ with~$\cH_m$ thus carries over word-for-word to~$\cH_m$.

Turning to Theorem~\ref{thm:prob}, using Remark~\ref{rem:C:NC} it is easy to see that the proof of~\eqref{eq:thm:e} carries over to~$\cH_m$ (with minor notational changes).

For~\eqref{eq:thm:deg} more care is needed. To avoid conditional probabilities and expectations, set $Y_{f} = \indic{f \setminus U \subseteq V_m(\cH)}$ for all $f \in \cI := \Gamma_U(\cH)$. 
Writing $\alpha \sim \beta$ if $(\alpha \cap \beta)\setminus U \neq \emptyset$, note that inequality~\eqref{eq:thm:deg:proof:overlap} readily carries over. 
It is folklore (analogous to, e.g., the proof of Theorem~15 in~\cite{JW}) that $\E Y_{f} = \Pr( f \setminus U \subseteq V_m(\cH)) \le p^{|f|-|f \cap U|}$ for~$p = m/v(\cH)$, so that $\sum_{f \in \cI}\E Y_f \le \sum_{f \in \Gamma_U(\cH)}p^{|f|-|U|} \le \mu_j$ by~\eqref{def:muj:basic}. 
Recalling the definition of~$\sim$, it is similarly folklore that the random variables $Y_{f} = \indic{f \setminus U \subseteq V_m(\cH)}$ satisfy the negative correlation condition of Remark~\ref{rem:C:NC}. 
Mimicking the argument leading to~\eqref{eq:thm:deg:Duxy}, using Theorem~\ref{thm:C} we obtain $\Pr(\cD_{U,x,y}) \le \Pr(Z_C \ge x) \le (e\mu_j/x)^{x/(ky)}$ for $\cH_m$, which by a simpler variant of~\eqref{eq:thm:deg:Duxy:UB} then establishes~\eqref{eq:thm:deg}.

As the proof of~\eqref{eq:thm:deg} carries over, for~\eqref{eq:thm:sparse} it remains to check that~\eqref{eq:lem:sparse:event} holds for~$\cH_m$. 
A close inspection of the proof of Lemma~\ref{lem:sparse:event} reveals that only the usage of the BK-inequality in~\eqref{eq:EZQ} needs to be justified.  
But, since $\cD_{U,x,y}$ is an increasing event, this application of~\eqref{eq:BKR:ind} is valid by Remark~\ref{rem:BKR:ind}, completing the proof. 
\end{proof}

\section{More general setup}\label{sec:general}
In this section we introduce our general Kim--Vu/Janson--Ruci{\'n}ski type setup, and show that the combinatorial and probabilistic arguments of Section~\ref{sec:core} carry over with somewhat minor changes. 
Readers only interested in random induced subhypergraphs~$\cH_p$ 
may wish to skip to Section~\ref{sec:UT} (see Remark~\ref{rem:general}).

\subsection{Setup}\label{sec:general:setup}
Our general setup is based on certain independence assumptions, i.e., we do not restrict ourselves to polynomials of independent random variables (and we also do not make any monotonicity assumptions). 
Given a hypergraph~$\cH$ and non-negative random variables $(Y_f)_{f \in \cH}$, for every $\cG \subseteq \cH$ we set 
\begin{equation}
\label{def:XG}
X(\cG)  = \sum_{f \in \cG} Y_f ,
\end{equation}
where our main focus\footnote{Usually we have $X=\sum_{f \in \cH} w_fI_f$ in mind, for random variables~$I_f \in \{0,1\}$ and constants~$w_f \in (0,\infty)$. All examples and applications in~\cite{KimVu2000,Vu2000,Vu2002,DL,UT,UTAP} are of this form, with $w_f=1$ (possibly after rescaling $X$ by a constant factor).}  
 is on the sum $X(\cH)$ of all the variables~$Y_f$ (sometimes~$\cH$ is also called the `supporting' or `underlying' hypergraph, see~\cite{KimVu2000,Vu2002}). 
Loosely speaking, the plan is to adapt the combinatorial arguments of Sections~\ref{sec:spars}--\ref{sec:prob} to the associated random subhypergraph 
\begin{equation}
\label{def:Hp:extended}
\cH_p = \{f \in \cH: Y_f > 0\} ,
\end{equation}
which due to $X(\cH) = X(\cH_p)$ loosely encodes all `relevant' variables (recall that $Y_f \ge 0$). 
Similar to~\cite{DL}, we shall use the following \emph{independence assumption}~(H$\ell$), where $\ell \in \NN$ is a parameter: 
{\begin{enumerate}
\itemindent 1.5em \itemsep 0.125em \parskip 0em  \partopsep=0pt \parsep 0em 
	\item[(H$\ell$)] 
	Let $(\xi_\sigma)_{\sigma \in \cA}$ be a family of independent finite random variables. 
	Suppose that there are families of subsets $\cA_U \subseteq \cA$ 
	such that (i) each non-negative random variable $Y_f$ with $f \in \cH$ is a function of the variables~$(\xi_\sigma)_{\sigma \in \cA_f}$, (ii) we have $\cA_e \cap \cA_f \subseteq \cA_{e \cap f}$ for all $e,f \in \cH$, and (iii) we have $\cA_{e} \cap \cA_{f} = \emptyset$ for all~$e,f \in \cH$ with~$|e \cap f| < \ell$. 
\end{enumerate}}\noindent%
The setup of Section~\ref{sec:basic} corresponds to the special case 
$\xi_\sigma=\indic{\sigma \in V_p(\cH)}$, $\cA_f=f$ and $Y_f = w_f \prod_{\sigma \in \cA_f} \xi_\sigma$. 
A key consequence of (H$\ell$) is 
that~$Y_e$ and~$Y_f$ are independent whenever $|e \cap f| < \ell$, since by~(i) and~(iii) then both depend on \emph{disjoint} sets of 
variables~$\xi_\sigma$. 
The `structural' assumption~(i) that each $Y_f$ depends only on 
the variables~$\xi_\sigma$ with $\sigma \in \cA_f$ is very common in applications; often $\cA_U=U$ suffices. 
The `consistency' assumption~(ii) and `independence' assumption~(iii) of the index sets~$\cA_U$ are also very natural. 
For example, in the frequent case $\cA_U=U$ we have $\cA_e \cap \cA_f = \cA_{e \cap f}$, so $\cA_e \cap \cA_f = \emptyset$ if $|e \cap f| < 1$. 
Example~\ref{ex:VE} in Section~\ref{sec:easy:examples} illustrates the case 
$\ell \neq 1$ with $\cA_U = \{f \in E(K_n): f \subseteq U\}$.

We now introduce the modified key parameters $\mu_j$, which intuitively quantify the `dependencies' among the variables~$Y_f$ (in the spirit of~\cite{DL,KimVu2000,Vu2000,Vu2002}). 
Recalling $\Gamma_U(\cH) = \{f \in \cH: U \subseteq f\}$, 
with Section~\ref{sec:basic} in mind we now define the following two crucial \emph{assumptions}~(P) and~(P$q$), where $q \in \NN$ is a parameter:
{\begin{enumerate}
\itemindent 1.5em \itemsep 0.125em \parskip 0em  \partopsep=0pt \parsep 0em 
	\item[(P)] Assume that 
$\max_{f \in \cH}|f| \le k$, $\max_{f \in \cH}\sup Y_f\le L$ and $v(\cH) \le N$. Define $\mu = \E X(\cH)$ and  
\begin{equation}
\label{def:muj:extended}
\mu_j = \max_{U \subseteq V(\cH): |U|=j} \sup \E\bigl( |\Gamma_U(\cH_p)| \; \big| \; (\xi_{\sigma})_{\sigma \in \cA_U}\bigr) , 
\end{equation}
\hspace{1.5em}where the supremum is over all values of the variables $\xi_{\sigma}$ with $\sigma \in \cA_U$. 
	\item[(P$q$)] Assume that $\Delta_q(\cH) \le D$. 
\end{enumerate}}\noindent
In view of~\eqref{def:muj:basic:alt}, property (P) is a natural extension of (P') from the basic setup of Section~\ref{sec:basic}. 
Our general setup lacks monotonicity, and so the conditioning in~\eqref{def:muj:extended} is with respect to all possible values of the~$\xi_{\sigma}$.

For the interested reader, we now briefly discuss how our setup and assumptions differ in some (usually irrelevant) minor details from the literature~\cite{DL,KimVu2000,Vu2000,Vu2002}. 
Firstly, the `normal' assumption of Vu implies $\max_{f \in \cH}\sup Y_f\le 1$ in~(P) above (see, e.g., Theorem 1.2 in~\cite{Vu2000} and Theorem 4.2 in~\cite{Vu2002}). 
Secondly, classical variants of the `maximum average effect' parameter~$\mu_j$ (see, e.g., Sections~3 in~\cite{DL} and Section~4 in~\cite{Vu2002}) are roughly defined as the maximum over all $\sup\E(\sum_{f \in \Gamma_U(\cH_p)} Y_f \mid (\xi_{\sigma})_{\sigma \in \cA_U})$ with $|U|=j$, but in most applications $\sum_{f \in \Gamma_U(\cH_p)} Y_f = \Theta(|\Gamma_U(\cH_p)|)$ holds, so the difference is usually immaterial. 
Thirdly, in~(H$\ell$) our assumptions for the index sets $\cA_U$ 
are slightly simpler than in Section~3 of~\cite{DL}. 
Finally, in contrast to~\cite{DL}, we assume 
%in~(H$\ell$) 
that the $(\xi_\sigma)_{\sigma \in \cA}$ are \emph{finite} random variables, which is very natural in combinatorial applications 
(this technicality can presumably be removed by approximation arguments, but we have not pursued this).

\subsubsection{Examples}\label{sec:easy:examples}
The above assumptions~(H$\ell$) and $(P)$ might seem a bit technical at first sight, 
and for this reason we shall below spell out three pivotal examples (see Section~3 of~\cite{DL} for more examples).  
\begin{example}[Random induced subhypergaphs]\label{ex:RIH}%
For a given $k$-uniform hypergraph~$\cH$, analogous to Section~\ref{sec:basic} we consider $X=e(\cH_p)=\sum_{f \in \cH}\indic{f \in \cH_p}$. 
Note that $\cA=\cH$, $\xi_{\sigma}=\indic{\sigma \in V_p(\cH)}$, $\cA_f=f$ and $Y_f=\prod_{\sigma \in \cA_f}\xi_{\sigma} \in \{0,1\}$ satisfy properties~(H$1$) and~(P$k$). 
In fact, for~(P) we can simplify the definition of~$\mu_j$. 
Namely, since $U \not\subseteq V_p(\cH)$ implies $f \not\in \cH_p = \cH[V_p(\cH)]$ for all $f \in \Gamma_U(\cH)$, we have 
\[
\sup \E\bigl( |\Gamma_U(\cH_p)| \; \big| \; (\xi_{\sigma})_{\sigma \in \cA_U}\bigr) = \E\bigl( |\Gamma_U(\cH_p)| \; \big| \; U \subseteq V_p(\cH) \bigr) = \sum_{f \in \Gamma_U(\cH)} \Pr\bigl(f \in \cH_p \: \big| \: U \subseteq V_p(\cH)\bigr) .
\]
As~$\cH$ is $k$-uniform, for any $f \in \Gamma_U(\cH)$ it is easy to see that $\Pr\bigl(f \in \cH_p \: \big| \: U \subseteq V_p(\cH)\bigr) = \Pr\bigl( f \setminus U \subseteq V_p(\cH)\bigr) = p^{k-|U|}$. 
Combining these observations, it follows that~\eqref{def:muj:extended} simplifies for $1 \le j \le k$ to 
\begin{equation}\label{eq:muj:extended:RIH}
\mu_j = \max_{U \subseteq V(\cH): |U|=j} |\Gamma_U(\cH)| \cdot  p^{k-j} .
\end{equation}
\end{example} 
\begin{example}[Subgraph counts in $G_{n,p}$: induced subhypergaphs approach]\label{ex:EE}%
Subgraph counts in $G_{n,p}$ can be viewed as a special case of Example~\ref{ex:RIH}, i.e., random induced subhypergaphs. 
Given a fixed subgraph~$H$ with $e=e_H$ edges, $v=v_H$ vertices and minimum degree $\delta=\delta_H \ge 1$, 
we consider the $e$-uniform hypergraph~$\cH$ with vertex set $V(\cH)=E(K_n)$, where edges correspond to copies of~$H$. 
Clearly, $k=e$ and $N = n^2$ suffice. 
Note that for the copy of $H$ counted by $Y_f$, any subset of the edges~$U \subseteq f \cap E(K_n) \subseteq V(\cH)$ is isomorphic to some subgraph $J \subseteq H$. 
So, taking all subgraphs of~$H$ with exactly $|U|=j$ edges into account, using~\eqref{eq:muj:extended:RIH} with $k=e$ and $V(\cH)=E(K_n)$
there is universal constant $B=B(H)>0$ such that for $1 \le j \le e$ we have 
\begin{equation}\label{eq:muj:extended:EE}
\mu_j \le \sum_{J \subseteq H: e_J=j} \max_{U \subseteq E(K_n): \: U \cong J} |\Gamma_U(\cH)| \cdot  p^{e-j} \le B \sum_{J \subseteq H: e_J=j} n^{v-v_J}p^{e-j} .
\end{equation}
Note that any $q=e-\delta+1 \le e$ edges already determine the vertex set, so (P$q$) holds with~$D=O(1)$. 
Finally, a minor variant of the described approach also applies to \emph{induced} subgraph counts 
(with $k=\binom{v_H}{2}$, by letting~$E(\cH)$ correspond to copies of the complete graph~$K_{v_H}$, and defining~$Y_f$ as the indicator for the event that the subgraph of $G_{n,p}$ defined by the edges in~$f$ is isomorphic to~$H$). 
\end{example} 
\begin{example}[Subgraph counts in $G_{n,p}$: vertex exposure approach]\label{ex:VE}%
Subgraph counts in $G_{n,p}$ can also be treated via a `vertex exposure' based approach. 
Given a fixed subgraph~$H$ with $e=e_H$ edges and $v=v_H$ edges, we consider the complete $v$-uniform hypergraph $\cH$ with vertex set $V(\cH)=[n]$, so $N=n$ and $k=v$.  
For $I \subseteq V(\cH)$ with $|I|=v$ the random variable~$Y_I$ counts the number of copies of~$H$ in $G_{n,p}$ that have vertex set~$I$.
Note that $0 \le Y_I \le L = O(1)$. 
Since $X=\sum_{I \in \cH}Y_I$, we take $\cA=E(K_n)$, $\xi_{\sigma}=\indic{\sigma \in V_p(\cH)}$, and $\cA_I = \{f \in E(K_n): f \subseteq I\}$. 
As $\cA_I \cap \cA_J = \cA_{I \cap J}$ is empty whenever $|I \cap J| < 2$, for $\ell=2$ properties~(H$\ell$) and~(P$k$) are satisfied. 
Conditioning on $(\xi_{\sigma})_{\sigma \in \cA_U}$ corresponds to conditioning on $G_{n,p}[U]$, so bounding~$\mu_j$ is conceptually analogous~\eqref{eq:muj:extended:EE}. 
Indeed, by similar reasoning as in Example~\ref{ex:EE}, we arrive for $1 \le j \le v$ at 
\begin{equation}\label{eq:muj:extended:VE}
\mu_j \le B \sum_{\text{induced}\:J \subseteq H: v_J=j} n^{v-j}p^{e-e_J} ,
\end{equation}
where $B=B(H)>0$.
Finally, \emph{induced} subgraph counts can clearly be treated analogously. 
\end{example}

\subsection{Adapting the arguments of Sections~\ref{sec:spars}--\ref{sec:prob}}\label{sec:prf:extended}
In this section we adapt the key results Theorem~\ref{thm:sparse:iter} and~\ref{thm:prob} from Sections~\ref{sec:spars}--\ref{sec:prob} to our more general setup. 
The crux is that the random variables $(Y_f)_{f \in \cH}$ satisfy $Y_f=Y_f(\xi_{\sigma}: \sigma \in \cA_f)$ by the independence assumption~(H$\ell$), 
so that the intersection properties of the index sets $\cA_f$ give us a handle on the dependencies. 
This allows us to adapt our combinatorial arguments to the auxiliary subhypergraph $\cH_p = \{f \in \cH: Y_f > 0\}$. 

We start with a natural analogue of Theorem~\ref{thm:sparse:iter}, which is at the heart of our arguments. 
\begin{theorem}[Combinatorial decomposition of the upper tail: general setup]%
\label{thm:sparse:iter:extended}
Given $\cH$ with $1 \le \ell \le q \le k$, assume that (H$\ell$) and (P) hold. 
Suppose that $t>0$. 
Given positive $(R_j)_{\ell \le j < q}$ and $(D_j)_{\ell \le j \le q}$, define $R_q=Q_q=D_q$ and $Q_j = \max\{S_j,D_j\}$ for $\ell \le j < q$.  
Then we have  
\begin{equation}\label{eq:thm:sparse:iter:extended}
\begin{split}
& \Pr\bigl(X(\cG) \ge \mu+t \text{ and } \Delta_{q}(\cG) \le D_q \text{ for some } \cG \subseteq \cH_p\bigr) \\
& \qquad \le \Pr\bigl(X(\cG) \ge \mu+t/2 \text{ and } \Delta_{\ell}(\cG) \le \min\{Q_\ell,R_\ell\} \text{ for some } \cG \subseteq \cH_p\bigr) \\
& \qquad \quad + \sum_{\ell \le j < q} \bigl[\Pr_{j,1} + \Pr_{j,2} +  \Pr_{j,3,\ell}\bigr] ,
\end{split}
\end{equation}
where $\Pr_{j,1}$, $\Pr_{j,2}$ and $\Pr_{j,3,\ell}$ are defined as in~\eqref{eq:thm:sparse:pj1}--\eqref{eq:thm:sparse:pj3}. 
\end{theorem}
\noindent 
Recalling $X(\cG)=\sum_{f \in \cG}Y_f$ and $\cH_p = \{f \in \cH: Y_f > 0\}$, 
the deterministic proof of Theorem~\ref{thm:sparse:iter} carries over to Theorem~\ref{thm:sparse:iter:extended} with minor obvious changes (inequality~\eqref{eq:thm:sparse:iter:extended} is trivial if $q=\ell$; for $q>\ell$ it suffices to 
construct $\cG = \cJ_q \supseteq \cdots \supseteq \cJ_{\ell}$, with indices of form $\ell \le i,j \le q$ in~\eqref{eq:seq:degij}); we omit the routine details.

Next we state an analogue of Lemma~\ref{lem:sparse:event} for the `sparsification' event $\cE_{j,\ell}(x,r,y,z)$ from Section~\ref{sec:spars}. 
\begin{lemma}[Auxiliary result for the sparsification event: general setup]%
\label{lem:sparse:event:extended}
Given $\cH$ with $1 \le \ell \le k$, assume that (H$\ell$) and $\max_{f \in \cH}|f| \le k$ hold. 
Then for all $x,r,y,z>0$ and $\ell \le j < k$ we have 
\begin{equation}\label{eq:lem:sparse:event:extended}
\Pr\bigl(\neg \cE_{j,\ell}(x,r,y,z)\bigr) \le \biggl(\sum_{U \subseteq V(\cH): |U|=j}\Pr(\cD_{U,x,y})\biggr)^{\ceilL{r/\bigl(\binom{k}{\ell}\ceil{x}z\bigr)}} . 
\end{equation}
\end{lemma}
\begin{remark}\label{rem:sparse:event:extended}
Inequality~\eqref{eq:lem:sparse:event:extended} remains valid after dividing the right hand side by $\ceil{r/(\binom{k}{\ell}\ceil{x}z)}!$. 
\end{remark}
\noindent 
For the proof of Lemma~\ref{lem:sparse:event:extended} we adapt the definition of $M_{j,x,y}(\cG)$ used for Lemma~\ref{lem:sparse:event}. 
Intuitively, the idea is to replace `vertex disjoint' by `depending on disjoint sets of variables'. 
Namely, here we define $M_{j,x,y}(\cG)$ as the size of the largest collection $\cM \subseteq \cS_{j,x,y}(\cG)$ of $(j,x,y)$-stars in~$\cG$ satisfying the following property for all distinct $(U,\cK_{U}), (W,\cK_{W}) \in \cM$: we have $|e \cap f| < \ell$ for all $e \in \cK_{U}$ and $f \in \cK_{W}$. 
The point will be (i)~that each $Y_f$ is a function of the variables $(\xi_\sigma)_{\sigma \in \cA_f}$, and (ii)~that $|e \cap f| < \ell$ implies $\cA_{e} \cap \cA_{f} = \emptyset$ by~(H$\ell$). 
\begin{proof}[Proof of Lemma~\ref{lem:sparse:event:extended}]
Using the above definition of $M_{j,x,y}(\cG)$, we shall adapt the proof of Lemma~\ref{lem:sparse:event}. 
Let $\tilde{r}=r/\bigl(\binom{k}{\ell}\ceil{x}z\bigr)$ and $R=\ceil{\tilde{r}}$. We first assume that $M_{j,x,y}(\cH_p) \le \tilde{r}$ holds, and claim that this implies the occurrence of $\cE_{j,\ell}(x,r,y,z)$. 
Fix $\cG \subseteq \cH_p$ with $\Delta_{j+1}(\cG) \le y$ and $\Delta_{\ell}(\cG) \le z$, and let $\cM \subseteq \cS_{j,x,y}(\cG)$ attain the maximum in the definition of $M_{j,x,y}(\cG)$. 
We remove all edges~$f \in \cG$ which `overlap' some star $(U,\cK_U) \in \cM$, where overlap means that $|f \cap g| \ge \ell$ for some edge $g \in \cK_U$. 
We denote the resulting subhypergraph by~$\cJ \subseteq \cG$. 
Recalling $\Delta_{j+1}(\cJ) \le \Delta_{j+1}(\cG) \le y$, by maximality of $\cM$ we infer $\Delta_{j}(\cJ) \le \ceil{x}-1 < x$. 
Similar to~\eqref{eq:removalG}, using $|\cM| = M_{j,x,y}(\cG) \le M_{j,x,y}(\cH_p) \le \tilde{r}$ and $\Delta_{\ell}(\cG) \le z$ it is easy to see that we removed at most  
\begin{equation}\label{eq:removalG:extended}
e(\cG \setminus \cJ) \le |\cM| \cdot \ceil{x} \cdot \biggl[\max_{f \in \cG} \binom{|f|}{\ell}\biggr] \cdot \Delta_{\ell}(\cG) \le \tilde{r} \cdot \ceil{x} \binom{k}{\ell} z = r
\end{equation}
edges. It follows that $M_{j,x,y}(\cH_p) \le \tilde{r}$ implies $\cE_{j,\ell}(x,r,y,z)$, as claimed.  

For~\eqref{eq:lem:sparse:event:extended} it remains to estimate $\Pr(M_{j,x,y}(\cH_p) > \tilde{r})$. 
Suppose that $M_{j,x,y}(\cH_p) > \tilde{r}$ occurs. If $\cM \subseteq \cS_{j,x,y}(\cH_p)$ attains the maximum in the definition of $M_{j,x,y}(\cH_p)$, then we know (i)~that $|\cM| \ge \ceil{\tilde{r}}=R$ holds, and (ii)~that $\bigcap_{(U,K_U) \in \cM}\cD_{U,x,y}$ occurs. 
In the following we argue that these events $\cD_{U,x,y}$ `occur disjointly' in the sense of Section~\ref{sec:BK}. 
For each $(U,K_U) \in \cM$, note that the variables indexed by 
\[
V(\cK_U) = \bigcup_{f \in \cK_{U}} \cA_f
\]
guarantee the occurrence of $\cD_{U,x,y}$. 
The crux is now that for all distinct $(U,\cK_{U}), (W,\cK_{W}) \in \cM$, by~(iii)~of~(H$\ell$) we have $\cA_{e} \cap \cA_{f}=\emptyset$ for all $e \in \cK_u$ and $f \in \cK_W$ (since $|e \cap f|<\ell$), so 
\begin{equation}\label{eq:var:disjoint}
V(\cK_{U}) \cap V(\cK_{W}) = \bigcup_{e \in \cK_{U}}\bigcup_{f \in \cK_{W}} (\cA_{e} \cap \cA_{f}) = \emptyset .
\end{equation} 
It follows that $\boxdot_{(U,\cK_U) \in \cM}\cD_{U,x,y}$ occurs (since the disjoint sets of variables indexed by $V(\cK_{U})$ guarantee the occurrence of each $\cD_{U,x,y}$).
Next we claim that all the corresponding sets $U$ are distinct. 
To see this, note that for distinct $(U,\cK_{U}), (W,\cK_{W}) \in \cM$ we have $\ell > |e \cap f| \ge |U \cap W|$ by definition of $\cM$, which due to $|U| = |W| = j \ge \ell$ implies $U \neq W$. 
To sum up, $M_{j,x,y}(\cH_p) > \tilde{r}$ implies $Z_R \ge R!$, where $Z_R$ is defined as in~\eqref{def:ZR}. 
The arguments of~\eqref{eq:MZQ} and~\eqref{eq:EZQ} now carry over unchanged, completing the proof of~\eqref{eq:lem:sparse:event:extended}.  
\end{proof}

Finally, we state a natural analogue of Theorem~\ref{thm:prob}, which contains our core probabilistic estimates 
(inequalities~\eqref{eq:thm:e:extended}--\eqref{eq:thm:sparse:extended} allow us 
to bound the right hand side of~\eqref{eq:thm:sparse:iter:extended} from Theorem~\ref{thm:sparse:iter:extended}). 
\begin{theorem}[Probabilistic upper tail estimates: general setup]%
\label{thm:prob:extended}
Given $\cH$ with $1 \le \ell \le k$, assume that (H$\ell$) and (P) hold. 
Set $\varphi(x)=(1+x)\log(1+x)-x$.  
Then for all $x,r,y,z,t>0$ and $\ell \le j < k$ we have 
\begin{align}
\label{eq:thm:e:extended}
\Pr\bigl(X(\cG) \ge \mu+t/2 \text{ and }\Delta_{\ell}(\cG) \le y \text{ for some } \cG \subseteq \cH_p\bigr) & \le \exp\left(-\frac{\varphi(t/\mu)\mu}{4L\binom{k}{\ell}y} \right) , \\
\label{eq:thm:deg:extended}
\Pr\bigl(\Delta_{j}(\cG) \ge x \text{ and } \Delta_{j+1}(\cG) \le y \text{ for some } \cG \subseteq \cH_p\bigr) & \le N^{j} \left(\frac{e\mu_j}{x}\right)^{x/(ky)} ,\\
\label{eq:thm:sparse:extended}
\Pr\bigl(\neg \cE_{j,\ell}(x,r,y,z)\bigr) & \le \left(N^j \left(\frac{e\mu_j}{\ceil{x}}\right)^{\ceil{x}/(4ky)}\right)^{\ceilL{r/\bigl(\binom{k}{\ell}\ceil{x}z\bigr)}}.
\end{align}
\end{theorem}
\noindent 
The proof is based on a minor modification of the proof of Theorem~\ref{thm:prob}. 
As we shall see, our main task is to adapt the definitions of the dependency relations~$\sim$. 
To this end recall (i) that each $Y_f$ is a function of the independent variables $(\xi_\sigma)_{\sigma \in \cA_f}$, and (ii) that~(H$\ell$) implies $\cA_{e} \cap \cA_{f} = \emptyset$ whenever $|e \cap f| < \ell$.  
\begin{proof}[Proof of Theorem~\ref{thm:prob:extended}]
For~\eqref{eq:thm:e:extended}, note that $\sum_{f \in \cH} \E Y_f =\E X(\cH) = \mu$. 
We define $\alpha \sim \beta$ if $|\alpha \cap \beta| \ge \ell$. In view of properties~(i) and~(ii) discussed above, the independence assumption of Theorem~\ref{thm:C} holds by analogous reasoning as in Remark~\ref{rem:C}. 
Furthermore, for any $f \in \cG \subseteq \cH$ with $\Delta_{\ell}(\cG) \le y$ we have 
\begin{equation*}
\sum_{e \in \cG: e \sim f} Y_e  
\le \bigl(\max_{e \in \cG} \sup Y_e \bigr) \cdot \sum_{e \in \cG: |e \cap f| \ge \ell}\indic{f \in \cG} 
\le L \cdot \sum_{U \subseteq f: |U|=\ell} |\Gamma_U(\cG)| \le L \cdot \binom{|f|}{\ell} \cdot \Delta_\ell(\cG) \le L \binom{k}{\ell} y. 
\end{equation*}
Setting $C=L \binom{k}{\ell} y$, the remaining proof of~\eqref{eq:thm:e} readily carries over to~\eqref{eq:thm:e:extended} with obvious notational changes. %, i.e., after replacing $w(\cG)$ by $X(\cG)$. 

Next we turn to~\eqref{eq:thm:deg:extended}, which is again based on~\eqref{eq:deg:UB}. 
As before, we may assume that $x > e\mu_j$ and $N \ge 1$ (otherwise the claim is trivial). 
Furthermore, given $U \subseteq V(\cH)$ with $|U|=j$, we set $\cI = \Gamma_U(\cH)$. 
With the random variables $\bigl(\indic{Y_f > 0}\bigr)_{f \in \cI}$ in mind, define $\alpha \sim \beta$ if $(\alpha \cap \beta)\setminus U \neq \emptyset$.
Note that, for any $f \in \cK \subseteq \cI$ with $\Delta_{|U|+1}(\cK) \le y$, analogous to~\eqref{eq:thm:deg:proof:overlap} we have $\sum_{e \in \cK: e \sim f} \indic{Y_f > 0} \le |f \setminus U| \cdot \Delta_{|U|+1}(\cK) \le ky$. 
Furthermore, by definition of $\cI=\Gamma_U(\cH)$, $\cH_p = \{f \in \cH: Y_f > 0\}$ and $\mu_j$ (see~\eqref{def:muj:extended}) we obtain 
\[
\sum_{f \in \cI}\E \bigl(\indic{Y_f > 0} \mid (\xi_\sigma)_{\sigma \in \cA_U}\bigr) =  
\E\bigl( |\Gamma_U(\cH_p)| \; \big| \; (\xi_{\sigma})_{\sigma \in \cA_U}\bigr) \le \mu_{|U|} = \mu_j. 
\]
Note that, conditional on $(\xi_\sigma)_{\sigma \in \cA_U}$, each $\indic{Y_f>0}$ is now a function of the independent random variables $(\xi_\sigma)_{\sigma \in \cA_f \setminus \cA_U}$. 
Furthermore, for all $e,f \in \cI= \{g \in \cH: U \subseteq g\}$ we see that $(e \cap f) \setminus U = \emptyset$ implies $e \cap f = U$, so that~(ii) of~(H$\ell$) yields $\cA_e \cap \cA_f \subseteq \cA_{e \cap f}=\cA_U$. 
For all $e,f \in \cI$ we thus infer that $e \not\sim f$ implies 
\begin{equation*}\label{eq:thm:deg:extended:proof:vars:2}
(\cA_e \setminus \cA_U) \cap (\cA_f \setminus \cA_U) = (\cA_e \cap \cA_f) \setminus \cA_U \subseteq \cA_U \setminus \cA_U = \emptyset .
\end{equation*}
Conditional on $(\xi_\sigma)_{\sigma \in \cA_U}$, it follows (by the reasoning of Remark~\ref{rem:C}) that the independence assumption of Theorem~\ref{thm:C} holds for the 
variables $\bigl(\indic{Y_f > 0}\bigr)_{f \in \cI}$. 
The remaining proof of~\eqref{eq:thm:deg} readily carries over to~\eqref{eq:thm:deg:extended}.

Finally, for~\eqref{eq:thm:sparse:extended} we recall that~\eqref{eq:thm:sparse} is based on Lemma~\ref{lem:sparse:event} and the argument leading to~\eqref{eq:thm:deg}. 
In view of Lemma~\ref{lem:sparse:event:extended} and the above proof of~\eqref{eq:thm:deg:extended}, the same line of reasoning carries over, establishing~\eqref{eq:thm:sparse:extended}. 
\end{proof}

\subsection{Adapting Section~\ref{sec:uniform}: vertex exposure approach for~$\cH_m$}\label{sec:uniform:Hm}
In this section we partially adapt our arguments 
to the uniform random induced subhypergraph~$\hH_m=\hH[V_m(\hH)]$. 
Generalizing the `vertex exposure' approach of Example~\ref{ex:VE}, 
we rely on the following assumption. 
{\begin{enumerate}
\itemindent 1.5em \itemsep 0.125em \parskip 0em  \partopsep=0pt \parsep 0em 
	\item[(H$\ell$P)] 
	Suppose that $\cH$, $\hI$ and $\hH$ are hypergraphs with $V(\cH)=V(\hI)$, $V(\hH)=\{h \in \hI\}$ and $\min_{h \in \hI} |h| \ge \ell$. 
	Defining $\cA_U = \{h \in \hI: h \subseteq U\}$ for all $U \subseteq V(\hI)$, assume that $\hH = \bigcup_{f \in \cH} \hH[A_f]$ is a disjoint union of induced subhypergraphs. 
	Suppose that $(w_g)_{g \in \hH}$ are non-negative weights. For all $f \in \cH$, let   
\begin{equation}
\label{def:Yf:Hm}
	Y_f = \sum_{g \in \hH[\cA_f]}w_g \indic{g \in \hH_m} .
\end{equation} 
Assume that 	$\max_{f \in \cH}|f| \le k$, $\max_{f \in \cH}Y_f\le L$ and $v(\cH) \le N$. Define $\mu=\E X(\cH)$, $p=m/v(\hH)$, and  
\begin{equation}
\label{def:muj:basic:Hm}
\mu_j = \max_{U \subseteq V(\hI): |U|=j} \sum_{f \in \Gamma_U(\cH)}\sum_{g \in \hH[\cA_f]}p^{|g|-|g \cap \cA_U|} . 
\end{equation}
\end{enumerate}}%\noindent%
\begin{example} 
\label{ex:VE:Gnm}
Using the `vertex exposure' setup discussed in Example~\ref{ex:VE}, subgraph counts in $G_{n,m}$ satisfy~(H$\ell$P) with $\ell=2$ and $k=v_H$ (by setting $\hI=K_n$, and defining~$\hH$ as the hypergraph~$\cH$ of Example~\ref{ex:EE}). 
In~\eqref{def:muj:basic:Hm} the modified parameter $\mu_j$ is again bounded from above by the right hand side of~\eqref{eq:muj:extended:VE}. 
\end{example}
\begin{remark}\label{rem:prob:Hm}
Theorems~\ref{thm:sparse:iter:extended} and~\ref{thm:prob:extended} remain valid after replacing the assumptions~(H$\ell$),(P) with~(H$\ell$P). 
\end{remark}
\begin{proof}% 
With the ideas of Remark~\ref{rem:prob} in mind, we only sketch the key modifications for~\eqref{eq:thm:e:extended}--\eqref{eq:thm:deg:extended} of Theorem~\ref{thm:prob:extended}.

For~\eqref{eq:thm:e:extended} it suffices to verify the negative correlation condition of Remark~\ref{rem:C:NC}, writing $\alpha \sim \beta$ if $|\alpha \cap \beta| \ge \ell$. 
Using~\eqref{def:Yf:Hm} and the negative correlation properties of $\hH_m$ (see Remark~\ref{rem:C:NC}), it is not hard to check that 
\begin{equation}\label{eq:rem:prob:Hm:cor}
\E\bigl(\prod_{i \in [s]} Y_{\alpha_i}\bigr) = \sum_{g_1 \in \hH[\cA_{\alpha_1}]} \cdots \sum_{g_s \in \hH[\cA_{\alpha_s}]} \E\bigl(\prod_{i \in [s]} w_{g_i} \indic{g_i \in \hH_m}\bigr) \le \prod_{i \in [s]} \E Y_{\alpha_i} ,
\end{equation}
and so the proof of~\eqref{eq:thm:e:extended} carries over (above we used that $\alpha_i \not\sim \alpha_j$ implies $\hH[\cA_{\alpha_i}] \cap \hH[\cA_{\alpha_j}] = \emptyset$).

For~\eqref{eq:thm:deg:extended} we define $\alpha \sim \beta$ if $(\alpha \cap \beta)\setminus U \neq \emptyset$, and replace $\bigl(\indic{Y_f > 0}\bigr)_{f \in \cI}$ by $\bigl(\indic{\fY_f}\bigr)_{f \in \cI}$, where $\fY_f$ denotes the event that $g \setminus \cA_U \subseteq V_m(\hH)$ for some $g \in \hH[\cA_f]$.  
Let $\lambda_{f} = \sum_{g \in \hH[\cA_{f}]} \Pr(g \setminus \cA_U \subseteq V_m(\hH))$. 
It is folklore that $\Pr(g \setminus \cA_U \subseteq V_m(\hH)) \le p^{|g|-|g \cap \cA_U|}$ (see Remark~\ref{rem:prob}), so 
$\cI=\Gamma_U(\cH)$ and~\eqref{def:muj:basic:Hm} yield $\sum_{f \in \cI} \lambda_f \le \mu_{|U|} = \mu_j$. 
Since $\indic{\fY_f} \le \sum_{g \in \hH[\cA_f]} \indic{g \setminus \cA_U \subseteq V_m(\hH)}$, analogous to~\eqref{eq:rem:prob:Hm:cor} we infer $\E(\prod_{i \in [s]} \indic{\fY_{\alpha_i}}) \le \prod_{i \in [s]} \lambda_{\alpha_i}$, establishing the correlation condition of Remark~\ref{rem:C:NC:2}. 
Mimicking Remark~\ref{rem:prob}, the proof of~\eqref{eq:thm:deg} then carries over to~\eqref{eq:thm:deg:extended}. 
\end{proof}

\section{Corollaries: upper tail inequalities}\label{sec:UT} 
The main results of Sections~\ref{sec:core}--\ref{sec:general} are Theorems~\ref{thm:sparse:iter},\ref{thm:sparse:iter:extended} of form $\Pr(X \ge (1+\eps)\E X) \le \sum_i \Pr(\neg\cE_i)$ and Theorems~\ref{thm:prob},\ref{thm:prob:extended} of form $\Pr(\neg\cE_i) \le \exp(-\Psi_i)$. 
In this section we derive upper tail inequalities that are convenient for the applications of Section~\ref{sec:applications}, 
and briefly compare some of our more general estimates with the literature.

\begin{remark}[Random induced subhypergraph setup]%
\label{rem:general}
The results in Sections~\ref{sec:easy}--\ref{sec:extended} are stated for the general setup of Section~\ref{sec:general:setup}. 
But, with minor changes, they remain valid in the simpler random induced subhypergraph setup of Section~\ref{sec:basic}. 
Indeed, setting~$\ell=1$ and replacing the assumptions~(H$\ell$),(P) with~(P'), all results carry over to $\cH_p$ by defining~$X(\cJ) :=w(\cJ)$.  
After setting $p=m/v(\cH)$ in~\eqref{def:muj:basic}, these results for~$\cH_p$ then also carry over to the uniform variant~$\cH_m$ defined in Section~\ref{sec:uniform}. 
Finally, after replacing the assumptions (H$\ell$),(P) with~(H$\ell$P), all results in Sections~\ref{sec:easy}--\ref{sec:extended} 
also remain valid in the setup of Section~\ref{sec:uniform:Hm}. 
\end{remark}
\noindent 
Henceforth, we tacitly set $\varphi(x)=(1+x)\log(1+x)-x$ for brevity (as in Theorems~\ref{thm:C},~\ref{thm:prob} and~\ref{thm:prob:extended}).

\subsection{Easy-to-apply tail inequalities}\label{sec:easy}
In this section we state some simplified upper tail inequalities that 
suffice for all the applications in Section~\ref{sec:applications} 
(we have not optimized the usually irrelevant constants); 
the proofs are deferred to Section~\ref{sec:prob:thm}.

On first reading of the following upper tail inequality for $X(\cH) = \sum_{f \in \cH} Y_f$, the reader may wish to set~$\ell=1$ and~$q=k$, 
so that~\eqref{eq:thm:pre:extendedp} is of form $\Pr(X(\cH) \ge 2\mu) \le \exp(-d \min\{\mu,\mu^{1/k}\log(e/\pi)\})$.     
Here our main novelty is the $\log(e/\pi)$ term: it allows us to gain an extra logarithmic factor if $\pi\in \{N^{-1},p\}$, 
which yields best possible tail estimates in the applications of Section~\ref{sec:applications:RISH}.
We think of~\eqref{eq:thm:pre:extendedp:asmpt1} as a `balancedness' condition, and mainly have parameters of form $\pi\in \{1,N^{-1},p\}$ in mind. 
In fact, for $\pi\in \{N^{-1},p\}$ the technical assumption~\eqref{eq:thm:pre:extendedp:asmpt2} usually holds automatically for small~$\tau$ (see Remark~\ref{rem:pre:extendedp} and the proof of Theorem~\ref{thm:randinduced}).  
%In~\eqref{eq:thm:pre:extendedp} our main novelty is the $\log(e/\pi)$ term, which is key for the applications in Section~\ref{sec:applications:RISH} (to obtain best possible tail estimates).
%
\begin{theorem}[Easy-to-apply upper tail inequality]
\label{thm:pre:extendedp}
Given $\cH$ with $1 \le \ell \le q \le k$, assume that (H$\ell$), (P) and (P$q$) hold. 
If there are constants $A,\alpha,\tau > 0$ and a parameter $\pi \in (0,1]$ such that 
\begin{gather}
\label{eq:thm:pre:extendedp:asmpt1}
\max_{\ell \le j < q}\frac{\mu_j}{\max\{\mu^{(q-j)/(q-\ell+1)},1\}}  \le A \pi^{\alpha}, \\
\label{eq:thm:pre:extendedp:asmpt2}
A\mu^{1/(q-\ell+1)} \ge \indic{\pi > N^{-\tau}} \log N ,
\end{gather}
then for $\eps> 0$ we have 
\begin{equation}\label{eq:thm:pre:extendedp}
\begin{split}
\Pr(X(\cH) \ge (1+\eps)\mu) &\le (1+bN^{-\ell}) \exp\left(-c\min\Bigl\{\varphi(\eps) \mu, \: \min\bigl\{\eps^2,1\bigr\} \mu^{1/(q-\ell+1)}\log(e/\pi)\Bigr\}\right) \\
&\le (1+bN^{-\ell}) \exp\left(-d\min\bigl\{\eps^2,1\bigr\} \min\bigl\{\mu, \:\mu^{1/(q-\ell+1)}\log(e/\pi)\bigr\}\right) ,
\end{split}
\end{equation}
where $b=3q$, $c=c(\ell,q,k,L,D,A,\alpha,\tau)>0$ and $d=c/3$. 
\end{theorem}
\begin{remark}\label{rem:pre:extendedp}
If $\pi=N^{-1}$, then~\eqref{eq:thm:pre:extendedp:asmpt2} is trivially satisfied for~$\tau=1/2$, and $\log(e/\pi) \ge \log N$ holds in~\eqref{eq:thm:pre:extendedp}. 
\end{remark}
\noindent 
Simple applications of the inductive approaches~\cite{KimVu2000,DL,Vu2002} often implicitly assume~\eqref{eq:thm:pre:extendedp:asmpt1} with $\pi=1$, and replace~\eqref{eq:thm:pre:extendedp:asmpt2} by the stronger assumption $\min\{\eps^2,1\} \mu^{1/(q-\ell+1)} = \omega(\log N)$, say (see, e.g., the proof of Corollary~6.3 in~\cite{Vu2002} or Theorem~2.1 in~\cite{Vu2001}).  
Their conclusion is then of the form $\Pr(X(\cH) \ge (1+\eps)\mu) \le \exp(-a\min\{\eps^2,1\} \mu^{1/(q-\ell+1)})$, 
where $\mu^{1/(q-\ell+1)} = \min\bigl\{\mu, \mu^{1/(q-\ell+1)}\log(e/\pi)\bigr\}$ holds by assumption. 
In other words, our inequality~\eqref{eq:thm:pre:extendedp} yields an extra logarithmic factor when $\pi \in \{N^{-1},p\}$ in~\eqref{eq:thm:pre:extendedp:asmpt1}. 
To illustrate this, for subgraph counts in~$G_{n,p}$ the setup of Example~\ref{ex:EE} (with $\ell=1$, $q=k=e$ and $N=n^2$) naturally yields 
\begin{equation*}
\max_{\ell \le j < q}\frac{\mu_j}{\mu^{(q-j)/(q-\ell+1)}} \le \max_{1 \le j < e} \frac{O(\sum_{J \subseteq H: e_J=j}n^{v-v_J}p^{e-j})}{\Theta((n^vp^e)^{(e-j)/e})} \le O(\sum_{J \subseteq H: 1 \le e_J < e}n^{e_Jv/e-v_J}) ,
\end{equation*}
which is well-known to be $O(n^{-\beta})$ for so-called `strictly balanced' graphs and $O(1)$ for `balanced' graphs (the details are deferred to~\eqref{eq:density:JH:0} and~\eqref{eq:density:muj} in Section~\ref{sec:applications:RG}; see also Section~6.3 in~\cite{Vu2002}).

The next upper tail result assumes that all the parameters $\mu_j$ are decaying polynomially in~$N$, which typically requires that $\mu=\E X(\cH)$ is small (as $v(\cH) \le N$). 
On first reading of Theorem~\ref{thm:pre:extended:small} the reader may wish to set $\ell=1$, $q=k$ and $K=1$, so that~\eqref{eq:thm:pre:extended:small} is of form 
$\Pr(X(\cH) \ge \mu + t) \le \exp(-a \min\{t^2/\mu,t^{1/k}\log N\})$ when $t \in [1,\mu]$. 
Here our main novelty is the $t^{1/k} \log N$ term, which is key for the applications in Section~\ref{sec:applications:RG:SG}. 
\begin{theorem}[Easy-to-apply upper tail inequality: the small expectations case]
\label{thm:pre:extended:small}
Given $\cH$ with $1 \le \ell \le q \le k$, assume that (H$\ell$), (P) and (P$q$) hold. 
If there are constants $A,\alpha>0$ such that 
\begin{align}
\label{eq:thm:pre:basic:small:asmpt}
\max_{\ell \le j < q} \mu_j & \le A N^{-\alpha} , 
\end{align}
then for $t,K> 0$ we have 
\begin{equation}\label{eq:thm:pre:extended:small}
\begin{split}
\Pr(X(\cH) \ge \mu+t) & \le (1+bN^{-q}) \exp\left(-\min\Bigl\{c\varphi(t/\mu) \mu, \; \max\bigl\{ct^{1/(q-\ell+1)},K\bigr\}\log N\Bigr\}\right) \\
& \le (1+bN^{-q}) \exp\left(-\min\Bigl\{dt^2/\mu, \; dt, \; \max\bigl\{ct^{1/(q-\ell+1)},K\bigr\}\log N\Bigr\}\right) ,
\end{split}
\end{equation}
where $b=2q$, $c=c(\ell,q,k,L,D,A,\alpha,K)>0$ and $d=c/3$. 
\end{theorem}
\noindent 
The inductive approaches~\cite{Vu2000,DL} yield variants of~\eqref{eq:thm:pre:extended:small} 
where $\max\bigl\{ct^{1/(q-\ell+1)},K\bigr\}$ is qualitatively replaced by~$K$ (see, e.g., Corollary~4.10 in~\cite{DL}). 
For~$K$ large enough this gives bounds of the form $\Pr(X(\cH) \ge (1+\eps)\mu) \le N^{-\beta}$ for $\mu \ge C(\eps,d,\beta) \log n$, and 
$\Pr(X(\cH) \ge (1+\eps)\mu) \le \exp(-d \eps^2 \mu)$ for $\mu \le \log n$ and $\eps \le 1$, say (see, e.g., Corollaries~4.11--4.12 in~\cite{DL}). 
To illustrate assumption~\eqref{eq:thm:pre:basic:small:asmpt}, for subgraph counts in~$G_{n,p}$ with $p = O(n^{-v/e + \sigma})$,  
the setup of Example~\ref{ex:EE} (with $\ell=1$, $q=k=e$ and $N=n^2$) yields $\mu=O(n^{e \sigma})$ and  
\begin{equation*}
\max_{\ell \le j < q}\mu_j \le O(\sum_{J \subseteq H: 1 \le e_J < e}n^{v-v_J}p^{e-e_J}) \le O(\sum_{J \subseteq H: 1 \le e_J < e}n^{e_Jv/e-v_J + \sigma(e-e_J)}) , 
\end{equation*}
which for `strictly balanced' graphs is well-known to be $O(n^{-\sigma/2})$ for sufficiently small $\sigma>0$ 
(the details are deferred to~\eqref{eq:density:JH:0} and~\eqref{eq:density:muj:2} in Section~\ref{sec:applications:RG}; see also Claim~6.2 in~\cite{Vu2002}).

\subsection{More general tail inequalities}\label{sec:extended}
In this section we state some more general upper tail inequalities which 
(i)~mimic the heuristic discussion of Section~\ref{sec:form},
and (ii)~are easier to compare with the work of Kim--Vu/Janson--Ruci{\'n}ski~\cite{KimVu2000,Vu2000,Vu2002,DL};  
the proofs are deferred to Section~\ref{sec:prob:thm}. 
Readers primarily interested in applications may proceed to Section~\ref{sec:applications}.

We start with a rigorous analogue of the basic upper tail inequality~\eqref{heur:ind} from Section~\ref{sec:form}, which is inspired by very similar classical results 
for the special case $\cG=\cH_p$ with $\Delta_q(\cH) \le D$ (see, e.g., Theorem~3.10 in~\cite{DL} and Theorem~4.2 in~\cite{Vu2002}). 
In applications convenient choices of the parameters $(R_j)_{\ell \le j < q}$ and~$D$ are often of form $D=\Theta(1)$, $R_j = \lambda^{q-j}D$ and $\lambda = B\max\{\mu^{1/(q-\ell+1)},1\}$, so that in~\eqref{eq:cl:basic} we have $\min\{\mu/R_{\ell} = \Theta(\lambda)$ and $R_j/R_{j+1} = \lambda$ when $\mu \ge 1$ (see, e.g., the proof of Corollary~6.3 in~\cite{Vu2002} or Theorem~2.1 in~\cite{Vu2001}). 
\begin{claim}[Basic upper tail inequality]%
\label{cl:basic}
Given $\cH$ with $1 \le \ell \le q \le k$, assume that (H$\ell$) and (P) hold. 
Suppose that $t>0$. Given positive $(R_j)_{\ell \le j < q}$ and $D$, let $R_q=D$. 
If inequality 
\begin{equation}\label{eq:cl:basic:asmpt}
\left(\frac{e \mu_j}{R_j}\right)^{R_j/R_{j+1}}  \le N^{-4kj} 
\end{equation}
holds for all $\ell \le j < q$, then there are $a,b>0$ (depending only on $\ell,k,L$) such that 
\begin{equation}\label{eq:cl:basic}
\begin{split}
& \Pr(X(\cG) \ge \mu+t \text{ and } \Delta_q(\cG) \le D \text{ for some }  \cG \subseteq \cH_p) \\
& \qquad \le \exp\left(-\frac{a\varphi(t/\mu)\mu}{R_\ell}\right) + \sum_{\ell \le j < q}N^{-j}\left(\frac{e \mu_j}{R_j}\right)^{bR_j/R_{j+1}} . 
\end{split}
\end{equation}
\end{claim}
\noindent 
To familiarize the reader with the form of assumption~\eqref{eq:cl:basic:asmpt} and inequality~\eqref{eq:cl:basic}, it is instructive to briefly relate them to work of Kim and Vu~\cite{KimVu2000,Vu2001,Vu2002}.
Theorem~4.2 in~\cite{Vu2002} qualitatively sets $t=\sqrt{\lambda \mu R_\ell}$, and (in our notation)
its parametrization assumes roughly $\Delta_q(\cH) \le D = R_q$, $\mu/R_{\ell} \ge \lambda = \omega(\log N)$, as well as $R_j \ge 2e \mu_j$ and $R_{j}/R_{j+1} \ge \lambda$ for all $\ell \le j < q$, say.  
In this case $(e\mu_j/R_j)^{R_j/R_{j+1}} \le 2^{-\lambda} = N^{-\omega(1)}$ follows, so assumption~\eqref{eq:cl:basic:asmpt} holds. 
We also have $t= \mu \sqrt{\lambda R_\ell/\mu} \le \mu$, so that Remark~\ref{rem:varphi} yields $\varphi(t/\mu)\mu/R_{\ell} \ge t^2/(3\mu R_{\ell}) = \lambda/3$, say. 
Recalling $\Delta_q(\cH) \le D$, for suitable $C=C(q)$ and $c=c(a,b)$ it follows that~\eqref{eq:cl:basic} yields 
\begin{equation}\label{eq:cl:basic:heur}
\Pr(X(\cH_p) \ge \mu+t) \le \exp\bigl(-a\lambda/3\bigr) + \indic{q > \ell} q N^{-\ell} 2^{-b\lambda} \le C \exp\bigl(-c\lambda\bigr) ,
\end{equation}
which is of similar form as~\eqref{heur:ind:2} or Theorem~4.2~\cite{Vu2002}.

We now state our improved variant\footnote{Note that by setting $s=1$ and $D_j=R_j$ we have $Q_j=R_j$ in~\eqref{eq:thm:basic:Qj}, so the indicators in~\eqref{eq:thm:basic:asmpt}--\eqref{eq:thm:extended} are zero and Theorem~\ref{thm:extended} %essentially 
recovers Claim~\ref{cl:basic} up to irrelevant constant factors.} of Claim~\ref{cl:basic}, which corresponds to a rigorous analogue of the upper tail inequality~\eqref{heur:iter} from Section~\ref{sec:form}. 
Convenient choices of the parameters $(R_j)_{\ell \le j < q}$ and $(D_j)_{\ell \le j \le q}$ 
are often of form $D_j = B^{q-j}D_q = \Theta(1)$, $R_j = \lambda^{q-j}D_q$ and $\lambda = B\max\{\mu^{1/(q-\ell+1)},1\}$, so that in~\eqref{eq:thm:extended} we have $R_j/R_{j+1} = \lambda$ and $t/R_{\ell} = \Theta(\lambda)$ when $t = \Theta(\mu)$ and $\mu \ge 1$.  
One key novelty of~\eqref{eq:thm:extended} is the $\mu/Q_{\ell} = \min\{\mu s/R_\ell,\mu/D_\ell\}$ term, which intuitively allows us to sharpen inequality~\eqref{eq:cl:basic} whenever $R_j = \omega(\mu_j)$ holds (by using $s=\omega(1)$ in~\eqref{eq:thm:basic:Qj}, so that usually $\mu/Q_\ell = \omega(\mu/R_\ell)$ in~\eqref{eq:thm:extended}, say). 
\begin{theorem}[Extended upper tail inequality]%
\label{thm:extended}
Given $\cH$ with $1 \le \ell \le q \le k$, assume that (H$\ell$) and (P) hold.  
Suppose that $s \ge 1$ and $t>0$. 
Given positive $(R_j)_{\ell \le j < q}$ and $(D_j)_{\ell \le j \le q}$ with $R_j \ge D_j$, define 
\begin{equation}\label{eq:thm:basic:Qj}
Q_j =\max\{R_j/s, D_j\} 
\end{equation}
for $\ell \le j < q$, and $R_q=Q_q=D_q$.  
If inequality
\begin{equation}\label{eq:thm:basic:asmpt}
\max\left\{\left(\frac{e \mu_j}{Q_j}\right)^{R_j/R_{j+1}}, \; \indic{Q_j < R_j \text{ and } Q_{j+1}=D_{j+1}} \left(\frac{e \mu_j}{Q_j}\right)^{Q_j/D_{j+1}}\right\} \le N^{-4kj} 
\end{equation}
holds for all $\ell \le j < q$,  
then for $a=1/\bigl(4L\binom{k}{\ell}\bigr)$, $b=1/(2k)$ and $d=1/\bigl(4Lqk \binom{k}{\ell}\bigr)$ we have 
\begin{equation}\label{eq:thm:extended}
\begin{split}
& \Pr(X(\cG) \ge \mu+t \text{ and } \Delta_q(\cG) \le D_q \text{ for some }  \cG \subseteq \cH_p) \\
& \qquad \le  \exp\left(-\frac{a\varphi(t/\mu)\mu}{Q_\ell}\right)  + 2\sum_{\ell \le j < q}N^{-j}\left(\frac{e \mu_j}{Q_j}\right)^{bR_j/R_{j+1}} \\
& \qquad \quad\quad + \sum_{\ell \le j < q}\indic{Q_j < R_j \text{ and }  Q_{j+1}=D_{j+1}}N^{-j}\left(\frac{e \mu_j}{Q_j}\right)^{\max\bigl\{dt/(R_\ell D_{j+1}), \; b Q_j/D_{j+1}\bigl\}}.
\end{split}
\end{equation}
\end{theorem}
\noindent 
To illustrate Theorem~\ref{thm:extended}, in the applications of Sections~\ref{sec:proof:easy} and~\ref{sec:applications:RISH} we have $e\mu_j/R_j \le p^{\alpha}/e$ with $p \in (0, 1]$, in which case $s = \log(e/p^{\alpha/2})$ is a convenient choice. 
Indeed, $x \log(e/x) \le 1$ then implies $e\mu_j/Q_j \le e\mu_j s/R_j \le p^{\alpha/2}/e = e^{-s}$. 
We thus think of the~\eqref{eq:thm:basic:asmpt} as a minor variant of the assumption~\eqref{eq:cl:basic:asmpt} from Claim~\ref{cl:basic} (note that $e\mu_j/R_j \le e^{-s}$ holds, and that $Q_j < R_j$ implies $Q_j=R_j/s$).  
Using $D_j = \Theta(1)$ and the additional Kim--Vu type assumptions discussed below Claim~\ref{cl:basic}, we now review inequality~\eqref{eq:thm:extended} of Theorem~\ref{thm:extended}. 
Since~$1/Q_\ell = \min\{s/R_\ell, 1/D_\ell\}$, 
using $t/R_{\ell} = \sqrt{\lambda \mu/R_{\ell}} \ge \lambda$ we obtain analogous to~\eqref{eq:cl:basic:heur} an estimate of the form   
\begin{equation}\label{eq:thm:extended:heur}
\begin{split}
\Pr(X(\cH_p) \ge \mu+t) & \le \exp\bigl(-\tilde{a}\min\{t^2/\mu, \lambda s\}\bigr) + \indic{q > \ell}3qN^{-\ell}e^{-\tilde{d}\lambda s} \\
& \le C \exp\bigl(-c\min\{t^2/\mu, \lambda \log(e/p)\}\bigr) .
\end{split}
\end{equation}
If $q > \ell$ then $t^2/\mu = \lambda R_{\ell} \ge \lambda^{q-\ell+1}R_{q} = \omega(\lambda \log N)$, so~\eqref{eq:thm:extended:heur} usually decays like $C\exp(- c\lambda \log(e/p))$. 
When $\lambda \approx \mu^{1/(q-\ell+1)}$ or $t=\eps \mu$ we similarly see that~\eqref{eq:thm:extended:heur} decays like $C\exp(-c\min\{\mu, \lambda \log(e/p)\})$.  
In all these cases we thus improve the exponential decay of the classical bound~\eqref{eq:cl:basic:heur} by an extra logarithmic factor.

The following upper tail inequality for polynomially small~$\mu_j$ is a minor extension of Theorem~\ref{thm:pre:extended:small}.  
%(see Appendix~\ref{sec:proof:basic:alternative} for a variant that is based on different assumptions). 
Note that~\eqref{eq:thm:extended:small} decays exponentially in $\min\{t^2/\mu,t^{1/(q-\ell+1)}\log N\}$ for $1 \le t \le O(\mu)$, 
which seems quite informative when $\mu=\Theta(\Var X(\cH))$ holds (i.e., in the Poisson range). 
\begin{theorem}[Upper tail inequality: the small expectations case]%
\label{thm:extended:small}
Given $\cH$ with $1 \le \ell \le q \le k$, assume that (H$\ell$) and (P) hold. 
If there are $A,\alpha > 0$ such that inequality~\eqref{eq:thm:pre:basic:small:asmpt} holds, then for $t,K>0$ we have 
\begin{equation}\label{eq:thm:extended:small}
\begin{split}
& \Pr(X(\cG) \ge \mu+t \text{ and } \Delta_q(\cG) \le D \text{ for some } \cG \subseteq \cH_p) \\
& \qquad \le  \exp\bigl(-a\varphi(t/\mu)\mu\bigr)  + \indic{q > \ell}2q N^{-q} \exp\left(-\max\{bt^{1/(q-\ell+1)}, \; K\}\log N\right),
\end{split}
\end{equation}
where $a,b>0$ depend only on $\ell,q,k,L,D,A,\alpha,K$. 
\end{theorem}

\subsection{Proofs}\label{sec:prob:thm}
\subsubsection{Proofs of Claim~\ref{cl:basic} and Theorems~\ref{thm:extended}--\ref{thm:extended:small}}\label{sec:prob:thm:basic}
Combining Theorem~\ref{thm:sparse:iter} and~\ref{thm:prob}, by setting $S_j = R_j/s$ the proof of Theorem~\ref{thm:extended} is straightforward. 
\begin{proof}[Proof of Theorem~\ref{thm:extended}]
We first consider the special case $q=\ell$. Since $R_q=D_q$, using $s \ge 1$ we thus infer $\max\{R_\ell/s,D_\ell\}=D_\ell=R_\ell$. 
Hence~\eqref{eq:thm:e:extended} of Theorem~\ref{thm:prob:extended} readily implies~\eqref{eq:thm:extended}. 

In the remainder we focus on the more interesting case $q > \ell$.
Analogous to the proof of Theorem~\ref{thm:prob}, inequality~\eqref{eq:thm:extended} is trivial when $N < 1$ (the left hand side is zero).
So we henceforth may assume $N \ge 1$, and using the assumption~\eqref{eq:thm:basic:asmpt} it follows that $Q_j \ge e \mu_j$. 
Let $S_j = R_j/s$, and recall that $Q_j = \max\{S_j,D_j\}$ in Theorem~\ref{thm:sparse:iter:extended}. 
Note that $s \ge 1$ and $R_j \ge D_j$ imply $Q_j \le R_j$. 
In view of~\eqref{eq:thm:sparse:iter:extended} and~\eqref{eq:thm:e:extended} of Theorem~\ref{thm:sparse:iter:extended} and~\ref{thm:prob:extended}, it remains to estimate $\Pr_{j,1}$, $\Pr_{j,2}$ and $\Pr_{j,3,\ell}$ defined in~\eqref{eq:thm:sparse:pj1}--\eqref{eq:thm:sparse:pj3}. 
Starting with~$\Pr_{j,1}$ and~$\Pr_{j,2}$, using~\eqref{eq:thm:deg:extended} together with $R_j \ge Q_j$, $Q_j/S_{j+1} \ge S_{j}/S_{j+1}=R_j/R_{j+1}$ and the assumption~\eqref{eq:thm:basic:asmpt} we infer 
\begin{equation}\label{eq:thm:basic:pj1pj2}
\begin{split}
\Pr_{j,1} + \Pr_{j,2} & \le N^{j} \left(\frac{e\mu_j}{R_j}\right)^{R_j/(kR_{j+1})} + N^{j} \left(\frac{e\mu_j}{Q_j}\right)^{Q_j/(kS_{j+1})} \\
& \le 2 N^{j} \left(\frac{e\mu_j}{Q_j}\right)^{R_j/(kR_{j+1})} 
\le 2 N^{-j}\left(\frac{e\mu_j}{Q_j}\right)^{R_j/(2kR_{j+1})} .
\end{split}
\end{equation}
Finally, for~$\Pr_{j,3,\ell}$ of~\eqref{eq:thm:sparse:pj3} we henceforth tacitly assume $Q_j < R_j$ and $Q_{j+1}=D_{j+1}$.  
With an eye on~\eqref{eq:thm:sparse:extended}, using $Q_j \ge e \mu_j$ and the assumption~\eqref{eq:thm:basic:asmpt} we then (with foresight) similarly deduce
\begin{equation*}
\begin{split}
\Pi := N^{j} \left(\frac{e\mu_j}{\ceil{Q_j}}\right)^{\ceil{Q_j}/(kD_{j+1})} & \le N^{j} \left(\frac{e\mu_j}{Q_j}\right)^{\ceil{Q_j}/(kD_{j+1})} 
\le N^{-j} \left(\frac{e\mu_j}{Q_j}\right)^{\ceil{Q_j}/(2kD_{j+1})}.
\end{split}
\end{equation*}
Since $\ceil{x} \ge \max\{x,1\}$, by applying~\eqref{eq:thm:sparse:extended} with $(x,r,y,z)=(Q_{j}, t/(2qL), D_{j+1}, R_\ell)$ it follows that 
\begin{equation*}\label{eq:thm:basic:pj3}
\Pr_{j,3,\ell} \le 
(\Pi)^{\ceilL{t/(2Lq\binom{k}{\ell}\ceil{Q_j}R_\ell)}} \le N^{-j}\left(\frac{e\mu_j}{Q_j}\right)^{\max\bigl\{dt/(R_\ell D_{j+1}), \; bQ_j/D_{j+1}\bigr\}} .
\end{equation*}
Recalling our tacit assumption for~$\Pr_{j,3,\ell}$, this completes the proof in view of~\eqref{eq:thm:sparse:iter:extended}, \eqref{eq:thm:e:extended} and~\eqref{eq:thm:basic:pj1pj2}. 
\end{proof}
The details of the similar but simpler proof of Claim~\ref{cl:basic} are omitted 
(the above proof carries over by setting $s=1$ and $D_j=R_j$, since $Q_j=\max\{R_j/s,D_j\}=R_j$ implies $\Pr_{j,2}=\Pr_{j,3,\ell}=0$).

For the proof of Theorem~\ref{thm:extended:small} we need to define the parameters $(R_j)_{\ell \le j \le q}$ and $(D_j)_{\ell \le j \le q}$ of Theorem~\ref{thm:sparse:iter} and~\ref{thm:prob} in a suitable way. 
Intuitively, we shall set $R_j = \lambda^{q-j}D$, $\lambda = \max\{t^{1/(q-\ell+1)},B\}$ and $D_j=Q_j=B^{q-j}D=\Theta(1)$, and the crux is that the assumption~\eqref{eq:thm:pre:basic:small:asmpt} eventually yields $e\mu_j/x \le N^{-\Theta(1)}$ in~\eqref{eq:thm:deg:extended}--\eqref{eq:thm:sparse:extended}.  
We shall also exploit the indicators in~Theorem~\ref{thm:sparse:iter:extended} for estimating $t/R_{\ell}$ 
in~\eqref{eq:thm:extended}, see~\eqref{eq:thm:basic:small:pj3} below. 
\begin{proof}[Proof of Theorem~\ref{thm:extended:small}]
With foresight, let $B = \max\bigl\{4qk/\alpha,2kK/\alpha ,Ae/D,1\bigr\}$ and $\lambda = \max\{t^{1/(q-\ell+1)},B\}$. 
Define $D_j = S_j = B^{q-j}D$ and $R_j = \lambda^{q-j}D$ for all $\ell \le j \le q$. 
Note that $Q_j = \max\{S_j,D_j\} = D_j$ and $\min\{Q_j,R_j\} = D_j$, so that $\Pr_{j,2}=0$ in~\eqref{eq:thm:sparse:pj2}. 
Combining~\eqref{eq:thm:sparse:iter:extended} and~\eqref{eq:thm:e:extended} of Theorem~\ref{thm:sparse:iter:extended} and~\ref{thm:prob:extended}, we obtain 
\begin{equation}\label{eq:thm:basic:w}
\Pr(X(\cG) \ge \mu+t \text{ and } \Delta_q(\cG) \le D \text{ for some } \cG \subseteq \cH_p) \le \exp\left(-\frac{\varphi(t/\mu)\mu}{4L \binom{k}{\ell} D_\ell} \right) + \sum_{\ell \le j < q} \bigl[\Pr_{j,1} + \Pr_{j,3,\ell} \bigr] .
\end{equation}
Tacitly assuming $q > \ell$, it remains to estimate $\Pr_{j,1}$ and $\Pr_{j,3,\ell}$ defined in~\eqref{eq:thm:sparse:pj1} and \eqref{eq:thm:sparse:pj3}. 
Starting with~$\Pr_{j,1}$, by inserting~\eqref{eq:thm:pre:basic:small:asmpt} into~\eqref{eq:thm:deg:extended}, using $R_j \ge D B \ge Ae$ and $R_j/R_{j+1}=\lambda \ge B \ge 4qk/\alpha$ we infer 
\begin{equation}\label{eq:thm:basic:small:pj1}
\Pr_{j,1} \le N^{j} \left(\frac{e\mu_j}{R_j}\right)^{R_j/(kR_{j+1})}\le N^{q} \bigl(\mu_j/A\bigr)^{\lambda/k} \le N^{q-\alpha \lambda/k} \le N^{-q-\alpha \lambda/(2k)} . 
\end{equation}
For~$\Pr_{j,3,\ell}$, using $\ceil{Q_j} \ge Ae$ and $Q_{j}/D_{j+1} \ge B \ge 4qk/\alpha$ we (with foresight) similarly deduce 
\[
\Pi := N^{j} \left(\frac{e\mu_j}{\ceil{Q_j}}\right)^{\ceil{Q_j}/(kD_{j+1})} 
\le N^{-q-\alpha \ceil{Q_j}/(2k D_{j+1})}.
\]
Note that $\lambda = B$ implies $R_j=D_j=Q_j$. Hence $Q_j < R_j$ ensures $\lambda = t^{1/(q-\ell+1)}$, so that $t/R_\ell = t^{1/(q-\ell+1)}/D$. 
Recalling $\ceil{Q_j}/D_{j+1} \ge B$, by applying~\eqref{eq:thm:sparse:extended} with $(x,r,y,z)=(Q_{j}, t/(2qL), D_{j+1}, R_\ell)$ we thus infer 
\begin{equation}\label{eq:thm:basic:small:pj3}
\Pr_{j,3,\ell} \le 
\indic{Q_j < R_j}(\Pi)^{\ceilL{t/(2Lq\binom{k}{\ell}\ceil{Q_j}R_\ell)}} \le N^{-q-\max\bigl\{\beta t^{1/(q-\ell+1)}/D_{j+1}, \; \alpha B/(2k)\bigr\} } ,
\end{equation}
with $\beta = \alpha / (4Lqk\binom{k}{\ell}D)$. 
With the above estimates~\eqref{eq:thm:basic:small:pj1} and~\eqref{eq:thm:basic:small:pj3} for~$\Pr_{j,1}$ and~$\Pr_{j,3,\ell}$ in hand, using $B \ge 2kK/\alpha$ and $D_{j+1} \le D_{\ell}$ it follows by definition of $\lambda = \max\{t^{1/(q-\ell+1)},B\}$ that  
\[
\sum_{\ell \le j < q} \bigl[\Pr_{j,1} + \Pr_{j,3,\ell} \bigr] \le \indic{q > \ell} 2 q N^{-q} \exp\left(-\max\bigl\{b t^{1/(q-\ell+1)}, \; K \bigr\} \log N\right) ,
\]
with $b=\min\bigl\{\alpha/(2k),\beta/D_{\ell}\bigr\}$. 
Recalling~\eqref{eq:thm:basic:w}, this establishes~\eqref{eq:thm:extended:small} with $a=1/(4L \binom{k}{\ell} D_\ell)$. 
\end{proof}

\subsubsection{Proofs of Theorem~\ref{thm:pre:extendedp} and~\ref{thm:pre:extended:small}}\label{sec:proof:easy}
The `easy-to-apply' inequalities from Section~\ref{sec:easy} are convenient corollaries of Theorems~\ref{thm:extended}--\ref{thm:extended:small}. 
Indeed, Remark~\ref{rem:varphi} implies $\varphi(t/\mu)\mu \ge \min\{t^2/\mu,t\}/3$, so Theorem~\ref{thm:pre:extended:small} follows readily from Theorem~\ref{thm:extended:small}. 
For Theorem~\ref{thm:pre:extendedp} the basic strategy is to apply Theorem~\ref{thm:extended} with $s = \log(e/\pi^{\alpha/2})$, $R_j = \lambda^{q-j}D$, $\lambda = B \max\{\mu^{1/(q-\ell+1)},1\}$ and $D_j=B^{q-j}D=\Theta(1)$. 
The crux is that the assumption~\eqref{eq:thm:pre:extendedp:asmpt1} eventually yields $e\mu_j/Q_j \le \pi^{\alpha/2}/e = e^{-s}$ in~\eqref{eq:thm:basic:asmpt}--\eqref{eq:thm:extended}. 
As before, the indicators in~Theorem~\ref{thm:extended} facilitate estimating $t/R_{\ell}$ in~\eqref{eq:thm:extended}, see~\eqref{eq:thm:randinduced:proof:2} below. 
\begin{proof}[Proof of Theorem~\ref{thm:pre:extendedp}]
The proof is naturally divided into four parts: (i) introducing definitions, (ii)~estimating~$e\mu_j/Q_j$, (iii)~applying inequality~\eqref{eq:thm:extended} of Theorem~\ref{thm:extended}, and (iv)~verifying assumption~\eqref{eq:thm:basic:asmpt}.

Analogous to the proof of Theorem~\ref{thm:prob} and~\ref{thm:extended}, we may henceforth assume $N \ge 1$. 
Furthermore, by increasing~$A$ or $D$ if necessary, we may of course assume $A,D \ge 1$. 
With foresight, let $\beta = \alpha/2$ and $s=\log(e/\pi^{\beta})$. 
Set $B=\max\{e^2A/D,4k^2/(\tau \beta),4k^2(4A)^q,1\}$ and $\lambda = B \max\{\mu^{1/(q-\ell+1)},1\}$. 
Define $R_j = \lambda^{q-j}D$ and $D_j = B^{q-j}D$, so that $R_j \ge D_j$ and $R_q=D_q=D$.

Next we estimate $e\mu_j/Q_j$, where $Q_j \ge R_j/s$. 
Using assumption~\eqref{eq:thm:pre:extendedp:asmpt1} and $\alpha = 2\beta$, for $\ell \le j < q$ we have 
\begin{equation}\label{eq:mujQj}
\frac{e\mu_j}{Q_j} \le \frac{e \mu_j s}{R_j} = \frac{e \mu_j s}{D B^{q-j} \max\{\mu^{(q-j)/(q-\ell+1)},1\}} \le \frac{eA\pi^{2\beta} \log(e/\pi^{\beta})}{DB} \le  \frac{\pi^{\beta}}{e} = e^{-s} ,
\end{equation}
where we tacitly used $\pi \in (0,1]$ and $x\log(e/x) \le 1$ for all $x \in [0,1]$.

We now apply inequality~\eqref{eq:thm:extended} of Theorem~\ref{thm:extended}, deferring the proof of the claim that assumption~\eqref{eq:thm:basic:asmpt} holds. 
Using~\eqref{eq:mujQj} and $R_j/R_{j+1}=\lambda$, note that $X(\cH)=X(\cH_p)$ and $\Delta_q(\cH) \le D=D_q$ yield 
\begin{equation}\label{eq:thm:randinduced:proof:1}
\begin{split}
\Pr(X(\cH) \ge (1+\eps)\mu) &\le \Pr(X(\cG) \ge \mu + \eps \mu \text{ and } \Delta_q(\cG) \le D_q \text{ for some }  \cG \subseteq \cH_p) \\
& \le \exp\left(-\frac{a\varphi(\eps)\mu}{\max\{R_\ell/s,D_\ell\}}\right)  + qN^{-\ell}\left[2e^{-b\lambda s} + \max_{\ell \le j < q}\indic{Q_j < R_j}e^{-d\eps \mu s/(R_\ell D_{j+1})}\right] .
\end{split}
\end{equation}
Note that $\lambda = B$ implies $R_j=D_j$, in which case $s \ge 1$ yields $Q_j = D_j = R_j$. 
Hence $Q_j<R_j$ ensures $\lambda = B \mu^{1/(q-\ell+1)}$, so that $R_\ell = (B\mu^{1/(q-\ell+1)})^{q-\ell}D$. 
Noting $D_{j+1} \le D_{\ell}$, it follows that 
\begin{equation}\label{eq:thm:randinduced:proof:2}
\max_{\ell \le j < q}\indic{Q_j < R_j}e^{-d\eps \mu s/(R_\ell D_{j+1})} 
\le \exp\left(-\frac{d}{D_{\ell}B^{q-\ell}D} \cdot \eps \mu^{1/(q-\ell+1)} s\right).
\end{equation}
Similarly, using $s \ge 1$ we also see that $R_\ell/s > D_\ell$ implies $R_\ell = (B \mu^{1/(q-\ell+1)})^{q-\ell}D$. Hence
\begin{equation}\label{eq:thm:randinduced:proof:3}
\exp\left(-\frac{a\varphi(\eps)\mu}{\max\{R_\ell/s,D_\ell\}}\right) \le \exp\left(-\min\left\{\frac{a}{D_\ell} \cdot \varphi(\eps)\mu, \; \frac{a}{B^{q-\ell}D} \cdot \varphi(\eps)\mu^{1/(q-\ell+1)}s\right\}\right) .
\end{equation}
Remark~\ref{rem:varphi} implies $\min\{\varphi(\eps),1,\eps\} \ge \min\{\eps^2,1\}/3$. 
So, combining~\eqref{eq:thm:randinduced:proof:1}--\eqref{eq:thm:randinduced:proof:3}, using $s \ge \allowbreak \min\{1,\beta\} \log(e/\pi)$ and $\lambda \ge B \mu^{1/(q+\ell-1)}$ our findings thus establish~\eqref{eq:thm:pre:extendedp} for suitable $c=c(\eps,k,q,D,L,\alpha)>0$.

In the following we verify assumption~\eqref{eq:thm:basic:asmpt}, i.e., the claim omitted above. 
Note that $R_j/R_{j+1} = \lambda \ge B$ and $Q_{j}/D_{j+1} \ge D_{j}/D_{j+1} = B$. 
Using~\eqref{eq:mujQj}, for $\pi \le N^{-\tau}$ the left hand side of \eqref{eq:thm:basic:asmpt} can thus be bounded by 
\begin{equation}\label{eq:thm:randinduced:proof:asmpt}
\left(\frac{e \mu_j}{Q_j}\right)^{B} 
\le \pi^{\beta B} 
\le N^{-\tau \beta B} \le N^{-4k^2} \le N^{-4k j} .
\end{equation}
For $\pi > N^{-\tau}$ we defer the proof of the claim that for $\ell \le j < q$ we have 
\begin{equation}\label{eq:exp:largep}
\min\{\lambda,R_j/D_{j+1}\} \ge 4k^2 \log N. 
\end{equation}
Using~\eqref{eq:mujQj}, $s \ge 1$, $Q_j \ge R_j/s$ and \eqref{eq:exp:largep} we 
see that the left hand side of \eqref{eq:thm:basic:asmpt} can be bounded by 
\[
\max\left\{\left(e^{-1} \right)^{R_j/R_{j+1}},\left(e^{-s}\right)^{R_j/(sD_{j+1})}\right\} 
\le \max\left\{e^{-\lambda},e^{-R_j/D_{j+1}}\right\} \le N^{-4k^2} \le N^{-4k j}.
\]
To sum up, we have verified~\eqref{eq:thm:basic:asmpt}, assuming that~\eqref{eq:exp:largep} holds for $\pi > N^{-\tau}$. 
Turning to the remaining claim~\eqref{eq:exp:largep}, using assumption~\eqref{eq:thm:pre:extendedp:asmpt2} we see that $\pi > N^{-\tau}$ implies  
\[
\lambda \ge B \mu^{1/(q-\ell+1)} \ge B (\log N)/A \ge 4k^2 \log N .
\] 
Similarly, $\pi > N^{-\tau}$, $\ell \le j < q$ and $N \ge 1$ imply 
\[
R_j/D_{j+1} = \lambda^{q-j}/ B^{q-j-1} \ge \bigl(B\mu^{1/(q-\ell+1)}\bigr)^{q-j}/ B^{q-j-1} \ge B \bigl((\log N)/A\bigr)^{q-j} \ge 4k^2 \log N ,
\]
establishing~\eqref{eq:exp:largep}. As discussed, this completes the proof of~\eqref{eq:thm:pre:extendedp}.
\end{proof}

\section{Applications}\label{sec:applications}
In this section we illustrate our concentration techniques, by applying the basic inequalities from Section~\ref{sec:easy} to several pivotal examples. 
In Section~\ref{sec:applications:RISH} we improve previous work of Janson and Ruci{\'n}ski~\cite{UTAP} on random induced subhypergraphs, and derive sharp upper tail inequalities for several quantities of interest in additive combinatorics. 
In Section~\ref{sec:applications:RG} we answer a question of Janson and Ruci{\'n}ski~\cite{DLP} on subgraph counts in binomial random graphs, and improve the main applications of Wolfovitz~\cite{WG2012} and {\v{S}}ileikis~\cite{MS}.

\subsection{Random induced subhypergraphs}\label{sec:applications:RISH}
In probabilistic combinatorics, random induced subhypergraphs~$\cH_p$ are a standard test-bed for upper tail inequalities (see, e.g., Section~3 in the survey~\cite{UT}).  
Janson and Ruci{\'n}ski studied the number of randomly induced edges in~\cite{UTAP}, 
and one of their principle results concerns $k$-uniform hypergraphs with $v(\cH) = N$ vertices, $e(\cH)  \ge \gamma N^{q}$ edges and $\Delta_q(\cH) \le D$ (for easier comparison with Theorem 2.1 in~\cite{UTAP}, note that $\Delta_j(\cH) \le N^{\max\{q-j,0\}}\Delta_q(\cH)$ holds). 
Writing $X=e(\cH_p)$ and $\mu=\E X$, they obtained bounds of form
\begin{equation}\label{eq:thm:randinduced:previous}
\exp\bigl(-C(\eps)\mu^{1/q}\log(1/p)\bigr) \le \Pr(X \ge (1+\eps)\mu) \le \exp\bigl(-c(\eps)\mu^{1/q}\bigr) ,
\end{equation}
determining $\log \Pr(X \ge (1+\eps)\mu)$ up to a missing logarithmic factor (in fact, their lower bound needs an extra assumption). 
For $2 \le q<k$ the following corollary of Theorem~\ref{thm:pre:extendedp} improves the exponential rate of decay of~\eqref{eq:thm:randinduced:previous} in the more general weighted case. 
Noteworthily, inequality~\eqref{eq:thm:randinduced} below closes the $\log(1/p)$ gap left open by Janson and Ruci{\'n}ski~\cite{UTAP} (for the special case $q=2$ this was already resolved in~\cite{AP}).  
%and implies Theorem~\ref{thm:intro:indshg}
%
\begin{theorem}[Weighted edge-count of random induced subhypergraphs]%
\label{thm:randinduced}
Let $1 \le q < k$ and $\gamma,D,a,L> 0$. 
Assume that~$\cH$ is a %Let $\cH$ be a 
$k$-uniform hypergraph with $v(\cH) \le N$, $e(\cH) \ge \gamma N^{q}$, $\Delta_q(\cH) \le D$, and $w_f \in [a,L]$ for all $f \in \cH$. 
Set $X = w(\cH_p)$ and $\mu=\E X$. 
For any $\eps > 0$ there is $c=c(\eps,k,\gamma,D,a,L)>0$ such that for all $p \in (0,1]$ we have 
\begin{equation}\label{eq:thm:randinduced}
\Pr(X \ge (1+\eps)\mu)  \le  \exp\Bigl(-c \min\bigl\{\mu, \; \mu^{1/q} \log(e/p)\bigr\}\Bigr) . 
\end{equation}
\end{theorem}
\begin{remark}
Setting $p = m/v(\cH)$, inequality~\eqref{eq:thm:randinduced} also carries over to~$\cH_m$ as defined in Section~\ref{sec:uniform}. 
\end{remark}
%
%\noindent 
%As we shall discuss, Theorem~\ref{thm:randinduced} is best possible in several ways: (i)~the restriction to $q<k$ is necessary, and (ii)~in many applications~\eqref{eq:thm:randinduced} is sharp, i.e., yields the correct exponential rate of decay. 
Inequality~\eqref{eq:thm:randinduced} does not always hold in the excluded case $q=k$. 
%In the excluded case $q=k$ inequality~\eqref{eq:thm:randinduced} does not hold in general.  
A concrete counterexample is the complete $k$-uniform hypergraph $\cH=\cH_N$ with $V(\cH)=[N]$ and $w_f=1$. 
Then $q=k$, $X = \binom{|[N]_p|}{k} \approx |[N]_p|^k/k!$ and $\mu = \binom{N}{k}p^k \approx (Np)^k/k!$. For $\mu = \omega(1)$, $p \le 1/2$ and $\eps=\Theta(1)$ it is routine to see that $\Pr(w(\cH_p) \ge (1+\eps)\mu) = \exp\bigl(-\Theta(Np)\bigr)= \exp\bigl(-\Theta(\mu^{1/q})\bigr)$ holds, i.e., that there is no logarithmic term. 

Concerning sharpness of~\eqref{eq:thm:randinduced}, in applications we usually do not consider a single hypergraph $\cH$, but sequences of hypergraph $(\cH_N)_{N \in \NN}$ which are nearly monotone, i.e., where $\cH_{N} \subseteq \cH_{N+1}$ holds up to some minor `defects' (arising, e.g., due to boundary effects). 
The following remark states that, in this frequent case, the upper tail inequality~\eqref{eq:thm:randinduced} is best possible up to the value of the parameter~$c$ (for $2 \le q < k$).  
\begin{remark}[Matching lower bound]%
\label{rem:randinduced:lower}
Let $2 \le q < k$ and $\gamma,D,a,L,n_1,n_2> 0$. 
Let $(\cH_N)_{N \ge n_1}$ be a sequence of $k$-uniform hypergraphs such that all $\cH=\cH_N$ satisfy the assumptions of Theorem~\ref{thm:randinduced}. 
Assume that there is $\beta \in (0,1]$ such that $e(\cH_N \cap \cH_M) \ge \beta e(\cH_N)$ for all $M \ge N \ge n_2$. 
Then for all~$\eps > 0$ there are~$n_0=n_0(k,\gamma,D,a,L,\beta,n_1,n_2)>0$ and~$C=C(\eps,\gamma,k,q,D,a,L,\beta,n_1,n_2,)>0$ such that for all~$\cH=\cH_N$ with~$N \ge n_0$, setting~$X = w(\cH_p)$ and~$\mu=\E X$, for all~$p \in (0,1]$ we have 
\begin{equation}\label{eq:thm:randinduced:lower}
\Pr(X \ge (1+\eps)\mu)  \ge  \indic{1 \le (1+\eps)\mu \le w(\cH)} \exp\Bigl(-C \min\bigl\{\mu, \; \mu^{1/q} \log(1/p)\bigr\}\Bigr) .
\end{equation}
\end{remark}
\noindent 
We omit the proof of Remark~\ref{rem:randinduced:lower}, which mimics the lower bound techniques from~\cite{AP} in a routine way. 
\begin{proof}[Proof of Theorem~\ref{thm:randinduced}]
Let $\delta = a \gamma$, and note that $\mu \ge e(\cH)p^k \cdot \min_{f \in \cH} w_f \ge \delta N^qp^k$ (we never use $w_f \ge a$ again, i.e., we could weaken our assumptions).  
Inequality~\eqref{eq:thm:randinduced} holds trivially whenever $N< k$ (since then $0 \le w(\cH_p) \le L \cdot e(\cH)=0$), so we may henceforth assume $N \ge k$.  
Our main task is to verify the assumptions of Theorem~\ref{thm:pre:extendedp}. 
Let $\ell=1$ and $\tau = q/(2k)$. 
As $N^{1/2} \ge \log N$ for all $N > 0$, for $p \ge N^{-\tau}$ we have 
\begin{equation}\label{eq:mu}
\mu^{1/(q-\ell+1)} = \mu^{1/q} \ge \delta^{1/q} Np^{k/q} \ge \delta^{1/q} N^{1-k\tau/q} \ge \delta^{1/q} N^{1/2} \ge \delta^{1/q} \log N .
\end{equation}
As discussed in Example~\ref{ex:RIH}, using~\eqref{eq:muj:extended:RIH} and $|\Gamma_U(\cH)| \le v(\cH)^{q-j} \cdot \Delta_q(\cH)$, for $1 \le j < q$ we thus have   
\begin{equation}\label{eq:muj:RIH}
\mu_j \le N^{q-j} \cdot D \cdot p^{k-j}.
\end{equation}
Recalling $\ell=1$, \eqref{eq:mu} and $q<k$, there thus is a constant $A=A(D,\delta)>0$ such that for $1 \le j < q$ we have  
\begin{equation}\label{eq:mujRj}
\frac{\mu_j}{\mu^{(q-j)/(q-\ell+1)}} \le \frac{D N^{q-j} p^{k-j}}{(\mu^{1/q})^{q-j}} \le D \delta^{j/q-1} p^{j(k/q-1)} \le A p^{1/q} .
\end{equation}
Hence assumptions~\eqref{eq:thm:pre:extendedp:asmpt1}--\eqref{eq:thm:pre:extendedp:asmpt2} hold with $\pi=p$ and $\alpha=1/q$. 
Using~\eqref{eq:thm:pre:extendedp} of Theorem~\ref{thm:pre:extendedp} it follows that 
\begin{equation}\label{eq:thm:randinduced:lower:raw}
\Pr(w(\cH_p) \ge (1+\eps)\mu) \le (1+3qN^{-1}) e^{-\Pi} ,
\end{equation}
where $\Pi=c' \min\bigl\{\eps^2,1\bigr\} \min\{\mu, \; \mu^{1/(q-\ell+1)}\log(e/p)\}$ and $c'=c'(\ell,q,k,L,D,A,\delta)>0$. 

The author finds~\eqref{eq:thm:randinduced:lower:raw} quite satisfactory, but in the literature the usually irrelevant prefactor $1+3qN^{-1}$ is often suppressed for cosmetic reasons. 
Below we shall achieve this by inflating the constant in the exponent (\emph{without} assuming that~$n$, $p$ or $\Pi$ are large). 
If $\Pi \ge 6$, then $N \ge k \ge q$ implies $3qN^{-1} \le 3 \le \Pi/2$, so that 
\[
\Pr(w(\cH_p) \ge (1+\eps)\mu) \le e^{-\Pi + 3qN^{-1}}  \le e^{-\Pi/2} .
\]
Otherwise $1 \ge \Pi/6$ holds, in which case $\eps/(1+\eps) \ge \min\{1,\eps\}/2$ and Markov's inequality yield 
\[
\Pr(w(\cH_p) \ge (1+\eps)\mu) \le \frac{1}{1+\eps} = 1-\frac{\eps}{1+\eps} \le e^{-\eps/(1+\eps)} \le e^{-\min\{1,\eps\}\Pi/12} ,
\]
establishing~\eqref{eq:thm:randinduced} for suitable $c=c(\eps,c')>0$.  
\end{proof}

Combining Theorem~\ref{thm:randinduced} and Remark~\ref{rem:randinduced:lower}, we obtain the following convenient upper tail result (see~\cite{AP} for a similar result in the special case~$q=2$).  
It applies to many widely-studied objects in additive combinatorics and Ramsey theory, each time closing the logarithmic gap present in previous work, see~\eqref{eq:thm:randinduced:previous} and~\cite{UTAP}. 
\begin{corollary}\label{cor:randinduced}
Let $2 \le q < k$  and $\gamma,D,a,L,n_1> 0$. 
Let $(\cH_n)_{n \ge n_1}$ be $k$-uniform hypergraphs such that $\cH_n \subseteq \cH_{n+1}$, $v(\cH_n) \le n$, $e(\cH_n)  \ge \gamma n^{q}$, $\Delta_q(\cH_n) \le D$, and $w_f \in [a,L]$ for all $f \in \cH_n$. 
Then for all~$\eps > 0$ there are $n_0=n_0(k,\gamma,D,a,L,n_1)>0$ 
and~$c,C > 0$ (depending only on $\eps,k,\gamma,D,a,L,n_1$) 
such that for all~$\cH=\cH_n$ with~$n \ge n_0$, setting~$X = w(\cH_p)$ and~$\mu=\E X$, for all $p \in (0,1]$ we have %with~$1 \le (1+\eps)\mu \le w(\cH)$ we have 
\begin{equation}\label{eq:indshg:both}
\indic{1 \le (1+\eps)\mu \le w(\cH)} \exp\Bigl(-C \Psi_{q,\mu}\Bigr) \le \Pr(X \ge (1+\eps)\mu)  \le  \exp\Bigl(-c  \Psi_{q,\mu}\Bigr) , %\quad \text{where } \Psi_{q,\mu} = \min\{\mu, \; \mu^{1/q} \log(1/p)\} .
\end{equation}
where~$\Psi_{q,\mu} = \min\{\mu, \; \mu^{1/q} \log(1/p)\}$. 
\end{corollary}
\noindent 
In particular, letting the edges of the $k$-uniform hypergraphs~$\cH_n$ with vertex-set~$V(\cH)=[n]$ encode the relevant objects, 
it is not difficult to check that Corollary~\ref{cor:randinduced} with uniform weights~$w_f=1$ implies\footnote{Note that using weights~$w_f=1$ we count unordered objects, i.e., treat the objects as~$k$-sets (if desired, we could also treat them as ordered~$k$-vectors by using non-uniform weights~$w_f >0$, say).} 
all the upper tail bounds presented in Examples~\ref{ex:AP}--\ref{ex:rssum} of Section~\ref{sec:intro:examples} (using~$q=2$ for $k$-term arithmetic progressions, $(k,q)=(3,2)$ for Schur triples, $(k,q)=(4,3)$ for additive quadruples, and~$(k,q)=(r+s,r+s-1)$ for $(r,s)$-sums). 
Motivated by Section~2.1 in~\cite{UTAP}, we now record a further common generalization of these~examples. 
\begin{example}[Integer solutions of linear homogeneous systems]
Let  $1 \le r \le k-2$. Let $A$ be a~$r \times k$ integer matrix. 
Following~\cite{UTAP}, we assume that every $r \times r$ submatrix~$B$ of~$A$ has full rank, i.e., $\mathrm{rank}(B)=r=\mathrm{rank}(A)$. 
We also assume that there exists a distinct-valued positive integer solution to $A x = \uzero$, where $x=(x_1, \ldots, x_k)$ is a column vector and $\uzero=(0, \ldots, 0)$ is an $r$-dimensional column vector. 
Let the edges of the $k$-uniform hypergraph~$\cH_n$ with~$V(\cH_n)=[n]$ encode solutions $\{x_1, \ldots, x_k\}\subseteq [n]$ of the system $A x = \uzero$ with distinct~$x_i$. 
The discussion of Section~2.1 in~\cite{UTAP} implies that $(\cH_n)_{n \ge n_1}$ satisfies the assumptions of Corollary~\ref{cor:randinduced} with~$q=k-r$, 
so the upper tail inequality~\eqref{eq:indshg:both} holds for~$X = e(\cH_p)$, say.  
\end{example}

\subsubsection{Small expectations case}\label{sec:applications:RISH:small}
Note that inequality~\eqref{eq:indshg:both} does not guarantee a similar dependence of $c,C>0$ on $\eps$. 
Of course, we can~also ask for finer results, which determine how the exponential decay of the upper tail depends on $\eps$. 
The following corollary of Theorem~\ref{thm:pre:extended:small} provides a partial answer for small~$p$ (see~\cite{AP} for results which for~$q=2$ cover all~$p$). 
\begin{theorem}\label{thm:randinduced:small}
Let $k \ge 2$. Let $1 \le q \le k$ and $D,L > 0$. 
Assume that~$\cH$ is a $k$-uniform hypergraph with $v(\cH) \le N$, $\Delta_q(\cH) \le D$ and $\max_{f \in \cH}w_f\le L$, where $N \ge 1$.   
Set $X = w(\cH_p)$ and $\mu=\E X$. 
For all $\sigma,\Lambda > 0$ there are $c=c(\sigma,\Lambda,k,D,L)>0$ and $d=d(q) \ge 1$ such that for all $p \le \Lambda N^{-(q-1)/(k-1)-\sigma}$ and $t > 0$ we have 
\begin{equation}\label{eq:thm:randinduced:small}
\Pr(X \ge \mu+t) \le d \exp\left(-c \min\Bigl\{\varphi(t/\mu) \mu, \; t^{1/q}\log N\Bigr\}\right) .
\end{equation}
Furthermore, setting $p = m/v(\cH)$, inequality~\eqref{eq:thm:randinduced:small} also holds with $\cH_p$ replaced by $\cH_m$. 
\end{theorem}
\noindent 
Assume that $\cH=\cH_N$ also satisfies $e(\cH_N) \ge \gamma N^q$, the monotonicity conditions of Remark~\ref{rem:randinduced:lower}, $w_f = 1$ and $2 \le q < k$. 
Mimicking the lower bound arguments from~\cite{AP}, inequality~\eqref{eq:thm:randinduced:small} can then shown to be best possible up to the 
values of~$d,c$ for some range of small~$p$ (we leave the details to the interested reader). 
\begin{proof}[Proof of Theorem~\ref{thm:randinduced:small}]
Our main task is to verify assumption~\eqref{eq:thm:pre:basic:small:asmpt} of Theorem~\ref{thm:pre:extended:small}.
To this end we exploit that 
\[
\frac{q-1}{k-1} = \max_{1 \le j < q }\frac{q-j}{k-j} .
\]
Indeed, using~\eqref{eq:muj:RIH} and $N \ge 1$ there thus is a constant $A=A(D,\Lambda)>0$ such that we have 
\begin{equation*}\label{eq:muj:RIH2}
\max_{1 \le j < q}\mu_j \le\sum_{1 \le j < q} D N^{q-j} p^{k-j} \le D \sum_{1 \le j < q} \Lambda^{k-j} N^{(q-j)-(k-j)(q-1)/(k-1)-(k-j)\sigma} \le A N^{-\sigma} .
\end{equation*}
Applying Theorem~\ref{thm:pre:extended:small} (with $\sigma=\alpha$ and $K=1$) now readily establishes inequality~\eqref{eq:thm:randinduced:small}. 
\end{proof}

\subsection{Subgraph counts in random graphs}\label{sec:applications:RG} 
In this section we consider subgraph counts in the binomial random graph~$G_{n,p}$, which are pivotal examples for illustrating 
various concentration methods (see, e.g.,~\cite{KimVu2000,Vu2001,Vu2002,UT,DL,UTSG} and Examples~\ref{ex:EE}--\ref{ex:VE} in Section~\ref{sec:easy:examples}). 
We shall discuss two qualitatively different upper tail bounds in Sections~\ref{sec:applications:RG:SG} and~\ref{sec:applications:RG:SG:large}.

We henceforth tacitly write $X=X_H$ for the number of copies of~$H$ in $G_{n,p}$, and set $\mu=\E X = \Theta(n^{v_H}p^{e_H})$. 
Let us recall some definitions from random graph theory. 
Writing $d(J)=e_J/v_J$, a graph~$H$ is called \emph{balanced} if $e_H \ge 1$ and $d(H) \ge d(J)$ for all $J \subsetneq H$ with $v_J \ge 1$. 
If this holds with $d(H) > d(J)$, then~$H$ is called \emph{strictly balanced}. % (examples include trees, regular connected graphs, cycles, complete graphs, complete $r$-partite graphs $K_{t,\ldots,t}$ and the $d$-dimensional cube).   
Writing $d_2(J)=(e_J-1)/(v_J-2)$, a graph~$H$ is called \emph{$2$-balanced} if $e_H \ge 2$ and $d_2(H) \ge d_2(J)$ for all $J \subsetneq H$ with $v_J \ge 3$. 
If this holds with $d_2(H) > d_2(J)$, then~$H$ is called \emph{strictly $2$-balanced}. % (examples include cycles, complete graphs~$K_s$, complete $r$-partite graphs $K_{t,\ldots,t}$ and the $d$-dimensional cube)

\subsubsection{Small deviations: sub-Gaussian type bounds}\label{sec:applications:RG:SG}
We first consider \emph{sub-Gaussian type}~$\Pr(X \ge \mu+t) \le C\exp(-ct^2/\Var X)$ upper tail inequalities. 
Our main focus is on the Poisson range, where $\Var X \sim \E X = \mu$ holds, which according to Kannan~\cite{Kannan} is the more difficult range. 
For small~$p$ the following simple corollary of Theorem~\ref{thm:pre:extended:small} extends/sharpens several results from~\cite{Vu2000,DL,MS,WG2012,Kannan,WG2011}, and implies Theorem~\ref{thm:intro:sgcount}. 
(For balanced and $2$-balanced graphs~$H$ it is folklore that $\delta_H \ge 1$. Furthermore, with the exception of perfect matchings, all $2$-balanced graphs are strictly balanced.)   
\begin{theorem}[Subgraph counts in random graphs: small expectations case]%
\label{thm:sg:small}
Let~$H$ be a graph with $v=v_H$~vertices, $e=e_H$~edges and minimum degree $\delta=\delta_H$. 
Let $X=X_H$ and $\mu=\E X$.  
Define $s=\min\{v-1,e-\delta+1\}$. 
If $H$ is strictly balanced, then for every $\Lambda >0$ there are $c=c(\Lambda,H)>0$ and $C=C(H) \ge 1$ such that for all $n \ge v$, $\eps \in (0,\Lambda]$ and $p \in [0,1]$ satisfying $\mu^{(s-1)/s} \le \Lambda \log n$ we have 
\begin{equation}\label{eq:thm:sg:small}
\Pr(X \ge (1+\eps)\mu) \le C \exp\Bigl(-c \eps^2\mu \Bigr) .
\end{equation}
If~$H$ is $2$-balanced, then for all $\sigma,\Lambda > 0$ there are $c=c(\sigma,\Lambda,H)>0$ and $C=C(H) \ge 1$ such that for all $n \ge v$, $0 \le p \le \Lambda n^{-(v-2)/(e-1)-\sigma}$ and $0 < t \le \Lambda \min\{(\mu \log n)^{1/(2-1/s)},\mu\}$ we have 
\begin{equation}\label{eq:thm:sg:sg}
\Pr(X \ge \mu+t) \le C \exp\Bigl(-c t^2/\mu \Bigr) .
\end{equation}
\end{theorem}
\begin{remark}\label{rem:sg:small}
It is well-known that in~\eqref{eq:thm:sg:small}--\eqref{eq:thm:sg:sg} we have $\mu=\E X \sim \Var X$ when $p=o(1)$.  
The proof shows that the constants~$C$ can be replaced by $1+o(1)$, and that~\eqref{eq:thm:sg:small}--\eqref{eq:thm:sg:sg} both carry over to $G_{n,m}$. 
Furthermore, \cite{SWUT} demonstrates that the sub-Gaussian type tail inequality~\eqref{eq:thm:sg:small} can already fail for balanced graphs~$H$. 
\end{remark}
%
%Recent work of {\v{S}}ileikis and the author~\cite{SW} yields lower bounds of form $\Pr(X \ge \mu+t) \ge C' \exp(-c' t^2/\mu)$ in some range, matching inequalities~\eqref{eq:thm:sg:small}--\eqref{eq:thm:sg:sg} up to constant factors in the exponent (recall Remark~\ref{rem:sg:small}). 
%In fact, in~\eqref{eq:thm:sg:sg} the $t$--range is best possible for the~$r$-armed star~$K_{1,r}$ by~\cite{star} (for $t \le \mu$ and small~$p$ inequality~\eqref{eq:thm:pre:extended:small:SG} is sharp). 
%For~$\eps=\Theta(1)$ this also entails that the $\mu$--range of~\eqref{eq:thm:sg:small} is best possible for~$K_{1,r}$. 

To put Theorem~\ref{thm:sg:small} into context, in the year 2000 Vu~\cite{Vu2000} showed that the sub-Gaussian inequality~\eqref{eq:thm:sg:small} holds for strictly balanced graphs as long as $\eps = O(1)$ and $\mu \le \log n$ (note that $\eps^2\mu \sim (\eps\mu)^2/\Var X$ by Remark~\ref{rem:sg:small}). 
Shortly afterwards, this result was reproved via a different method by Janson and Ruci{\'n}ski~\cite{DL}, who also raised the question whether the restriction $\mu=O(\log n)$ is necessary (see Section~6 in~\cite{DLP}). 
For the special case~$\eps=\Theta(1)$ the aforementioned results were yet again reproved by {\v{S}}ileikis~\cite{MS} in~2012.  
Our methods allow us (i)~to go beyond all these three approaches from 2000--2012, and (ii)~to answer the aforementioned question of Janson and Ruci{\'n}ski: inequality~\eqref{eq:thm:sg:small} still holds in the wider range $\mu = O((\log n)^{1+\xi})$. % for some $\xi>0$. 

Wolfovitz demonstrated the applicability of his sub-Gaussian concentration result~\cite{WG2012} via the complete graph $K_r$ %with $r \ge 3$ 
and the complete bipartite graph $K_{r,r}$, %with $r \ge 2$
showing that inequality~\eqref{eq:thm:sg:sg} holds for both strictly $2$-balanced graphs in certain ranges of the parameters~$p,t$.   
Theorem~\ref{thm:sg:small} generalizes these main applications from~\cite{WG2012} to all $2$-balanced graphs (for a slightly wider parameter range). 
For $n^{-1} \le p \le n^{-1/2-\sigma}$ inequality~\eqref{eq:thm:sg:sg} also 
slightly extends the $t$--range of two $K_3$-specific results of Kannan~\cite{Kannan} and Wolfovitz~\cite{WG2011}. 
\begin{proof}[Proof of Theorem~\ref{thm:sg:small}]
The proofs of~\eqref{eq:thm:sg:small}--\eqref{eq:thm:sg:sg} are very similar: each time we shall apply Theorem~\ref{thm:pre:extended:small} twice, using the two different setups of Examples~\ref{ex:EE}--\ref{ex:VE}. 
Hence our main task is to check assumption~\eqref{eq:thm:pre:basic:small:asmpt}.

For~\eqref{eq:thm:sg:small} we assume that~$H$ is strictly balanced, in which case $\delta = \delta_H \ge 1$ is folklore. 
By assumption there is a constant $\beta=\beta(H) >0$ such that for all subgraphs $J \subsetneq H$ with $v_J \ge 1$ we have 
\begin{equation}\label{eq:density:JH:0}
v_J \cdot \frac{e}{v} \ge e_J + \beta \quad \text{and} \quad e_J \cdot \frac{v}{e} \le v_J - \beta .
\end{equation}
Using the setup of Example~\ref{ex:EE}, by~\eqref{eq:muj:extended:EE} there is a constant $B_1>0$ such that the corresponding $\mu_j$ satisfy 
\begin{equation}\label{eq:density:muj:EE:2}
\max_{1 \le j < e-\delta+1}\mu_j \le B_1 \sum_{J \subseteq H: 1 \le e_J < e-\delta+1}n^{v-v_J}p^{e-e_J} .
\end{equation}
Similarly, using the setup of Example~\ref{ex:VE}, by~\eqref{eq:muj:extended:VE} there is a constant $B_2>0$ such that 
\begin{equation}\label{eq:density:muj:VE:2}
\max_{2 \le j < v}\mu_j \le B_2 \sum_{J \subseteq H: 2 \le v_J < v}n^{v-v_J}p^{e-e_J} .
\end{equation}
Recalling $s=\min\{v-1,e-\delta+1\}$, in our further estimates of~\eqref{eq:density:muj:EE:2}--\eqref{eq:density:muj:VE:2} we may assume $s>1$ (otherwise $H=K_2$ and~\eqref{eq:density:muj:EE:2}--\eqref{eq:density:muj:VE:2} are both equal to zero). 
Recalling $\mu=\Theta(n^vp^e)$, we now pick $S=S(\Lambda,H) \ge 1$ large enough such that the assumption $\mu^{(s-1)/s} \le \Lambda \log n$ implies $p \le S n^{-v/e+\beta/(2e)}$ for all $n \ge v$. 
Using $\delta=\delta_H \ge 1$ and the density condition~\eqref{eq:density:JH:0}, it follows that there are constants $B_3,B_4,B_5>0$ such that 
\begin{equation}\label{eq:density:muj:2}
\begin{split}
\text{\eqref{eq:density:muj:EE:2} $ + $ \eqref{eq:density:muj:VE:2}}& \le B_3 \sum_{J \subseteq H: v_J \ge 2, e_J < e}n^{v-v_J}p^{e-e_J} \le B_4 \sum_{J \subseteq H: v_J \ge 2, e_J < e}n^{e_Jv/e-v_J+\beta/2} \le B_5 n^{-\beta/2} .
\end{split}
\end{equation}
Armed with~\eqref{eq:density:muj:2}, we now apply Theorem~\ref{thm:pre:extended:small} with $K=1$, $A=B_5$ and~$\alpha=\beta/4$, using the setup of Example~\ref{ex:EE} (with $\ell=1$, $k=e$, $q=e-\delta+1$ and $N=n^2$) and Example~\ref{ex:VE} (with $\ell=2$, $k=q=v$ and $N=n$). 
So, applying~\eqref{eq:thm:pre:extended:small} twice, there is a constant $c_1 > 0$ such that for $t=\eps \mu$ we have 
\begin{equation}\label{eq:thm:pre:extended:small:SG}
\begin{split}
\Pr(X \ge \mu+t)  \le \bigl(1+2\max\{v_H,e_H\}n^{-1}\bigr) \exp\left(-c_1\min\Bigl\{t^2/\mu, \; t, \; t^{1/s} \log n\Bigr\} \right) .
\end{split}
\end{equation}
Since $t =\eps \mu \le \Lambda \mu$, we infer $t \ge t^2/(\Lambda \mu)$. 
Hence, after adjusting the constant~$c_1$, the $t$-term is irrelevant for the exponent of~\eqref{eq:thm:pre:extended:small:SG}. 
As $t^{2-1/s} \le (\Lambda \mu)^{1 + (s-1)/s} = O(\mu \log n)$ by assumption, this establishes~\eqref{eq:thm:sg:small}.

For~\eqref{eq:thm:sg:sg} we proceed similarly, assuming that~$H$ is $2$-balanced. 
In this case, for all subgraphs $J \subsetneq H$ with $2 \le v_J < v$, the assumption that~$H$ is $2$-balanced (and noting that~\eqref{eq:density:JH:2} is trivial when $v_J=2$) implies 
\begin{equation}\label{eq:density:JH:2}
\frac{e-e_J}{v-v_J} = \frac{(e-1)-(e_J-1)}{(v-2)-(v_J-2)} \ge \frac{e-1}{v-2} . 
\end{equation}
Analogous to~\eqref{eq:density:muj:2}, in Examples~\ref{ex:EE} and~\ref{ex:VE} (with $1 \le j < e-\delta+1$ and $2 \le j < v$) 
the assumption $p \le \Lambda n^{-(v-2)/(e-1)-\sigma}$ and the density result~\eqref{eq:density:JH:2} entail existence of constants $B_6,B_7>0$ such that  
\begin{equation}\label{eq:density:muj:3}
\mu_j  
\le B_6 \sum_{J \subseteq H: v_J \ge 2, e_J < e}n^{(v-v_J)-(e-e_J)(v-2)/(e-1)-(e-e_J)\sigma} \le B_7 n^{-\sigma} .
\end{equation}
Armed with~\eqref{eq:density:muj:3}, we now obtain~\eqref{eq:thm:pre:extended:small:SG} by applying Theorem~\ref{thm:pre:extended:small} twice (with $A=B_7$ and~$\alpha=\sigma/2$) analogous to the proof of~\eqref{eq:thm:sg:small}. 
Noting $t \le \Lambda \mu$ and $t^{2-1/s}=O(\mu \log n)$ then readily completes the proof of~\eqref{eq:thm:sg:sg}. 
\end{proof}
\noindent 
Parts of Theorem~\ref{thm:sg:small} can be proved in a simpler/more direct way, 
but in view of the previous work~\cite{Vu2000,DL,MS,WG2012,Kannan,WG2011} here the main point 
is to illustrate that~\eqref{eq:thm:sg:small}--\eqref{eq:thm:sg:sg} follow \emph{routinely} from our general bounds.

\subsubsection{Large deviations: upper tail problem}\label{sec:applications:RG:SG:large}
Next we consider the classical upper tail problem for subgraph counts,
which concerns~$\Pr(X \ge (1+\eps)\mu)$ for constant~$\eps>0$.
Here our general methods usually give much weaker estimates than modern specialized approaches such as~\cite{UTSG,KkTailDK,K3TailCh}, 
but it turns out that our methods can \emph{routinely} sharpen results based on classical inductive approaches (which might potentially be useful in other contexts). 
Indeed, for balanced graphs Kim and Vu used two different inductions (see Sections~6.3 and~6.6 in~\cite{Vu2002}), which together 
establish the following tail estimate: 
if~$\eps \le C$ and~$\eps^2 \max\{\mu^{1/(v-1)},\mu^{1/e}\} = \omega(\log n)$,~then 
\begin{equation}\label{eq:thm:sg:KV}
\Pr(X \ge (1+\eps)\mu) \le \exp\Bigl(-c \eps^2 \max\bigl\{\mu^{1/(v-1)}, \mu^{1/e}\bigr\} \Bigr) .
\end{equation}
This inequality was reproved by Janson and Ruci{\'n}ski~\cite{DL} via their alternative inductive method. 
Using Theorem~\ref{thm:pre:extendedp}, we shall go beyond both approaches for strictly balanced graphs:  
 (i)~we improve the exponential rate of decay by an extra logarithmic factor, and (ii)~we remove the restriction to `large' expectations~$\mu$. 
\begin{theorem}\label{thm:sg}
Let~$H$ be a strictly balanced graph with $v=v_H$~vertices and $e=e_H$~edges. 
Let $X=X_H$ and $\mu=\E X$.   
For any $\eps > 0$ there is $c=c(\eps,H)>0$ such that for all $n \ge v$ and $p \in [0,1]$ we have 
\begin{equation}\label{eq:thm:sg}
\Pr(X \ge (1+\eps)\mu) \le \exp\biggl(-c \min\Bigl\{\mu, \; \max\bigl\{\mu^{1/(v-1)}, \mu^{1/e}\bigr\} \log n \Bigr\} \biggr) .
\end{equation}
\end{theorem}
\begin{remark}\label{rem:sg}
Writing the exponent of~\eqref{eq:thm:sg} in the form $\exp(-c\Psi)$, the proof shows that $c=c'\min\{\eps^2,1\}$ with $c'=c'(H)>0$ suffices when~$\min\{\eps^2,1\} \Psi \ge 1$. Furthermore, inequality~\eqref{eq:thm:sg} also carries over to $G_{n,m}$.  
\end{remark}
\begin{remark}\label{rem:sg:bal}
For balanced graphs~$H$, the proof yields the following variant: 
for all $n \ge v$, $p \ge \xi n^{-v/e+\sigma}$ and $\eps>0$ we have $\Pr(X \ge (1+\eps)\mu) \le \exp(-c\mu^{1/(v-1)} \log n)$, where $c=c(\sigma,\xi,\eps,H)>0$. 
\end{remark}
\noindent 
For $r$-armed stars~$H=K_{1,r}$ inequality~\eqref{eq:thm:sg} yields an~$\exp\bigl(-\Omega(\min\{\mu,\mu^{1/r}\log n\})\bigr)$ exponential decay, 
which by~\cite{star} is best possible for~$p \le n^{-1/r}$ and~$\eps=\Theta(1)$.  
However, for general graphs~$H$ other approaches such as~\cite{UTSG,KkTailDK,K3TailCh}
 yield better estimates (as mentioned before), 
so we defer the proof of Theorem~\ref{thm:sg} to~Appendix~\ref{apx:proof}. 

\begin{ack}
We are grateful to the referees for helpful suggestions concerning the presentation. 
\end{ack}

\small
\bibliographystyle{plain}

\begin{thebibliography}{10}


\bibitem{BHKLS}
A.~Baltz, P.~Hegarty, J.~Knape, U.~Larsson, and T.~Schoen.
\newblock The structure of maximum subsets of {$\{1,\dots,n\}$} with no solutions to {$a+b=kc$}.
\newblock {\em Electron.\ J.\ Combin.} {\bf 12} (2005), Paper 19. 

\bibitem{BK2012}
M.~Bateman, and N.H.~Katz.
\newblock New bounds on cap sets. 
\newblock {\em J.~Amer.~Math.~Soc.} {\bf 25} (2012), 585--613. 

\bibitem{BK}
J.~van~den~Berg and H.~Kesten.
\newblock Inequalities with applications to percolation and reliability.
\newblock {\em J.\ Appl.\ Probab.} {\bf 22} (1985), 556--569.

\bibitem{BKkoutn}
J.~van~den~Berg and J.~Jonasson.
\newblock A {BK} inequality for randomly drawn subsets of fixed size.
\newblock {\em Probab.\ Theory Related Fields} {\bf 154} (2012), 835--844.

\bibitem{B2014}
T.~Bloom.
\newblock {A quantitative improvement for Roth's theorem on arithmetic progressions}.
\newblock {\em J.\ Lond.\ Math.\ Soc.} {\bf 93} (2016), 643--663. 

\bibitem{K3TailCh}
S.~Chatterjee.
\newblock The missing log in large deviations for triangle counts.
\newblock {\em Random Struct.\ Alg.} {\bf 40} (2012), 437--451.

\bibitem{KkTailDK}
B.~DeMarco and J.~Kahn.
\newblock Tight upper tail bounds for cliques.
\newblock {\em Random Struct.\ Alg.} {\bf 41} (2012), 469--487.

\bibitem{disjoint}
P.~Erd{\H{o}}s and P.~Tetali.
\newblock Representations of integers as the sum of {$k$} terms.
\newblock {\em Random Struct.\ Alg.} {\bf 1} (1990), 245--261.

\bibitem{GRR}
R.~Graham, V.~R{\"o}dl, and A.~Ruci{\'n}ski.
\newblock On {S}chur properties of random subsets of integers.
\newblock {\em J.\ Number Theory} {\bf 61} (1996), 388--408.

\bibitem{Gr}
B.~Green.
\newblock The Cameron--Erd{\H o}s conjecture.
\newblock {\em Bull.\ London Math. Soc.} {\bf 36} (2004), 769--778.

\bibitem{Janson}
S.~Janson.
\newblock Poisson approximation for large deviations.
\newblock {\em Random Struct.\ Alg.} {\bf 1} (1990), 221--229.  

\bibitem{UTSG}
S.~Janson, K.~Oleszkiewicz, and A.~Ruci{\'n}ski.
\newblock Upper tails for subgraph counts in random graphs.
\newblock {\em Israel J.\ Math.} {\bf 142} (2004), 61--92.

\bibitem{DLP}
S.~Janson and A.~Ruci{\'n}ski.
\newblock The deletion method for upper tail estimates.
\newblock Preprint (2000).
\texttt{\detokenize{http://www2.math.uu.se/~svante/papers/sj135_ppt.pdf}}

\bibitem{UT}
S.~Janson and A.~Ruci{\'n}ski.
\newblock The infamous upper tail.
\newblock {\em Random Struct.\ Alg.} {\bf 20} (2002), 317--342.

\bibitem{DL}
S.~Janson and A.~Ruci{\'n}ski.
\newblock The deletion method for upper tail estimates.
\newblock {\em Combinatorica} {\bf 24} (2004), 615--640.

\bibitem{UTAP}
S.~Janson and A.~Ruci{\'n}ski.
\newblock Upper tails for counting objects in randomly induced subhypergraphs and rooted random graphs.
\newblock {\em Ark.\ Mat.} {\bf 49} (2011), 79--96.

\bibitem{JSuen}
S.~Janson.
\newblock New versions of Suen's correlation inequality. 
\newblock {\em Random Struct.\ Alg.} {\bf 13} (1998), 467--483.

\bibitem{JW}
S.~Janson and L.~Warnke.
\newblock {The lower tail: Poisson approximation revisited}.
\newblock {\em Random Struct.\ Alg.} {\bf 48} (2016), 219--246.

\bibitem{Kannan}
R.~{Kannan}.
\newblock {Two new Probability inequalities and Concentration Results}.
\newblock Preprint (2010). \texttt{arXiv:0809.2477v4}.

\bibitem{KimVu2000}
J.H.~Kim and V.H.~Vu.
\newblock Concentration of multivariate polynomials and its applications.
\newblock {\em Combinatorica} {\bf 20} (2000), 417--434. 

\bibitem{BKR}
D.~Reimer.
\newblock Proof of the van den {B}erg-{K}esten conjecture.
\newblock {\em Combin.\ Probab.\ Comput.} {\bf 9} (2000), 27--32.

\bibitem{RW}
O.~Riordan and L.~Warnke.
\newblock The {J}anson inequalities for general up-sets.
\newblock {\em Random Struct.\ Alg.} {\bf 46} (2015), 391--395. 

\bibitem{RR1994}
V.~R{\"o}dl and A.~Ruci{\'n}ski.
\newblock Random graphs with monochromatic triangles in every edge coloring.
\newblock {\em Random Struct.\ Alg.} {\bf 5} (1994), 253--270.

\bibitem{SGN}
A.~Ruci{\'n}ski.
\newblock When are small subgraphs of a random graph normally distributed?
\newblock {\em Probab.\ Theory Related Fields} {\bf 78} (1988), 1--10.

\bibitem{Schacht2009}
M.~Schacht.
\newblock Extremal results for random discrete structures.
\newblock {\em Ann.\ of Math.} {\bf 184} (2016), 333--365. 

\bibitem{MS}
M.~{\v{S}}ileikis.
On the upper tail of counts of strictly balanced subgraphs.
\newblock {\em Electron.\ J.\ Combin.} {\bf 19} (2012), Paper 4. 

%\bibitem{SW}
%M.~{\v{S}}ileikis and L.~Warnke. 
%In preparation (2016)

\bibitem{SWUT}
M.~{\v{S}}ileikis and L.~Warnke. 
\newblock A counterexample to the DeMarco-Kahn Upper Tail Conjecture. 
\newblock {\em Random Struct.\ Alg.}, to appear. \texttt{arXiv:1809.09595}.
%\newblock Preprint (2018). \texttt{arXiv:1809.09595}.

\bibitem{star}
M.~{\v{S}}ileikis and L.~Warnke. 
\newblock Upper tail bounds for Stars.
\newblock Preprint (2019). \texttt{arXiv:1901.10637}.
%\newblock In preparation. % (2019). 

\bibitem{Spencer1990}
J.~Spencer.
\newblock Counting extensions.
\newblock {\em J.\ Combin.\ Theory Ser.~A} {\bf 55} (1990), 247--255.

\bibitem{SSW}
R.~Sp{\"o}hel, A.~Steger and L.~Warnke.
\newblock General deletion lemmas via the Harris inequality. 
\newblock {\em J.\ Combin.} {\bf 4} (2013), 251--271.

\bibitem{Vu2000}
V.H.~Vu.
\newblock On the concentration of multivariate polynomials with small expectation.
\newblock {\em Random Struct.\ Alg.} {\bf 16} (2000), 344--363.

\bibitem{Vu2001}
V.H.~Vu.
\newblock A large deviation result on the number of small subgraphs of a random graph.
\newblock {\em Combin.\ Probab.\ Comput.} {\bf 10} (2001), 79--94.

\bibitem{Vu2002}
V.H.~Vu.
\newblock Concentration of non-{L}ipschitz functions and applications.
\newblock {\em Random Struct.\ Alg.} {\bf 20} (2002), 262--316. 

\bibitem{K4free}
L.~Warnke.
\newblock When does the {$K_4$}-free process stop?
\newblock {\em Random Struct.\ Alg.} {\bf 44} (2014), 355--397.

\bibitem{TBD}
L.~Warnke.
\newblock On the method of typical bounded differences. 
\newblock {\em Combin.\ Probab.\ Comput.} {\bf 25} (2016), 269--299.

\bibitem{AP}
L.~Warnke.
\newblock {Upper tails for arithmetic progressions in random subsets}.
\newblock {\em Israel J.\ Math.} {\bf 221} (2017), 317--365.

\bibitem{WG2011}
G.~Wolfovitz.
\newblock Sub-{G}aussian tails for the number of triangles in {$G(n,p)$}.
\newblock {\em Combin.\ Probab.\ Comput.}, 20(1):155--160, 2011.

\bibitem{WG2012}
G.~Wolfovitz.
\newblock A concentration result with application to subgraph count.
\newblock {\em Random Struct.\ Alg.} {\bf 40} (2012), 254--267. 

\end{thebibliography}

\normalsize

\begin{appendix}

\section{Proofs omitted from Section~\ref{sec:applications:RG:SG:large}}\label{apx:proof} % Appendix: 
In this appendix we give the proof of Theorem~\ref{thm:sg}, which proceeds similar to Theorem~\ref{thm:randinduced} and~\ref{thm:sg:small}. 
Namely, we prove~\eqref{eq:thm:sg} by two applications of Theorem~\ref{thm:pre:extendedp} and Remark~\ref{rem:pre:extendedp} (using the setups of Examples~\ref{ex:EE}--\ref{ex:VE}).
\begin{proof}[Proof of Theorem~\ref{thm:sg}]
We first use the setup of Example~\ref{ex:EE} % exploiting the folklore fact that balanced graphs contain no isolated vertices. 
with $\ell=1$, $q=k=e$ and $N=n^2$. 
Using the bound~\eqref{eq:muj:extended:EE} for~$\mu_j$, the expectation~$\mu=\Theta(n^vp^e)$ and the density result~\eqref{eq:density:JH:0}, for $1 \le j < e=e_H$ we infer 
\begin{equation}\label{eq:density:muj:0}
\frac{\mu_j}{\mu^{(q-j)/(q-\ell+1)}} \le \frac{B \sum_{J \subseteq H: e_J=j}n^{v-v_J}p^{e-j}}{(\mu^{1/e})^{e-j}} \le B_1 \sum_{J \subseteq H: e_J=j}n^{e_Jv/e-v_J} \le B_2 n^{-\beta} . %= B_2 N^{-\beta/2}.
\end{equation}
Applying Theorem~\ref{thm:pre:extendedp} and Remark~\ref{rem:pre:extendedp} with $A=B_2$ and $\alpha=\beta/2$, there thus is $c_1>0$ such that  
\begin{equation}\label{eq:thm:sg:lower:EE}
\Pr(X \ge (1+\eps)\mu) \le (1+3e_H n^{-2}) \exp\left(-c_1 \min\bigl\{\eps^2,1\bigr\} \min\{\mu, \; \mu^{1/e}\log n\}\right) .
\end{equation}

Next we use the setup of Example~\ref{ex:VE} with $\ell=2$, $k=q=v$ and $N=n$. 
We distinguish several cases. 
If $p \le n^{-v/e}$, then using the bound~\eqref{eq:muj:extended:EE} for $\mu_j$ and the density result~\eqref{eq:density:JH:0}, we infer for $2 \le j < v=v_H$ that 
\begin{equation}\label{eq:density:muj}
\mu_j \le  B \sum_{J \subseteq H: v_J=j}n^{v-v_J}p^{e-e_J} \le B \sum_{J \subseteq H: 2 \le v_J < v_H}n^{e_Jv/e-v_J} \le B_3 n^{-\beta} . %= B_3N^{-\beta}. 
\end{equation}
Otherwise $p \ge n^{-v/e}$, so $n^vp^e \ge 1$. 
Note that for $j < v$ we have $(v-j)/(v-1) \ge (v-j)/v + 1/v^2$. 
Recalling~$\ell=2$ and~$q=v$, using~\eqref{eq:muj:extended:VE}, $\mu=\Theta(n^vp^e)$ and \eqref{eq:density:JH:0} we infer for $2 \le j < v=v_H$ that 
\begin{equation}\label{eq:density:VE1}
\frac{\mu_j}{\mu^{(q-j)/(q-\ell+1)}} 
%= \frac{\mu_j}{\mu^{(v-j)/(v-1)}}
 \le \frac{\mu_j}{B_4(n^vp^e)^{(v-j)/v + 1/v^2}} \le \frac{B_5 \sum_{J \subseteq H: v_J=j}p^{v_J e/v-e_J}}{(n^vp^e)^{1/v^2}} 
\le \frac{B_6 p^{\beta}}{(n^vp^e)^{1/v^2}} .
\end{equation}
Distinguishing $n^{-v/e} \le p \le n^{-v/(2e)}$ and $n^{-v/(2e)} \le p \le 1$, we see that  
\begin{equation}\label{eq:density:VE2}
\frac{\mu_j}{\mu^{(q-j)/(q-\ell+1)}} \le B_6 \max\{n^{-\beta v/(2e)},n^{-1/(2v)}\} . 
\end{equation}
Applying Theorem~\ref{thm:pre:extendedp} and Remark~\ref{rem:pre:extendedp} with $A=\max\{B_3,B_6\}$ and $\alpha=\min\{\beta,\beta v/(2e),1/(2v)\}$, we deduce 
\begin{equation}\label{eq:thm:sg:lower:VE}
\Pr(X \ge (1+\eps)\mu) \le (1+3v_H n^{-1}) \exp\left(-c_2 \min\bigl\{\eps^2,1\bigr\} \min\{\mu, \; \mu^{1/(v-1)}\log n\}\right) .
\end{equation}

Finally, we combine the two upper bounds~\eqref{eq:thm:sg:lower:EE} and~\eqref{eq:thm:sg:lower:VE}, and then remove (for cosmetic reasons) the multiplicative prefactor $1+O(n^{-1})$ analogous to the proof of Theorem~\ref{thm:randinduced}, which establishes~\eqref{eq:thm:sg}.  
\end{proof}
\noindent 
For Remark~\ref{rem:sg:bal} the point is that for balanced graphs~$H$ the density condition~\eqref{eq:density:JH:0} only holds with~$\beta=0$, so in~\eqref{eq:density:VE1} we need $p \ge \xi n^{-v/e+\sigma}$ to establish~\eqref{eq:density:VE2} with $\le O(n^{-e\sigma/v^2})$, say.

\end{appendix}

\end{document}